\documentclass[letterpaper, 11pt]{article}

\usepackage[headheight=110pt,margin=1in]{geometry}

\usepackage{tikz-cd}

\usepackage{macros}

\usepackage{hyperref}

\usepackage{amsthm}

\usepackage[capitalise,nameinlink]{cleveref}

\newtheorem{theorem}{Theorem}[section]
\newtheorem{lemma}[theorem]{Lemma}
\newtheorem{corollary}[theorem]{Corollary}
\newtheorem{proposition}[theorem]{Proposition}

\theoremstyle{definition}
\newtheorem{definition}[theorem]{Definition}

\theoremstyle{remark}
\newtheorem{remark}[theorem]{Remark}

\usepackage{graphicx}
\usepackage{booktabs}
\usepackage{tabularx}
\usepackage{makecell}
\usepackage{pifont}
\usepackage[table]{xcolor}

\newcommand{\cmark}{\ding{51}}
\newcommand{\xmark}{\ding{55}}

\newcolumntype{C}[1]{>{\centering\arraybackslash}p{#1}}
\newcolumntype{L}[1]{>{\raggedright\arraybackslash}p{#1}}

\colorlet{thispaper}{blue!7}

\newcommand{\keywords}[1]{%
    \par\medskip
    \noindent\textbf{Keywords:} #1
}

\newcommand{\msc}[1]{%
    \par\medskip
    \noindent\textbf{2020 Mathematics Subject Classification:} #1
}

\hypersetup{bookmarksopen=true}

\usepackage{crossreftools}
\crefformat{equation}{\textup{#2(#1)#3}}
\crefrangeformat{equation}{\textup{#3(#1)#4--#5(#2)#6}}
\crefmultiformat{equation}{\textup{#2(#1)#3}}{ and \textup{#2(#1)#3}}
{, \textup{#2(#1)#3}}{, and \textup{#2(#1)#3}}
\crefrangemultiformat{equation}{\textup{#3(#1)#4--#5(#2)#6}}%
{ and \textup{#3(#1)#4--#5(#2)#6}}{, \textup{#3(#1)#4--#5(#2)#6}}{, and \textup{#3(#1)#4--#5(#2)#6}}

\Crefformat{equation}{#2Equation~\textup{(#1)}#3}
\Crefrangeformat{equation}{Equations~\textup{#3(#1)#4--#5(#2)#6}}
\Crefmultiformat{equation}{Equations~\textup{#2(#1)#3}}{ and \textup{#2(#1)#3}}
{, \textup{#2(#1)#3}}{, and \textup{#2(#1)#3}}
\Crefrangemultiformat{equation}{Equations~\textup{#3(#1)#4--#5(#2)#6}}%
{ and \textup{#3(#1)#4--#5(#2)#6}}{, \textup{#3(#1)#4--#5(#2)#6}}{, and \textup{#3(#1)#4--#5(#2)#6}}

\crefname{theorem}{Theorem}{Theorems}
\crefname{lemma}{Lemma}{Lemmas}
\crefname{corollary}{Corollary}{Corollaries}
\crefname{proposition}{Proposition}{Propositions}  %
\crefname{definition}{Definition}{Definitions}     %
\crefname{assumption}{Assumption}{Assumptions}
\crefname{example}{Example}{Examples}              %
\crefname{xca}{Exercise}{Exercises}                %
\crefname{remark}{Remark}{Remarks}
\crefname{figure}{Figure}{Figures}

\usepackage[inline]{enumitem}

\numberwithin{equation}{section}

\title{Representation Costs in Data Science: \\ Foundations and the Quasi-Banach Spaces of Deep Neural Networks}

\author{
Greg Ongie\\
{\normalsize Marquette University}\\
{\small\texttt{gregory.ongie@marquette.edu}}
\and
Rahul Parhi\\
{\normalsize University of California, San Diego}\\
{\small\texttt{rahul@ucsd.edu}}
}

\date{}

\begin{document}
\maketitle

\begin{abstract}

We develop a general framework for analyzing representation costs induced by parameter-space regularizers in data-fitting methods. For an arbitrary parametric method, we define its representation cost and native function space, prove existence, and identify conditions under which parameter-space and function-space problems have equal infimal values and minimizers transfer between them. This framework yields representer theorems and recovers classical formulations---including kernel methods and RKHSs, wavelets and Besov spaces, and shallow neural networks and variation spaces---as special cases. Our main new results concern depth-$L$ feedforward ReLU networks with weight-decay regularization. For these networks, we prove that the representation cost is a power of a quasi-seminorm and that, under suitable hypotheses, the native space is a quasi-Banach space with nonconvex unit ball when $L > 2$. These results identify a novel depth-dependent quasi-Banach geometry induced by weight decay.

    \keywords{representation costs; lower semicontinuity; quasi-Banach spaces; deep neural networks; inductive bias}

    \msc{Primary 46A16, 68T07; Secondary 46E15, 46B06, 41A46, 49J45.}
\end{abstract}

\tableofcontents

\newpage

\section{Introduction}
Understanding the \emph{inductive bias} of a data-fitting method is one of the principal outstanding problems in data science research. The goal of a data-fitting method is, given a collection of data $\curly{(x_i, y_i)}_{i=1}^N \subset \X \times \Y$, to construct an \emph{estimator} $f: \X \to \Y$ such that $f(x_i)$ is ``close'' to $y_i$, and, for prediction problems in particular, given a new sample $(x_\mathrm{new}, y_\mathrm{new}) \in \X \times \Y$, it holds that $f(x_\mathrm{new})$ is also ``close'' to $y_\mathrm{new}$. There are many different methods to solve this problem such as kernel methods, sparsity-promoting methods, and neural networks. Each of these methods has a different inductive bias. For example, kernel methods result in smooth estimators~\cite{WahbaSmoothingSplines1,WahbaSmoothingSplines2,WahbaSmoothingSplines3}, while sparsity-promoting methods, such as those implemented with wavelets, result in estimators that can adapt to unknown levels of smoothness~\cite{donoho1994ideal,donoho1995adapting,DonohoWaveletShrinkage}. Thus, a systematic approach to characterizing the inductive bias of a data-fitting method, i.e., properties of the \emph{end-to-end function} implemented by an estimator, is paramount for understanding, comparing, and contrasting different methods.

A fruitful area of research in this direction is to consider the \emph{representation cost} as opposed to the \emph{parameter cost} of a data-fitting method. The goal of this perspective is to be agnostic to the parameterization that realizes a given estimator and instead to consider the end-to-end function implemented by a given estimator. This perspective is crucial in modern deep learning. Indeed, consider the simple parametric model $f_{v, w}: \R^d \ni x \mapsto v(w^\T x)_+ \in \R$, with parameters $(v, w) \in \R \times \R^d$, that consists of a single rectified linear unit (ReLU) neuron. Given any $(v_0, w_0)$, we always have the equality $f_{v_0, w_0} = f_{v_0 / \gamma, \gamma w_0}$ for any $\gamma > 0$ due to the positive homogeneity of the ReLU. Thus, we see that there are \emph{uncountably many} parameterizations that implement the same function. Furthermore, the parameterizations can be of \emph{arbitrary size} in parameter space. Thus, only considering the parameters of a model does not correctly capture the complexity of the end-to-end function. This simple observation motivates the study of representation costs and the function-space perspective.

The notion of a ``representation cost'' can be traced back to classical work on regularization theory in Hilbert spaces dating back to the 1970s, particularly in reproducing kernel Hilbert spaces~(RKHSs)~\cite{theory-reproducing-kernels}. In that setting, the RKHS norm serves as the regularizer of a data-fitting problem and serves as the representation cost for smoothing splines~\cite{WahbaSmoothingSplines1,WahbaSmoothingSplines2,WahbaSmoothingSplines3,WahbaSplineModels} and kernel methods~\cite{ScholkopfKernels}. In the 1990s representation costs for shallow neural networks based on $L^1$-type norms of the Fourier transform were introduced~\cite{BarronApprox,BarronApproxEst}. In parallel, work on wavelets and Besov spaces introduced analogous interpretations in the context of nonlinear approximation and nonparametric function estimation where the representation cost corresponds to Besov (quasi-)norms based on sequence-space norms of wavelet coefficients~\cite{devore1998nonlinear,DeVoreLorentz1993Constructive,DeVoreBesovDomain,DonohoUnconditional,MeyerWaveletsOperators,Triebel1983Theory}. The development of AlexNet~\cite{krizhevsky2012imagenet} in 2012 marked a turning point and the deep learning era began~\cite{LeCunDeepLearning}. This led to norm-based complexity measures for neural networks, such as path norms~\cite{neyshabur2017exploring,neyshabur2015path,neyshabur2015norm} and a renewed interest in variation norms~\cite{BachConvexNN,kurkova2001bounds,kurkova2002comparison, mhaskar2004tractability}. These modern representation costs have further emphasized the function-space view induced by parameter-space regularization, especially due to the recent works~\cite{OngieFunctionSpace,parhi2021banach,parhi2025function,SavareseInfWidth,zeno2023minimum}.

Despite these commonalities, this line of research has largely evolved in a case-by-case fashion. Kernel methods, wavelets, and shallow neural networks are typically studied with distinct mathematical tools, leaving their unifying principles obscured. In particular, what is missing is a \emph{unified framework} that defines representation costs in generality for arbitrary parametric methods, and that systematically identifies the corresponding induced function spaces. Such a framework would place classical results into a single conceptual framework and make it possible to analyze new methods (e.g., deep neural networks) using the same mathematical tools. 

In this paper, we propose and develop such a framework. We define representation costs for arbitrary parametric methods directly through their parameter-space regularizers and show that these costs induce native function spaces that capture inherent properties of the method. Within this framework, we prove existence results and identify conditions under which the parameter-space and function-space problems have the same infimal value and minimizers transfer between them. In the concrete examples, these results combine with model-specific arguments to yield representer theorems. This connection provides a rigorous bridge between parametric methods and their nonparametric descriptions under sufficient overparameterization.

Beyond recovering classical cases---such as kernel methods and RKHSs, wavelets and Besov spaces, and shallow neural networks and variation spaces---our framework yields new insights into deep learning. For depth-$L$ feedforward ReLU neural networks with weight-decay regularization, we prove that the representation cost induces a quasi-seminorm and that the corresponding native space is a quasi-Banach space (\cref{thm:quasi-Banach-geometry}), a structural property that mirrors the familiar connection between deep linear neural networks and Schatten-$(2/L)$ regularization~\cite{dai2021representation,ShangSchatten,MertParallel}. More precisely, when $L>2$, the input and output dimensions are at least two, and the input domain has nonempty interior, the unit representation cost ball is \emph{nonconvex} (\cref{prop:deep-unit-ball-nonconvex}). Hence, the corresponding quasi-norm is not a \textit{bona fide} norm as it violates the triangle inequality. This provides mathematical justification for recent empirical observations from~\cite{razin2020implicit} that the inductive bias of deep neural networks may not be explainable by norms.

\subsection{Overview and Goals}
The setup of this paper will be that of \emph{explicit regularization} as we adopt a perspective that is agnostic to the optimization algorithm used to fit the data. Given a data set $\curly{(x_i, y_i)}_{i=1}^N \subset \X \times \Y$ and a \emph{parametric model} $f_\theta$ parameterized by a parameter $\theta$ that lies in the \emph{parameter space} $\Theta$, we focus on the data-fitting problem formulated as the optimization problem
\begin{equation} \label{eq:param_space_opt_old}
    \min_{\theta \in \Theta} \sum_{i=1}^N \Loss(y_i, f_\theta(x_i)) + \lambda C(\theta),
\end{equation}
where $\Loss: \Y \times \Y \to [0, +\infty]$ is a \emph{loss function} and $C: \Theta \to [0, +\infty)$ is the \emph{parameter cost}, which acts as the explicit regularizer for the optimization problem. More generally, we can consider data-fitting problems of the form
\begin{equation} \label{eq:param_space_opt}
    \min_{\theta \in \Theta} G\paren[\big]{f_\theta(x_1), \ldots, f_\theta(x_N)} + \lambda C(\theta),
\end{equation}
where $G: \Y^N \to [0, +\infty]$ plays the role of the loss function.

It is difficult to infer directly from \cref{eq:param_space_opt} what impact a given parametric model $f_\theta$ and parameter cost $C(\theta)$ will have on function space. For example, $\ell^2$-regularization (i.e., weight decay)  promotes parameter vectors $\theta$ with low norm, but how does this influence the resulting estimator $f_\theta$ in terms of function-space properties? To gain an understanding, we can instead consider the optimization problem \emph{posed over functions}
\begin{equation}\label{eq:function_space_opt}
\min_{f \in \F_\Theta} G\paren[\big]{f(x_1), \ldots, f(x_N)} + \lambda \Rcirc(f),
\end{equation}
where $\F_\Theta = \{ f_\theta \st \theta \in \Theta\}$ denotes the \emph{parametric model class} and $\Rcirc: \F_\Theta \to [0, +\infty)$ is the \emph{parametric representation cost} of a function, defined as
\begin{equation}
\Rcirc(f) = \inf_{\substack{\theta \in \Theta \\ f=f_\theta}} C(\theta), \quad f \in \F_\Theta.
\end{equation}

Although \cref{eq:function_space_opt,eq:param_space_opt} clearly have the same infimal value and minimizers of \cref{eq:param_space_opt} induce minimizers of \cref{eq:function_space_opt}, by considering \cref{eq:function_space_opt}, $\Rcirc$ can be thought of as the \emph{induced function-space regularizer} from the parametric model $f_\theta$ and parameter cost $\theta \mapsto C(\theta)$. However, from a functional-analytic or topological perspective, there are challenges in characterizing solutions to \cref{eq:function_space_opt} as $\F_\Theta$ is typically not closed in any reasonable topology~\cite{petersen2021topological}. In other words, for most $\F_\Theta$, there exist functions that can be approximated arbitrarily well (in many choices of topology) by members of $\F_\Theta$ that do not belong to $\F_\Theta$.

Thus, in order to establish sharp results on the function-space inductive bias for a parametric method, we must minimally ``extend'' $\Rcirc$ and ``close'' $\F_\Theta$ so that we can consider \textit{bona fide} function-space optimization problems of the form
\begin{equation} \label{eq:function_space_opt_closed}
    \min_{f \in \F} G\paren[\big]{f(x_1), \ldots, f(x_N)} + \lambda R(f),
\end{equation}
where the \emph{native space} $\F$ is the ``closure'' of $\F_\Theta$ and $R: \F \to [0, +\infty)$ is the \emph{extended} (or \emph{nonparametric}) representation cost, which captures the induced function-space regularizer. This paper rigorously characterizes the extension and closure procedures and gives conditions under which \cref{eq:param_space_opt,eq:function_space_opt,eq:function_space_opt_closed} have the same infimal value and their minimizers transfer between one another. The end result is a unified framework to study the function-space inductive bias of parametric data-fitting methods.

\begin{figure}[!htb]
\centering
\begingroup

\definecolor{rcInk}{HTML}{25313C}
\definecolor{rcGray}{HTML}{697586}
\definecolor{rcBorder}{HTML}{92A5B5}
\definecolor{rcNavy}{HTML}{31546D}
\definecolor{rcBlue}{HTML}{2F76A8}
\definecolor{rcTeal}{HTML}{2E8F84}
\definecolor{rcGold}{HTML}{D29223}
\definecolor{rcCoral}{HTML}{C96052}
\definecolor{rcPurple}{HTML}{7568A5}
\definecolor{rcBall}{HTML}{536F8F}

\newcommand{\DrawPBall}[4]{%
  \begin{scope}[shift={#1},scale=#4]
    \foreach \sx in {-1,1}{%
      \foreach \sy in {-1,1}{%
        \path[fill=#3!10]
          (0,0) --
          plot[domain=0:90,samples=90,smooth,variable=\t]
          ({\sx*pow(cos(\t),2/(#2))},{\sy*pow(sin(\t),2/(#2))})
          -- cycle;
      }%
    }%
    \draw[rcBorder!55,line width=0.28pt] (-1.12,0)--(1.12,0);
    \draw[rcBorder!55,line width=0.28pt] (0,-1.12)--(0,1.12);
    \foreach \sx in {-1,1}{%
      \foreach \sy in {-1,1}{%
        \draw[#3,line width=0.90pt]
          plot[domain=0:90,samples=90,smooth,variable=\t]
          ({\sx*pow(cos(\t),2/(#2))},{\sy*pow(sin(\t),2/(#2))});
      }%
    }%
  \end{scope}%
}

\begin{tikzpicture}[
  x=1cm,
  y=1cm,
  >=Latex,
  every node/.style={font=\normalsize,text=rcInk},
  panel/.style={
    draw=rcBorder,
    rounded corners=3pt,
    line width=0.65pt,
    fill=white
  },
  paneltitle/.style={
    font=\normalsize\bfseries,
    anchor=west,
    text=rcInk
  },
  panelmark/.style={
    font=\normalsize\bfseries,
    anchor=west
  },
  sectiontag/.style={
    font=\footnotesize\itshape,
    anchor=east,
    text=rcGray
  },
  stage/.style={
    draw,
    rounded corners=3pt,
    line width=0.75pt,
    align=center,
    inner xsep=6pt,
    inner ysep=6pt
  },
  objective/.style={
    draw=rcPurple!60,
    fill=rcPurple!4,
    rounded corners=3pt,
    line width=0.60pt
  },
  geometrybox/.style={
    draw=rcPurple!50,
    fill=rcPurple!3,
    rounded corners=3pt,
    line width=0.55pt,
    align=center,
    inner xsep=7pt,
    inner ysep=5pt
  },
  casecard/.style={
    draw,
    rounded corners=3pt,
    line width=0.60pt,
    fill=white
  },
  deepbox/.style={
    draw=rcCoral!72,
    fill=rcCoral!4,
    rounded corners=3pt,
    line width=0.65pt,
    align=center,
    inner xsep=7pt,
    inner ysep=5pt
  },
  smallresult/.style={
    rounded corners=3pt,
    line width=0.60pt,
    align=center,
    inner xsep=6pt,
    inner ysep=5pt,
    font=\footnotesize
  },
  arrow/.style={
    ->,
    line width=0.90pt
  },
  arrowlabel/.style={
    font=\footnotesize,
    align=center,
    text=rcGray
  },
  note/.style={
    font=\footnotesize,
    text=rcGray,
    align=center
  }
]

\path[use as bounding box] (0,0) rectangle (16.0,16.80);

\draw[panel] (0.10,8.50) rectangle (15.90,16.70);
\node[panelmark,text=rcNavy] at (0.42,16.30) {(a)};
\node[paneltitle] at (1.22,16.30) {From parameter costs to native spaces};
\node[sectiontag] at (15.58,16.30) {\crefrange{sec:background}{sec:Lip}};
\draw[rcNavy!55,line width=0.55pt] (0.42,15.90)--(15.58,15.90);

\node[
  stage,
  draw=rcBlue!82,
  fill=rcBlue!5,
  text width=3.05cm,
  minimum height=2.35cm
] (parameters) at (2.18,14.50)
  {{\bfseries Parameters}\\
   $\theta\in\Theta$\\
   $\theta\mapsto f_\theta$\\
   cost $C(\theta)$};

\node[
  stage,
  draw=rcGold!88,
  fill=rcGold!6,
  text width=4.50cm,
  minimum height=2.35cm
] (repcost) at (8.00,14.50)
  {{\bfseries Parametric model}\\
   $\mathcal F_\Theta=\{f_\theta\st\theta\in\Theta\}$\\
   $\displaystyle R_\Theta(f)=\inf_{\substack{\theta\in\Theta\\f_\theta=f}} C(\theta)$};

\node[
  stage,
  draw=rcTeal!86,
  fill=rcTeal!5,
  text width=3.05cm,
  minimum height=2.35cm
] (native) at (13.82,14.50)
  {{\bfseries Native space}\\
   l.s.c.\ reg.\ $R$\\
   of $R_\Theta$\\
   $\mathcal F=\operatorname{dom}R$};

\draw[arrow,draw=rcBlue!82] (parameters.east)--(repcost.west);
\node[arrowlabel,text width=1.40cm] at (4.72,13.88)
  {fiberwise\\infimum};

\draw[arrow,draw=rcTeal!84] (repcost.east)--(native.west);
\node[arrowlabel,text width=1.40cm] at (11.28,13.88)
  {l.s.c.\\reg.};

\draw[objective] (0.55,10.72) rectangle (15.45,12.62);
\node[font=\normalsize\bfseries] at (8.00,12.28)
  {Equivalent data-fitting problems};
\node[font=\footnotesize,align=center,text width=14.30cm] at (8.00,11.62)
  {$\displaystyle
    \inf_{\theta\in\Theta}
      G\big(f_\theta(x_1),\ldots,f_\theta(x_N)\big)+\lambda C(\theta)
    =
    \inf_{f\in\mathcal F}
      G\big(f(x_1),\ldots,f(x_N)\big)+\lambda R(f)$};
\node[note,font=\footnotesize\bfseries] at (8.00,11.06)
  {existence; equal infimal values; minimizer transfer};

\node[geometrybox,text width=6.55cm,minimum height=0.94cm] at (4.00,9.62)
  {homogeneity $\Longrightarrow$ gauge geometry};
\node[geometrybox,text width=6.55cm,minimum height=0.94cm] at (12.00,9.62)
  {subadditivity and homogeneity\\
   $\Longrightarrow$ (quasi-)Banach geometry};

\draw[panel] (0.10,0.10) rectangle (7.82,8.25);
\node[panelmark,text=rcTeal] at (0.42,7.86) {(b)};
\node[paneltitle] at (1.18,7.86) {Classical instances};
\node[sectiontag] at (7.54,7.86) {\crefrange{sec:RKHS}{sec:shallow}};
\draw[rcTeal!55,line width=0.55pt] (0.42,7.48)--(7.54,7.48);

\node[
  casecard,
  draw=rcBlue!72,
  fill=rcBlue!4,
  minimum width=3.48cm,
  minimum height=2.16cm
] at (2.05,6.16) {};
\begin{scope}[shift={(0.83,6.16)},scale=0.33]
  \draw[rcBlue!42,line width=0.35pt] (-1.12,-0.70)--(1.12,-0.70);
  \draw[rcBlue,line width=1.12pt,smooth]
    plot[domain=-1.08:1.08,samples=45] (\x,{1.05*exp(-2.4*\x*\x)-0.70});
  \fill[rcBlue] (-0.50,-0.44) circle (1.5pt);
  \fill[rcBlue] (0.48,-0.43) circle (1.5pt);
\end{scope}
\node[anchor=west,align=center,text width=2.22cm] at (1.32,6.16)
  {Kernel methods\\$\longmapsto$\\\textcolor{rcBlue}{\bfseries RKHSs}};

\node[
  casecard,
  draw=rcPurple!70,
  fill=rcPurple!4,
  minimum width=3.48cm,
  minimum height=2.16cm
] at (5.87,6.16) {};
\begin{scope}[shift={(4.65,6.16)},scale=0.33]
  \draw[rcPurple!42,line width=0.35pt] (-1.12,0)--(1.12,0);
  \draw[rcPurple,line width=1.08pt]
    (-1.05,0) .. controls (-0.78,0.02) and (-0.63,0.78) .. (-0.35,0.78)
    .. controls (-0.05,0.78) and (0.02,-0.82) .. (0.34,-0.82)
    .. controls (0.66,-0.82) and (0.77,-0.02) .. (1.05,0);
\end{scope}
\node[anchor=west,align=center,text width=2.22cm] at (5.14,6.16)
  {Wavelets\\$\longmapsto$\\\textcolor{rcPurple}{\bfseries Besov spaces}};

\node[
  casecard,
  draw=rcGold!82,
  fill=rcGold!5,
  minimum width=3.48cm,
  minimum height=2.16cm
] at (2.05,3.66) {};
\begin{scope}[shift={(0.83,3.66)},scale=0.33]
  \draw[rcGold!42,line width=0.35pt] (-1.12,-0.68)--(1.12,-0.68);
  \foreach \x/\h in {-0.72/0.45,-0.05/1.05,0.70/0.68}{
    \draw[rcGold,line width=1.08pt] (\x,-0.68)--(\x,\h-0.68);
    \fill[rcGold] (\x,\h-0.68) circle (1.7pt);
  }
\end{scope}
\node[anchor=west,align=center,text width=2.22cm] at (1.32,3.66)
  {Sparse kernels\\$\longmapsto$\\\textcolor{rcGold}{\bfseries I-RKBSs}};

\node[
  casecard,
  draw=rcCoral!75,
  fill=rcCoral!4,
  minimum width=3.48cm,
  minimum height=2.16cm
] at (5.87,3.66) {};
\begin{scope}[shift={(4.65,3.66)},scale=0.33]
  \draw[rcCoral!42,line width=0.35pt] (-1.12,-0.62)--(1.12,-0.62);
  \draw[rcCoral,line width=1.12pt]
    (-1.04,-0.62)--(-0.22,-0.62)--(1.02,0.78);
  \fill[rcCoral] (-0.22,-0.62) circle (1.7pt);
\end{scope}
\node[anchor=west,align=center,text width=2.22cm] at (5.14,3.66)
  {Shallow ReLU\\$\longmapsto$\\\textcolor{rcCoral}{\bfseries Radon BV$^2$}};

\node[note,text width=6.90cm] at (3.96,1.08)
  {One framework recovers classical native spaces\\
   and their representer theorems.};

\draw[panel] (8.06,0.10) rectangle (15.90,8.25);
\node[panelmark,text=rcCoral] at (8.38,7.86) {(c)};
\node[paneltitle] at (9.14,7.86) {Deep ReLU geometry};
\node[sectiontag] at (15.58,7.86) {\cref{sec:deep-nets}};
\draw[rcCoral!55,line width=0.55pt] (8.38,7.48)--(15.58,7.48);

\node[align=center,text width=7.10cm] at (11.98,6.82)
  {Depth $L$ induces $(\mathcal F_L,R_L)$\\
   with homogeneity exponent $p=2/L$.};

\draw[->,draw=rcGray!72,line width=0.60pt] (9.48,5.92)--(14.48,5.92);
\node[arrowlabel,fill=white,inner xsep=4pt] at (11.98,5.92)
  {increasing depth: $L\uparrow$, $p\downarrow$};

\DrawPBall{(9.60,5.00)}{1}{rcBall}{0.43}
\DrawPBall{(11.98,5.00)}{2/3}{rcBall}{0.43}
\DrawPBall{(14.36,5.00)}{1/2}{rcBall}{0.43}

\node[align=center,text width=2.12cm,font=\footnotesize] at (9.60,4.03)
  {$L=2$, $p=1$\\\textbf{convex}};
\node[align=center,text width=2.12cm,font=\footnotesize] at (11.98,4.03)
  {$L=3$, $p=2/3$\\\textbf{nonconvex}};
\node[align=center,text width=2.12cm,font=\footnotesize] at (14.36,4.03)
  {$L=4$, $p=1/2$\\\textbf{nonconvex}};

\node[
  deepbox,
  text width=6.95cm,
  minimum height=1.45cm
] at (11.98,2.80)
  {{\bfseries Quasi-Banach geometry}\\
   {\footnotesize $R_L^{L/2}$ is a quasi-seminorm, $\F_L$ is quasi-Banach.\\
   For $L>2$, the unit ball is nonconvex.}};

\node[
  smallresult,
  draw=rcGold!68,
  fill=rcGold!4,
  text width=3.12cm,
  minimum height=1.10cm
] at (9.99,1.18)
  {finite-width\\representer theorem\\hidden widths $\leq N^2$};

\node[
  smallresult,
  draw=rcTeal!68,
  fill=rcTeal!4,
  text width=3.12cm,
  minimum height=1.10cm
] at (13.99,1.18)
  {depth hierarchy\\$\mathcal F_L\hookrightarrow\mathcal F_{L+1}$\\$\mathcal F_2\subsetneq\mathcal F_3$};

\end{tikzpicture}
\endgroup

\caption{\textbf{From parameter costs to native spaces.}
\textup{(a)} Minimizing $C$ over all parameters that realize a function gives $R_\Theta$; its extension $R$ defines $\F=\operatorname{dom}R$. Under the hypotheses of \cref{sec:background,sec:rep-costs}, the parameter-space and function-space problems have equal infimal values and support minimizer transfer. \textup{(b)} The construction recovers the displayed classical native spaces and their representer theorems. \textup{(c)} For depth-$L$ ReLU networks with weight decay, $p=2/L$: $R_L^{L/2}$ is a quasi-seminorm, $\F_L$ is quasi-Banach, and the unit ball is nonconvex for $L>2$ under \cref{prop:deep-unit-ball-nonconvex}. The width and depth-hierarchy statements follow from \cref{thm:deep-relu-representer,thm:deep-native-space-nesting,prop:F2_F3_depth_separation}.}
\label{fig:representation-cost-framework}
\end{figure}
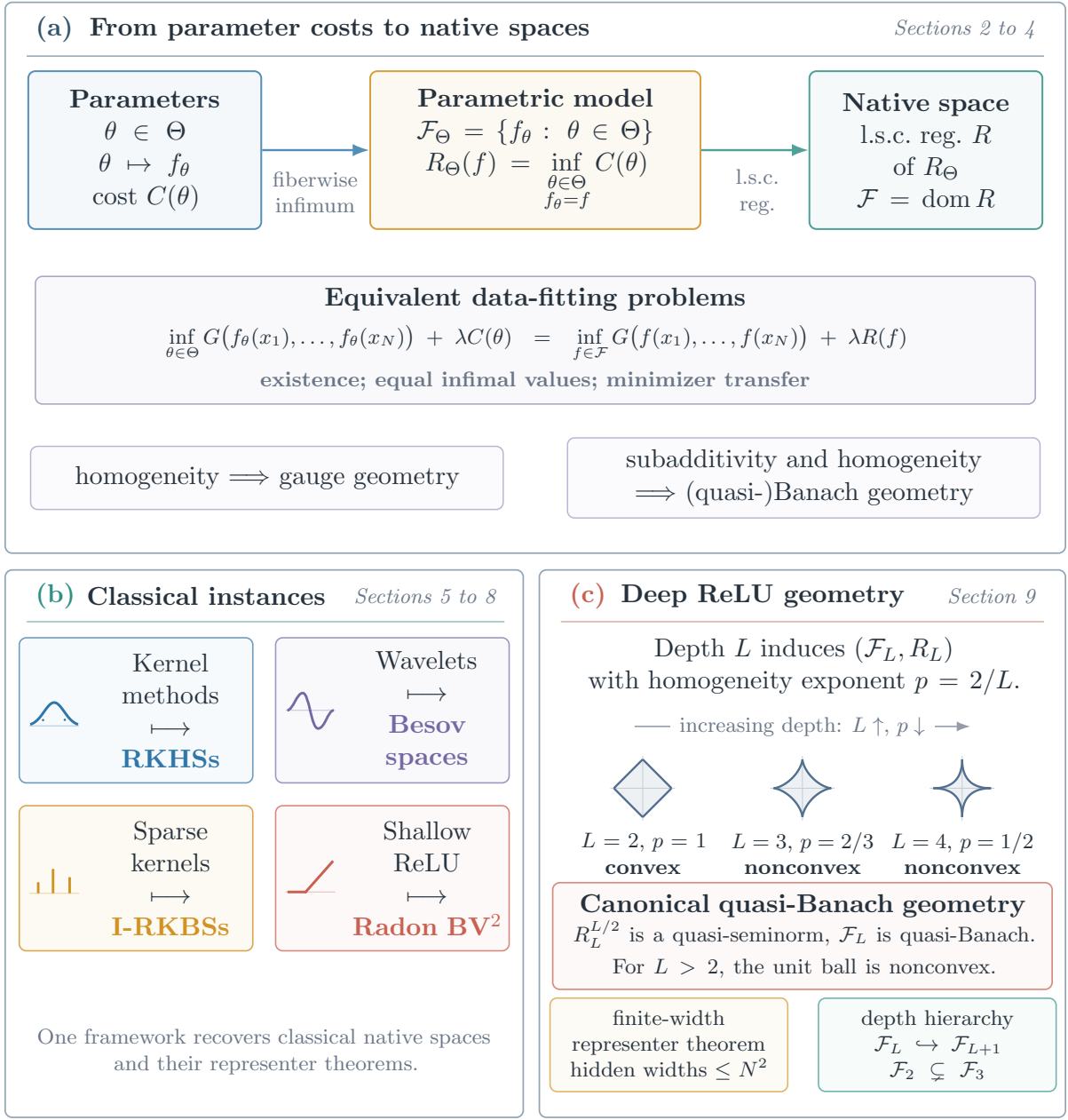

\subsection{Roadmap}
\label{sec:roadmap}

\Cref{fig:representation-cost-framework} summarizes the organization and main themes of the paper. \Crefrange{sec:background}{sec:Lip} develop the analytic foundations and the representation-cost framework. \Cref{sec:RKHS} recovers kernel methods and RKHSs, \Cref{sec:wavelet} treats wavelet methods and Besov spaces, \Cref{sec:I-RKBS} gives, to the best of our knowledge, a new perspective on sparse kernel methods through integral reproducing kernel Banach spaces, and \Cref{sec:shallow} treats shallow ReLU networks and Radon-domain bounded-variation spaces. In each case, the  framework also yields a corresponding representer theorem. Together, these sections show how methods normally treated disparately can be studied using a common set of functional-analytic tools via representation costs.

\Cref{sec:deep-nets} contains the principal new application of the framework. For deep ReLU neural networks with weight-decay regularization, we uncover a novel depth-dependent quasi-Banach geometry induced by the representation costs. The framework also yields a finite-width representer theorem in this setting. In particular, under suitable hypotheses, the unit ball of the representation cost is \emph{nonconvex} for depths $L>2$. We further show that the native spaces form a depth hierarchy, with strict separation already occurring between depths two and three.

\section{Preliminaries} \label{sec:background}
This section introduces the primary mathematical tools used in this paper. The central construction passes from a parametric representation cost, initially defined only on a model class $\F_\Theta$, to an extended representation cost on an ambient function space. This is accomplished by extending the cost by $+\infty$ outside $\F_\Theta$ and then taking its \emph{lower semicontinuous regularization} with respect to a chosen topology. We introduce the notation needed for this construction and record structural results that will be used repeatedly in this paper. The supporting results and proofs are collected in \cref{app:lsc-reg,app:coercive-mod-subspace-proof}; for pedagogical treatments of the underlying functional and variational analysis, see, e.g.,~\cite{brezis2011functional,buttazzo1989semicontinuity,conway2019course,nariciTopological2010}. Throughout the paper, $\eR\coloneqq\R\cup\{\pm\infty\}$ denotes the extended real line, and we adopt the convention that $\inf\varnothing=\min\varnothing=+\infty$.

\subsection{Lower Semicontinuity and Lower Semicontinuous Regularization}
Let $(\B,\tau)$ denote a topological space.

\begin{definition}
A function $K:\B\to\eR$ is called \emph{$\tau$-lower semicontinuous}
($\tau$-l.s.c.) if, for every $\alpha\in\R$, the sublevel set $\sls{K}{\alpha} \coloneqq \{f\in\B \st K(f)\leq\alpha\}$ is $\tau$-closed.
\end{definition}

\begin{definition}
The \emph{$\tau$-l.s.c.\ regularization} of a function $K:\B\to\eR$ is defined by
\begin{equation}
    \overline{K}^{\,\tau}(f) \coloneqq \sup\{J(f): J:\B\to\eR\text{ is $\tau$-l.s.c.\ and }J\leq K\}.
\end{equation}
Equivalently, $\overline{K}^{\,\tau}$ is the largest $\tau$-l.s.c.\ function bounded above by $K$. Thus, if $J$ is $\tau$-l.s.c.\ and $J\leq K$, then $J\leq \overline{K}^{\,\tau}\leq K$. We refer to this as the \emph{defining property of l.s.c.\ regularization}.
\end{definition}

\begin{definition}
The \emph{sequential $\tau$-l.s.c.\ regularization} of a function $K:\B\to\eR$ is defined by
\begin{align}\label{eq:seq_lsc_reg_def}
    \overline{K}^{\mathrm{seq},\tau}(f) \coloneqq \inf\left\{
        \liminf_{n\to\infty} K(f_n)
        \st
         (f_n)_{n\in\N}\subset\B,\ f_n \tauconverges f
    \right\}.
\end{align}
\end{definition}

In general, $\overline{K}^{\,\tau}\leq\overline{K}^{\mathrm{seq},\tau}\leq K$, and the first two functions need not coincide. We display the topology as a superscript whenever more than one topology is under consideration; when the topology is fixed or clear from context, we write simply $\overline K$ and $\seqreg K$. Equivalent descriptions, stability properties, and conditions under which the two regularizations coincide are developed in \cref{app:lsc-reg}.

\subsection{Constancy Spaces and Factorization}
Many representation costs are invariant along unpenalized directions. The following notions separate those directions from a complementary subspace on which the cost has better compactness or coercivity properties. In the sequel, suppose that $(B,\tau)$ is a topological vector space (TVS).

\begin{definition}\label{def:constancy_space}
Let $K:\B\to\eR$. The \emph{constancy space} of $K$, denoted $\const(K)$, is defined by
\begin{equation}
    \const(K) \coloneqq \{g\in\B \st K(f+\alpha g)=K(f)\text{ for every }f\in\B \text{ and every }\alpha\in\R\}.
\end{equation}
\end{definition}
By definition, $\const(K)$ is the largest subspace of directions along which $K$ is invariant under translation. Let $\Null\subset\B$ be a $\tau$-closed subspace. We say that $\Null$ is $\tau$-complemented if there exists a $\tau$-closed subspace $\M\subset\B$ and $\tau$-$\tau$-continuous linear projectors
\begin{equation}
    \P_{\Null}:\B\to\B
    \quad\text{and}\quad
    \P_{\M}:\B\to\B,
\end{equation}
such that $\P_{\Null}(\B)=\Null$, $\P_{\M}(\B)=\M$, and $\P_{\Null}+\P_{\M}=\Id$. In this case, $\B$ admits the topological direct-sum decomposition $\B=\M\oplus\Null$. In the locally convex spaces used throughout this paper, finite-dimensional subspaces are closed and complemented.

One of the key ingredients in this paper is the following factorization lemma. It says that if a function is constant along a complemented subspace, then its l.s.c.\ regularization can be computed entirely on any topological complement.

\begin{lemma}\label{lem:chatacterizing_R_mod_N}
Let $(\B,\tau)$ be a TVS and let $K:\B\to\eR$. Suppose $\Null\subset\const(K)$ is a $\tau$-closed subspace with $\tau$-complement $\M\subset\B$. Let $\P_{\M}:\B\to\M$ denote the corresponding $\tau$-$\tau$-continuous projector onto $\M$, and let $\tau_{\M}$ denote the subspace topology on $\M$. Then,
\begin{equation}
    \overline{K}^{\,\tau} = \overline{K|_{\M}}^{\,\tau_{\M}} \circ \P_{\M}.
\end{equation}
\end{lemma}
The proof is given in \cref{subsec:factorization-proof}.

\subsection{Dual Banach Spaces, Coercivity, and Compatible Topologies} \label{sec:dual-Banach}
We now specialize to the setting used throughout the paper for ambient function spaces. Let $(\B,\norm{\dummy}_{\B})$ be a Banach space with predual $(\A,\norm{\dummy}_{\A})$, i.e., $\B=\A'$ isometrically. For $f\in\B$ and $\mu\in\A$, we write $\ang{\mu,f}$ for the dual pairing and endow $\B$ with the \emph{weak$^*$ topology}
\begin{equation}
    w^*\coloneqq\sigma(\B,\A).
\end{equation}
The Banach--Alaoglu theorem states that the closed unit ball of $\B$ is $w^*$-compact. Consequently, a subset of $\B$ is $w^*$-compact if and only if it is $w^*$-closed and norm-bounded.

\begin{definition} \label{def:coercive}
A function $K:\B\to\eR$ is called \emph{coercive} if its sublevel sets are norm-bounded. That is, for every $\alpha\in\R$, there exists $\beta>0$ such that $\sls{K}{\alpha} \subset \sls{\norm{\dummy}_\B}{\beta}$.
\end{definition}

Many natural topologies on a dual Banach space agree with the weak$^*$ topology on norm-bounded subsets. The following definition isolates this property.
\begin{definition}\label{def:weakstar_compatible}
A Hausdorff vector-space topology $\tau$ on $\B$ is called \emph{weak$^*$-compatible} if it coincides with $w^*$ when restricted to any norm-bounded subset of $\B$.
\end{definition}

\begin{definition} \label[definition]{defn:coercive-mod-N}
Given a subspace $\Null \subset \B$, a function $K: \B \to \eR$ is said to be \emph{coercive modulo $\Null$} if, for every $\alpha \in \R$, there exists $\beta > 0$ such that for any $f \in \B$,
\begin{equation}
    K(f) \leq \alpha \quad\Rightarrow\quad \inf_{g \in \Null} \norm{f + g}_\B \leq \beta. \label{eq:coercive-mod-N}
\end{equation}
\end{definition}

\begin{proposition}\label{prop:coercive_mod_subspace_implies_seqlscreg_compatible}
Suppose $\A$ is separable, $\tau$ is a weak$^*$-compatible topology on $\B$, and $K:\B\rightarrow \eR$ is coercive modulo a $w^*$-complemented subspace $\Null \subset \const(K)$. Assume in addition that $\Null$ is sequentially $\tau$-closed. Then
\begin{equation}\label{eq:set1}
\overline{K}^{\,w^*}(f) = \overline{K}^{\mathrm{seq},\tau}(f)
  = \min\{\alpha \in \eR \st (f_n)_{n\in\N}\subset \dom(K),\
      f_n\tauconverges f,\ K(f_n) \rightarrow \alpha\}.
\end{equation}
\end{proposition}
The proof is given in \cref{subsec:recovery-sequence-proof}.

\section{Representation Costs and the Function-Space Perspective} \label{sec:rep-costs}

We now introduce the abstract representation-cost framework and the function-space perspective. The starting point is a parametric model class together with a parameter cost. The goal is to pass from this parameter-space description to a function-space description that is agnostic to the particular parameterization used to realize a function. Throughout this section, we work in the dual Banach space setting of \cref{sec:dual-Banach}. That is, $\B=\A'$ and $\B$ is endowed with the weak$^*$ topology $w^*=\sigma(\B,\A)$. Unless otherwise specified, all l.s.c.\ regularizations that follow are taken with respect to $w^*$. Thus, for a function $K:\B\to\eR$, we write $\overline K$ for $\overline K^{\,w^*}$.

Let $\Theta$ be a parameter space, let $\theta\mapsto f_\theta\in\B$ be a parametric model, and let
\begin{equation}
    \F_\Theta \coloneqq \{f_\theta:\theta\in\Theta\} \subset\B
\end{equation}
be the associated parametric model class. Given a parameter cost $C:\Theta\to[0,+\infty)$, the parametric representation cost is the function $\Rcirc:\F_\Theta\to[0,+\infty)$ defined by
\begin{equation}
\Rcirc(f) = \inf_{\substack{\theta \in \Theta \\ f=f_\theta}} C(\theta), \quad f \in \F_\Theta.
\end{equation}
Thus, $\Rcirc(f)$ is the smallest parameter cost among all parameterizations that realize the same function $f$. We identify $\Rcirc$ with its \emph{trivial extension} to $\B$, abusively denoted by $\Rcirc$, by setting
\begin{equation}
    \Rcirc(f) = \begin{cases}
        \Rcirc(f) & \text{if $f \in \F_\Theta$}, \\
        +\infty & \text{if $f \in \B \setminus \F_\Theta$}.
    \end{cases}
\end{equation}
Therefore, $\Rcirc: \B \to [0, +\infty]$ and $\dom (\Rcirc) = \F_\Theta$. The extended representation cost is defined by
\begin{equation}
    R (f) \coloneqq\overline{\Rcirc}(f) = \liminf_{g \to f} \Rcirc(g), \quad f \in \B.
\end{equation}
The associated native space is
\begin{equation}
    \F \coloneqq \dom(R) = \curly{f\in\B \st R(f)<+\infty}.
\end{equation}

The pair $(\F,R)$ is the function-space object induced by the parametric model and the parameter cost. The model class $\F_\Theta$ describes what is \emph{exactly} representable, while the native space $\F$ describes what is representable in a limiting sense with finite representation cost. In this sense, $(\F,R)$ captures the \emph{function-space inductive bias} induced by the parametric model \emph{and} the parameter cost.

The rest of this section develops the consequences of this construction. We first record structural interpretations of $R$, including its gauge formulation (\cref{sec:gauge}) and the quasi-normed structures that arise from homogeneous and subadditive costs (\cref{sec:quasi}). We then investigate how useful properties of $\Rcirc$ and $R$ can be verified directly from the parameter cost $C$ (\cref{sec:interplay}). Finally, we use these properties to prove existence of minimizers and to compare the parameter-space and function-space optimization problems (\cref{subsec:function-space-minimization}). Thus, we establish the sense in which the following equivalence holds:
\begin{equation}
    \min_{\theta \in \Theta} G\paren[\big]{f_\theta(x_1), \ldots, f_\theta(x_N)} + \lambda C(\theta) \quad\Leftrightarrow\quad \min_{f \in \F} G\paren[\big]{f(x_1), \ldots, f(x_N)} + \lambda R(f) \label{eq:equivalence}
\end{equation}
This equivalence is precisely the \emph{function-space perspective} of a data-fitting method.

\subsection{Properties and Interpretations of Representation Costs}
\label{subsec:representation-costs}

In this section, we collect structural properties of the extended representation cost. We show that $w^*$-l.s.c.\ regularization preserves basic algebraic and compactness properties. This allows for the identification of geometric and topological properties of a native space. In particular, homogeneous costs admit a \emph{gauge interpretation}, while suitable powers of subadditive homogeneous costs induce \emph{Banach} or \emph{quasi-Banach} native spaces.

\subsubsection{The Gauge Perspective} \label{sec:gauge}
The gauge interpretation is governed by homogeneity. Recall that a function $K:\B\to[0,+\infty]$ is called \emph{absolutely $p$-homogeneous}, for some $p>0$, if $K(0)=0$ and $K(\alpha f)=|\alpha|^p K(f)$ for every $\alpha\in\R\setminus\{0\}$ and $f\in\B$. This property is preserved under lower semicontinuous regularization whenever $\B$ is a TVS (\cref{prop:lscreg_preserves_properties}).

Assume that $\Rcirc$ is absolutely $p$-homogeneous for some $p>0$. Define the \emph{parametric unit ball}
\begin{equation}
    \U_\Theta \coloneqq \sls{\Rcirc}{1} = \curly{f\in\B \st \Rcirc(f)\leq1}
\end{equation}
and its $w^*$-closure
\begin{equation}
    \U
    \coloneqq
    \overline{\U_\Theta}.
\end{equation}
Since $R_\Theta$ is absolutely $p$-homogeneous, the set $\U_\Theta$ is \emph{star-shaped} with respect to the origin, i.e.,
\begin{equation}
t \, \U_\Theta \subset \U_\Theta, \quad 0\le t\le 1.
\end{equation}
By continuity of scalar multiplication, its $w^*$-closure $\U$ is also star-shaped. Thus, although $\U$ need not be convex, it is necessarily star-shaped. The following result says that $\U$ is exactly the unit ball of the extended representation cost.

\begin{lemma}
\label{lem:unit-ball-representation-cost}
If $\Rcirc$ is absolutely $p$-homogeneous for some $p>0$, then $\U = \sls{R}{1}$.
\end{lemma}

\begin{proof}
Since $R\leq\Rcirc$, we have $\U_\Theta\subset\sls{R}{1}$. Since $R$ is $w^*$-l.s.c., $\sls{R}{1}$ is $w^*$-closed. Hence $\U\subset\sls{R}{1}$. Conversely, let $f\in\B$ satisfy $R(f)\leq1$. If $\Rcirc(f)\leq1$, then $f\in\U_\Theta\subset\U$. Hence, assume $\Rcirc(f)>1$. We shall show that $f$ is a $w^*$-limit point of $\U_\Theta$ (i.e., $f \in \U$). First, suppose $R(f)<1$. Then for every $w^*$-neighborhood $\mathcal O$ of $f$, we have that $\inf_{g \in \mathcal{O}} \Rcirc(g) < 1$. Hence, there exists $g\in\mathcal O$ with $\Rcirc(g)<1$. In particular, $g \in \U_\Theta$ and $g \neq f$ since $\Rcirc(f) > 1$. Thus $f\in\U$.

It remains to treat the case $R(f)=1$. Let $\mathcal O$ be any $w^*$-neighborhood of $f$. Since scalar multiplication is $w^*$-continuous, there exists $0<t<1$ arbitrarily close to $1$ (from below) such that $tf\in\mathcal O$. In particular, since $\Rcirc(f) > 1$, we can choose $t$ sufficiently close to $1$ such that  $\Rcirc(tf) = t^p \Rcirc(f) > 1$. By \cref{prop:lscreg_preserves_properties}~\ref{item:homo-lsc-reg}, $R$ is absolutely $p$-homogeneous, so $R(tf)=t^pR(f)=t^p<1$. Thus, we can apply the argument in the previous paragraph to find that $tf\in\U$. Since $\mathcal O$ is a $w^*$-open neighborhood of $tf$, it follows that $\mathcal O\cap\U_\Theta\neq\varnothing$. Thus, every $w^*$-neighborhood of $f$ intersects $\U_\Theta$, and hence $f\in\U$. Therefore, $\sls{R}{1}\subset\U$, which completes the proof.
\end{proof}

Recall that the \emph{gauge} of a set $S\subset\B$ is the function $G_S:\B\to[0,+\infty]$ defined by
\begin{equation}
    G_S(f) \coloneqq \inf\curly*{\lambda>0 \st f/\lambda\in S}.
\end{equation}
This is the same \emph{Minkowski function} that underlies many geometric formulations of regularization and inverse problems; see, e.g., \cite{chandrasekaran2012convex,rockafellarConvex1970}.

\begin{theorem}
\label{thm:gauge-representation-cost}
If $\Rcirc$ is absolutely $p$-homogeneous for some $p>0$, then
\begin{equation}
    R(f)^{1/p}=G_\U(f),
    \quad f\in\B.
\end{equation}
\end{theorem}

\begin{proof}
By \cref{prop:lscreg_preserves_properties}~\ref{item:homo-lsc-reg}, $R$ is
absolutely $p$-homogeneous. By \cref{lem:unit-ball-representation-cost},
$\U=\sls{R}{1}$. Therefore,
\begin{equation}
    f/\lambda\in\U
    \quad\Leftrightarrow\quad
    R(f/\lambda)\leq1
    \quad\Leftrightarrow\quad
    \lambda^{-p}R(f)\leq1
    \quad\Leftrightarrow\quad
    \lambda\geq R(f)^{1/p}.
\end{equation}
Taking the infimum over $\lambda>0$ gives the result.
\end{proof}

\begin{remark}
The homogeneity assumption is necessary. Otherwise, the $w^*$-closure of the
parametric unit ball need not be the unit ball of the extended representation
cost. For example, suppose that $\Rcirc(0)=2$ and
\begin{equation}
    \Rcirc(f)
    =
    1+\norm{f}_{\B},
    \quad f\neq0.
\end{equation}
Then $\U_\Theta=\U=\varnothing$, while $\sls{R}{1}=\curly{0}$.
\end{remark}

\subsubsection{The (Quasi-)Norm Perspective and (Quasi-)Banach Native Spaces} \label{sec:quasi}

We next identify conditions under which the native space carries a natural Banach or quasi-Banach structure. The point is to allow representation costs that are powers of norms or quasi-norms such as squared Hilbert-space norms (\cref{sec:RKHS}) and the costs that arise in the context of shallow and deep neural networks, which are based on norms and quasi-norms, respectively (\cref{sec:shallow,sec:deep-nets}).

A function $K:\B\to[0,+\infty]$ is called \emph{subadditive} if $K(f+g)\leq K(f)+K(g)$ for every $f,g\in\B$. Together with homogeneity, subadditivity provides the bridge from the gauge perspective to quasi-normed geometry. Both properties are preserved under l.s.c.\ regularization by \cref{prop:lscreg_preserves_properties}.

\begin{definition}
A function $Q:\B\to[0,+\infty)$ is called a \emph{quasi-seminorm} if it is
absolutely $1$-homogeneous and there exists a constant $C>0$ such that $Q(f+g)\leq C\paren[\big]{Q(f)+Q(g)}$ for every $f,g\in\B$. If, in addition, $\ker Q \coloneqq \{f \in \B \st Q(f) = 0\} = \{0\}$, then $Q$ is called a \emph{quasi-norm}.
\end{definition}

\begin{definition}[{cf.~\cite[p.~319]{albiac2009lipschitz}}]
Let $0<p\leq 1$. A quasi-seminorm $Q: \B \to [0, +\infty)$ is called a \emph{$p$-seminorm} if $Q(f+g)^p\leq Q(f)^p+Q(g)^p$ for every $f,g\in\B$. If, in addition, $\ker Q=\{0\}$, then $Q$ is called a \emph{$p$-norm}. A quasi-normed space is called \emph{$p$-normable} if its topology is induced by an equivalent $p$-norm.
\end{definition}

The following elementary observation is the bridge between homogeneous
and subadditive functions and quasi-normed spaces.

\begin{proposition}
\label{prop:homogeneous-subadditive-quasinorm}
If $K:\B\to[0,+\infty]$ is absolutely $p$-homogeneous and subadditive for some $0<p\leq 1$, then $Q(f) \coloneqq K(f)^{1/p}$ defines a $p$-seminorm on $\dom(K) \coloneqq \{f\in\B:K(f)<+\infty\}$. In particular, $Q$ is a quasi-seminorm on $\dom(K)$.
\end{proposition}

\begin{proof}
First note that $\dom(K)$ is a vector space. Indeed, absolute
$p$-homogeneity implies that $\dom(K)$ is closed under scalar multiplication,
and subadditivity implies that it is closed under addition. The function $Q$ is nonnegative and absolutely $1$-homogeneous. Moreover, for
$f,g\in\dom(K)$, subadditivity gives
\begin{equation}
    Q(f+g)^p = K(f+g) \leq K(f)+K(g) = Q(f)^p+Q(g)^p.
\end{equation}
Thus, $Q$ is a $p$-seminorm on $\dom(K)$. Note that since $0<p\leq 1$, this also implies the quasi-triangle inequality with $C = 2^{1/p-1}$.
\end{proof}

Let $q>0$ and $0<p\leq1$. Suppose that there exists a function $Q_\Theta:\B\to[0,+\infty]$ such that $\Rcirc=Q_\Theta^q$, and such that $Q_\Theta$ is subadditive and absolutely $p$-homogeneous. We write $Q\coloneqq\overline{Q_\Theta}$ for its $w^*$-l.s.c.\ regularization. By \cref{prop:lscreg_properties}~\ref{item:cont-compose}, $R=Q^q$. Moreover, by \cref{prop:lscreg_preserves_properties}, the function $Q$ is subadditive and absolutely $p$-homogeneous. Thus, $Q^{1/p}=R^{1/(pq)}$ is a natural candidate for a (quasi-)seminorm on $\F$.

\begin{theorem}
\label{thm:representation-cost-quasi-banach}
Suppose $\B=\A'$ where $\A$ is separable. Let $q>0$ and $0<p\leq1$. Assume that $\Rcirc=Q_\Theta^q$, where $Q_\Theta:\B\to[0,+\infty]$ is subadditive and absolutely $p$-homogeneous. Assume further that $\Rcirc$ is coercive. Then $R^{1/(pq)}$ is a $p$-norm on $\F$. Moreover, $\F$ is complete with respect to the metric
\begin{equation}
    d(f,g) \coloneqq R(f-g)^{1/q}, \quad f,g\in\F.
\end{equation}
In particular, $\F$ is a $p$-normable quasi-Banach space. When $p=1$, $\F$ is
a Banach space.
\end{theorem}

\begin{proof}
Let $Q\coloneqq\overline{Q_\Theta}$. As above, $R=Q^q$, and $Q$ is
subadditive and absolutely $p$-homogeneous. By
\cref{prop:homogeneous-subadditive-quasinorm}, $Q^{1/p}=R^{1/(pq)}$ is a
$p$-seminorm on $\F=\dom(R)=\dom(Q)$.

Since $\Rcirc$ is coercive, by \Cref{prop:lscreg_is_coercive} so is $R$.
We claim that $Q^{1/p}$ has trivial null space. Indeed, if
$Q(f)=0$, then $R(\gamma f)=Q(\gamma f)^q=0$ for every $\gamma>0$. If
$f\neq0$, this contradicts the norm-boundedness of the sublevel set
$\curly{R\leq0}$. Hence $f=0$, and so $Q^{1/p}=R^{1/(pq)}$ is a $p$-norm.

It remains to prove completeness with respect to the metric $d(f,g) = R(f-g)^{1/q} = Q(f-g)$. Let $(f_n)_{n\in\N}\subset\F$ be Cauchy
with respect to $d$. Choose $N\in\N$ such that
\begin{equation}
    Q(f_n-f_m)\leq1,
    \quad n,m\geq N.
\end{equation}
Then, for all $n\geq N$,
\begin{equation}
    Q(f_n)
    \leq
    Q(f_n-f_N)+Q(f_N)
    \leq
    1+Q(f_N).
\end{equation}
This shows $\sup_n Q(f_n)<+\infty$, or equivalently, $\sup_n R(f_n)<+\infty$. Since $R$ is coercive,
$(f_n)_{n\in\N}$ is norm-bounded in $\B$. Since $\A$ is separable, $w^*$ is
metrizable on bounded subsets of $\B$. By Banach--Alaoglu, there exists a
subsequence $(f_{n_j})_{j\in\N}$ and some $f\in\B$ such that $f_{n_j}\starconverges f$.
Since $Q$ is $w^*$-l.s.c.,
\begin{equation}
    Q(f)
    \leq
    \liminf_{j\to\infty}Q(f_{n_j})
    <+\infty.
\end{equation}
This shows that $f\in\F$. We now claim that $d(f_n, f)\rightarrow 0$. Indeed, let $\varepsilon>0$. Since $(f_n)_{n\in\N}$ is Cauchy with respect to $d$, there exists $N\in\N$ such that $Q(f_n-f_m)\leq\varepsilon$, for all $n,m\geq N$. Fix $n\geq N$. Passing to the subsequence with $n_j\geq N$, we have $(f_n-f_{n_j})\starconverges (f_n-f)$. By $w^*$-lower semicontinuity of $Q$,
\begin{equation}
    Q(f_n-f)
    \leq
    \liminf_{j\to\infty}Q(f_n-f_{n_j})
    \leq
    \varepsilon.
\end{equation}
Thus $d(f_n,f)=Q(f_n-f)\to0$, and so $\F$ is complete with respect to $d$.
\end{proof}

\begin{remark}
The Aoki--Rolewicz theorem states that every quasi-normed space is $p$-normable for some $0<p\leq 1$~\cite{aoki1942locally,rolewicz1957metric}. Observe that in \cref{prop:homogeneous-subadditive-quasinorm}, the exponent is not obtained abstractly from the Aoki--Rolewicz theorem, but that it is inherited directly. Indeed, if $K$ is absolutely $p$-homogeneous and subadditive, then $K^{1/p}$ is a $p$-seminorm. This will be important in \cref{sec:deep-nets} when we identify explicit $p$-normed quasi-Banach native spaces associated with deep neural networks.
\end{remark}

The previous theorem assumes coercivity on all of $\B$. In many examples, the representation cost is invariant along an unpenalized subspace, such as constants, biases, or affine components. The correct substitute is coercivity modulo this subspace. If the subspace is complemented, one obtains a \textit{bona fide} quasi-norm by adding the ambient norm on the unpenalized component.

\begin{corollary}
\label{cor:F_is_quasi_banach}
Suppose $\B=\A'$ where $\A$ is separable. Let $q>0$ and $0<p\leq1$. Assume that $\Rcirc=Q_\Theta^q$, where $Q_\Theta:\B\to[0,+\infty]$ is subadditive and absolutely $p$-homogeneous. Let $\Null\subset\ker(\Rcirc)$ be a $w^*$-complemented subspace, and suppose that
$R_\Theta$
is coercive modulo $\Null$. Let $\P_\Null:\B\to\B$ be any $w^*$-$w^*$-continuous projector onto $\Null$. Then, $\F$ is a quasi-Banach space under the quasi-norm
\begin{equation}
    \norm{f}_{\F}
    \coloneqq
    \max\curly{
        R(f)^{1/(pq)},
        \norm{\P_\Null f}_{\B}
    },
    \quad f\in\F.
\end{equation}
More precisely, $\norm{\dummy}_{\F}$ is a $p$-norm. When $p=1$, $\F$ is a
Banach space.
\end{corollary}

\begin{proof}
Let $\M\coloneqq\ker(\P_\Null)$ and $\P_\M\coloneqq\Id-\P_\Null$. Since $\Rcirc=Q_\Theta^q$, we have $\ker(\Rcirc)=\ker(Q_\Theta)$. Next, since $Q_\Theta$ is subadditive and absolutely $p$-homogeneous, $\ker(Q_\Theta)=\const(Q_\Theta)$. Thus, $\Null\subset\const(Q_\Theta)$. Also, $Q_\Theta$ is coercive modulo $\Null$, since its level-sets are in one-to-one correspondence with $R_\Theta$.

Now define $Q\coloneqq\overline{Q_\Theta}$. Then by
\cref{lem:chatacterizing_R_mod_N}, we have
\begin{equation}
    Q=Q_\M\circ\P_\M,
    \label{eq:fact-M-root}
\end{equation}
where $Q_\M$ denotes the
$w^*$-l.s.c.\ regularization of
$Q_\Theta|_\M:\M\to[0,+\infty]$ with respect to the subspace topology on
$\M$. Consequently, $R=Q^q=(Q_\M\circ\P_\M)^q$.

By \cref{prop:coercive-mod-N-equivalences}~\ref{item:coercive_on_M}, $R_\Theta|_\M$ is coercive on $\M$.
Since $\M$ is a $w^*$-closed subspace of $\B=\A'$, it is canonically a dual Banach space. More precisely, if
\begin{equation}
    \M_\perp \coloneqq \curly{\mu\in\A \st \langle\mu,m\rangle=0\text{ for all }m\in\M},
\end{equation}
then $\M=(\A/\M_\perp)'$ isometrically, and the corresponding weak$^*$ topology is the subspace topology inherited from $\B$. The quotient $\A/\M_\perp$ is separable. Applying
\cref{thm:representation-cost-quasi-banach} on $\M$ shows that
\begin{equation}
    \F_\M
    \coloneqq
    \curly{f\in\M \st Q_\M(f)<+\infty}
\end{equation}
is quasi-Banach under the $p$-norm $Q_\M^{1/p}$. Moreover, because $\Null\subset\ker(\Rcirc)$, we have $\Null\subset\F$, and the factorization \cref{eq:fact-M-root} gives the topological direct-sum decomposition $\F=\F_\M\oplus\Null$. Under this decomposition, for every $f \in \F$, write
\begin{equation}
    \norm{f}_{\F} = \max\curly{ Q_\M(\P_\M f)^{1/p}, \norm{\P_\Null f}_{\B} }.
\end{equation}
This is the maximum $p$-norm on the product of the quasi-Banach space $\F_\M$ and the Banach space $\Null$. Hence $\F$ is quasi-Banach, and $\norm{\dummy}_{\F}$ is a $p$-norm.
\end{proof}

\begin{remark}
The case $p=1$ yields Banach native spaces. For example, $R$ being a squared Hilbert-space norm corresponds to $q=2$ and $p=1$. When $0<p<1$, the construction gives a $p$-normed quasi-Banach structure. In applications, a separate issue is whether the unit ball of the representation cost is convex. This distinction is familiar from finite-dimensional regularization by $x \mapsto \|x\|_p^p$, $0<p<1$, whose unit ball is nonconvex. A similar phenomenon occurs for depth-$L$ ReLU neural networks in \cref{sec:deep-nets}. In particular, under the hypotheses of \cref{prop:deep-unit-ball-nonconvex}, the representation cost has a \emph{nonconvex} unit ball.
\end{remark}

\subsubsection{Interplay Between Parameter Costs and Representation Costs} \label{sec:interplay}

The preceding results place assumptions directly on the parametric representation cost $\Rcirc$, or, equivalently, on $Q_\Theta$ that satisfies $\Rcirc=Q_\Theta^q$. In applications, these objects are induced by the parameter cost $C: \Theta \to [0, +\infty)$. We now record parameter-level conditions that make the preceding assumptions operational. The preservation of the relevant properties when passing from $\Rcirc$ to $R$ then follows from the l.s.c.\ regularization properties collected in \cref{app:lsc-reg}.

\begin{definition}
Let $q>0$. The parameter cost $C: \Theta \to [0, +\infty)$ is called \emph{parameter
$q$-subadditive} if, for every $\theta,\vartheta\in\Theta$, there exists
$\eta\in\Theta$ such that
\begin{equation}
    f_\eta=f_\theta+f_\vartheta \quad\text{and}\quad C(\eta)^{1/q} \leq C(\theta)^{1/q}+C(\vartheta)^{1/q}.
\end{equation}
\end{definition}

\begin{definition}
Let $q>0$ and $p>0$. The parameter cost $C: \Theta \to [0, +\infty)$ is called
\emph{parameter $(q,p)$-homogeneous} if there exists $\theta_0\in\Theta$ such
that
\begin{equation}
    f_{\theta_0}=0
    \quad\text{and}\quad
    C(\theta_0)=0,
\end{equation}
and, for every $\theta\in\Theta$ and every
$\alpha\in\R\setminus\curly{0}$, there exists
$\theta_\alpha\in\Theta$ such that
\begin{equation}
    f_{\theta_\alpha}=\alpha f_\theta \quad\text{and}\quad C(\theta_\alpha)^{1/q} \leq \abs{\alpha}^p C(\theta)^{1/q}.
\end{equation}
\end{definition}

\begin{proposition}
\label{prop:parameter-root-properties-transfer}
Let $q>0$ and $p>0$. Suppose that $C: \Theta \to [0, +\infty)$ is parameter $q$-subadditive and
parameter $(q,p)$-homogeneous. Define the ``rooted'' parametric representation cost
\begin{equation}
    Q_\Theta(f) \coloneqq \inf\curly{ C(\theta)^{1/q} \st \theta\in\Theta,\ f_\theta=f },
    \quad f\in\F_\Theta,
\end{equation}
and identify $Q_\Theta$ with its trivial extension to $\B$. Then $\Rcirc = Q_\Theta^q$. Moreover, $Q_\Theta$ is subadditive and absolutely $p$-homogeneous on $\B$. Consequently, if we write $Q\coloneqq\overline{Q_\Theta}$, then $Q$ is subadditive and absolutely $p$-homogeneous, and $R=Q^q$. Equivalently, $Q=R^{1/q}$.
\end{proposition}

\begin{proof}
Since $t\mapsto t^{1/q}$ is increasing on $[0,+\infty)$, for every $f\in\F_\Theta$ we have
\begin{equation}
    Q_\Theta(f) = \inf\curly{ C(\theta)^{1/q} \st \theta\in\Theta,\ f_\theta=f } = \Rcirc(f)^{1/q}.
\end{equation}
Thus, $\Rcirc=Q_\Theta^q$ on $\F_\Theta$, and the identity also holds for the
trivial extensions to $\B$.

We now prove that $Q_\Theta$ is absolutely $p$-homogeneous. The condition in the definition of parameter $(q,p)$-homogeneity gives $Q_\Theta(0)=0$. Let $f\in\F_\Theta$ and $\alpha\in\R\setminus\curly{0}$. Given $\varepsilon>0$, choose $\theta\in\Theta$ such that $f_\theta=f$ and $C(\theta)^{1/q} \leq Q_\Theta(f)+\varepsilon$. By parameter $(q,p)$-homogeneity, there exists $\theta_\alpha\in\Theta$ such that $f_{\theta_\alpha}=\alpha f$ and $C(\theta_\alpha)^{1/q} \leq \abs{\alpha}^p C(\theta)^{1/q}$. Hence, $Q_\Theta(\alpha f) \leq \abs{\alpha}^p Q_\Theta(f)+\abs{\alpha}^p\varepsilon$. Taking $\varepsilon\to0$ gives
\begin{equation}
    Q_\Theta(\alpha f) \leq \abs{\alpha}^p Q_\Theta(f). \label{eq:homo-forward}
\end{equation}
Applying the same inequality to $h=\alpha f$ with scalar $\alpha^{-1}$ gives
\begin{equation}
    Q_\Theta(f) = Q_\Theta(\alpha^{-1}h) \leq \abs{\alpha}^{-p}Q_\Theta(h) = \abs{\alpha}^{-p}Q_\Theta(\alpha f). \label{eq:homo-backward}
\end{equation}
Combining \cref{eq:homo-forward,eq:homo-backward} yields $Q_\Theta(\alpha f) = \abs{\alpha}^p Q_\Theta(f)$. The same identity holds for the trivial extension to $\B$, since parameter
$(q,p)$-homogeneity implies
\begin{equation}
    f\in\F_\Theta
    \quad\Leftrightarrow\quad
    \alpha f\in\F_\Theta,
    \quad \alpha\neq0.
\end{equation}
Thus, $Q_\Theta$ is absolutely $p$-homogeneous on $\B$.

We next prove subadditivity. Let $f,g\in\F_\Theta$. Given $\varepsilon>0$,
choose $\theta,\vartheta\in\Theta$ such that $f_\theta=f$,
$f_\vartheta=g$,
\begin{equation}
    C(\theta)^{1/q} \leq Q_\Theta(f)+\varepsilon \quad\text{and}\quad C(\vartheta)^{1/q} \leq Q_\Theta(g)+\varepsilon.
\end{equation}
By parameter $q$-subadditivity, there exists $\eta\in\Theta$ such that
$f_\eta=f+g$ and
\begin{equation}
    C(\eta)^{1/q}
    \leq
    C(\theta)^{1/q}+C(\vartheta)^{1/q}.
\end{equation}
Therefore,
\begin{equation}
    Q_\Theta(f+g)
    \leq
    Q_\Theta(f)+Q_\Theta(g)+2\varepsilon.
\end{equation}
Taking $\varepsilon\to0$ proves subadditivity on $\F_\Theta$. The same
inequality holds for the trivial extension to $\B$, since it is automatic
whenever one of the two terms on the right-hand side is $+\infty$.

Finally, by \cref{prop:lscreg_preserves_properties}, the $w^*$-l.s.c.\
regularization $Q=\overline{Q_\Theta}$ is subadditive and absolutely
$p$-homogeneous. Since $t\mapsto t^q$ is continuous and nondecreasing on
$[0,+\infty]$, \cref{prop:lscreg_properties}~\ref{item:cont-compose} gives
\begin{equation}
    R = \overline{\Rcirc} = \overline{Q_\Theta^q} = \overline{Q_\Theta}^{\,q} = Q^q.
\end{equation}
This completes the proof.
\end{proof}

In particular, the subadditivity and homogeneity of the parameter cost $C: \Theta \to [0, +\infty)$ transfer to the representation cost. These are two key properties needed to apply \cref{thm:representation-cost-quasi-banach,cor:F_is_quasi_banach}. We next give a parameter-level condition that implies coercivity, or more generally coercivity modulo a subspace.

\begin{definition}
\label{defn:param-coercive}
Let $\Null\subset\B$ be a $w^*$-closed subspace. The parameter cost $C: \Theta \to [0, +\infty)$ is
called \emph{parameter coercive modulo $\Null$} if there exists a
nondecreasing continuous function $\gamma:[0,+\infty)\to[0,+\infty)$ such
that
\begin{equation}
    \lim_{t\to+\infty}\gamma(t)=+\infty
\end{equation}
and
\begin{equation}
    \norm{[f_\theta]}_{\B/\Null}
    =
    \inf_{g\in\Null}\norm{f_\theta+g}_{\B}
    \leq
    \gamma(C(\theta)),
    \quad \theta\in\Theta.
\end{equation}
\end{definition}

\begin{proposition}
\label{lem:suff_cond_R_coercive}
Let $\Null\subset\B$ be a $w^*$-closed subspace. If $C: \Theta \to [0, +\infty)$ is parameter coercive modulo
$\Null$, then both $\Rcirc$ and $R$ are coercive modulo $\Null$. In particular, if $\gamma$ is the nondecreasing continuous function as specified in \cref{defn:param-coercive}, then
\begin{equation}
\inf_{g \in \Null}\norm{f + g}_\B \leq \gamma(R(f)) \leq \gamma(\Rcirc(f)), \quad f\in\B, \label{eq:param-extend-coercive-est}
\end{equation}
with the convention that $\gamma(+\infty)=+\infty$.
\end{proposition}

\begin{proof}
Let $\gamma$ be as in \cref{defn:param-coercive}. We first show that
$\Rcirc$ is coercive modulo $\Null$. Fix $\alpha<+\infty$. If $\Rcirc(f)\leq\alpha$, then $f\in\F_\Theta$. By the definition of $\Rcirc$, for every $\varepsilon>0$ there exists $\theta\in\Theta$ such that $f_\theta=f$ and $C(\theta)\leq\alpha+\varepsilon$. Hence,
\begin{equation}
    \norm{[f]}_{\B/\Null} = \norm{[f_\theta]}_{\B/\Null} \leq \gamma(C(\theta)) \leq \gamma(\alpha+\varepsilon).
    \label{eq:quotient-estimate-cost}
\end{equation}
Taking $\varepsilon\to0$ gives $\norm{[f]}_{\B/\Null} \leq \gamma(\alpha)$. Thus, $\sls{\Rcirc}{\alpha}$ is bounded in $\B/\Null$, and so $\Rcirc$ is coercive modulo $\Null$ by \cref{prop:coercive-mod-N-equivalences} \ref{item:coerive_in_quotient}.

We now prove the corresponding statement for $R$. Define $q_\Null(f) \coloneqq \norm{[f]}_{\B/\Null}$, for $f\in\B$. Since $\Null$ is $w^*$-closed, the quotient seminorm $q_\Null$ is $w^*$-l.s.c. The estimate above implies
\begin{equation}
    q_\Null(f)
    \leq
    \gamma(\Rcirc(f)),
    \quad f\in\B, \label{eq:param123}
\end{equation}
where the inequality is trivial when $\Rcirc(f)=+\infty$. Since $q_\Null$ is
$w^*$-l.s.c., the defining property of l.s.c.\ regularization implies $q_\Null \leq \overline{\gamma\circ\Rcirc}$. By \cref{prop:lscreg_properties}~\ref{item:cont-compose},
\begin{equation}
    \overline{\gamma\circ\Rcirc} = \gamma\circ R. \label{eq:aaaaa}
\end{equation}
Therefore,
\begin{equation}
    \norm{[f]}_{\B/\Null}
    \leq
    \gamma(R(f)),
    \quad f\in\B. \label{eq:ext123}
\end{equation}
Consequently, if $R(f)\leq\alpha$, then $\norm{[f]}_{\B/\Null} \leq \gamma(\alpha)$. Thus, $R$ is coercive modulo $\Null$. \Cref{eq:param-extend-coercive-est} follows from \cref{eq:param123,eq:aaaaa,eq:ext123} after noting that $\inf_{g \in \Null}\norm{f + g}_\B = \norm{[f]}_{\B/\Null}$ and $R \leq \Rcirc$.
\end{proof}

\subsection{Function-Space Minimization: Existence and Transfer of Minimizers}
\label{subsec:function-space-minimization}

We now turn from the construction of representation costs to the associated function-space optimization problems. The goal is to make precise the sense in which the parametric problem
\begin{equation}
    \min_{\theta\in\Theta}
    G\paren[\big]{f_\theta(x_1),\ldots,f_\theta(x_N)}
    +\lambda C(\theta)
\end{equation}
is equivalent to the function-space problem
\begin{equation}
    \min_{f\in\F}
    G\paren[\big]{f(x_1),\ldots,f(x_N)}
    +\lambda R(f).
\end{equation}
There are three separate issues. First, one needs conditions under which the function-space problem admits minimizers. Second, one needs conditions under which the function-space problem and the original parametric problem have the same infimal value. Third, one needs conditions under which minimizers of one problem induce minimizers of the other. The latter two issues can be handled simultaneously.

\subsubsection{Existence of Minimizers}
\label{sec:existence-minimizers}

We begin with existence. We consider objectives of the form
\begin{equation}
    L(f)
    \coloneqq
    G(\AOp f)+\lambda R(f),
    \quad f\in\B,
    \label{eq:general-objective}
\end{equation}
where $\AOp:\B\to\R^M$ is a linear observation operator, $G:\R^M\to[0,+\infty]$ is a data-fidelity term, and $\lambda>0$. The finite data-fitting problem above is recovered by taking $\AOp$ to be an evaluation operator, for instance
\begin{equation}
    \AOp f = \paren[\big]{f(x_1),\ldots,f(x_N)}.
\end{equation}
If the functions are $\R^D$-valued, then in this example $M=DN$. Since $R(f)=+\infty$ for $f\notin\F$, the objective $L$ is finite only on the native space $\F$. Thus minimizing $L$ over $\B$ is equivalent to minimizing $L$ over $\F$.

\begin{theorem}
\label{thm:existence_of_min}
Consider the objective $L$ in \cref{eq:general-objective}. Suppose that there
exists a finite-dimensional subspace $\Null\subset\B$ such that
\begin{enumerate}[label=(\roman*)]
    \item $R$ is coercive modulo $\Null$ and
    \item $\AOp:\B\to\R^M$ is $w^*$-continuous and satisfies
    \begin{equation}
        \ker(\AOp)\cap\Null=\curly{0}.
    \end{equation}
\end{enumerate}
Assume also that $\lambda>0$, that $G:\R^M\to[0,+\infty]$ is l.s.c.\ and
coercive, and that $L$ is proper. Then there exists $f^\star\in\F$ such that
\begin{equation}
    L(f^\star)
    =
    \inf_{f\in\F}L(f).
\end{equation}
\end{theorem}

\begin{proof}
Since $R$ is $w^*$-l.s.c., $\AOp$ is $w^*$-continuous, and $G$ is l.s.c., the objective $L$ is $w^*$-l.s.c. It remains to show that $L$ is coercive. Since $\Null$ is finite-dimensional, it is $w^*$-closed and admits a $w^*$-closed complement $\M$. Let $\P_\Null$ and $\P_\M$ denote the associated $w^*$-$w^*$-continuous projectors. Since
\begin{equation}
    \ker(\AOp)\cap\Null=\curly{0},
\end{equation}
the restriction $\AOp|_\Null:\Null\to\R^M$ is injective. By finite
dimensionality, there exists a constant $c_1>0$ such that
\begin{equation}
    \norm{g}_{\B}
    \leq
    c_1\norm{\AOp g}_2,
    \quad g\in\Null.
\end{equation}
Moreover, since $\AOp$ is $w^*$-continuous, it is norm-continuous, so there
exists $c_2>0$ such that
\begin{equation}
    \norm{\AOp f}_2
    \leq
    c_2\norm{f}_{\B},
    \quad f\in\B.
\end{equation}
For $f\in\B$, write $f=\P_\Null f+\P_\M f$.
Then
\begin{equation}
\|f\|_\B \leq \|\P_\Null f\|_\B + \|\P_\M f\|_{\B}
\end{equation}
and
\begin{equation}
    \norm{\P_\Null f}_{\B}
    \leq
    c_1\norm{\AOp\P_\Null f}_2
    \leq
    c_1\norm{\AOp f}_2
    +
    c_1\norm{\AOp\P_\M f}_2
    \leq
    c_1\norm{\AOp f}_2
    +
    c_1c_2\norm{\P_\M f}_{\B}.
\end{equation}
Combining the previous estimates yields a constant $C>0$ such that
\begin{equation}
    \norm{f}_{\B}
    \leq
    C\paren[\big]{
        \norm{\AOp f}_2
        +
        \norm{\P_\M f}_{\B}
    },
    \quad f\in\B.
    \label{eq:existence-main-bound}
\end{equation}

We now show that $L$ is coercive. Fix $\alpha<+\infty$ and suppose
$L(f)\leq\alpha$. Since $G\geq0$, $R\geq0$, and $\lambda>0$, we have
\begin{equation}
    G(\AOp f)\leq\alpha
    \quad\text{and}\quad
    R(f)\leq\alpha/\lambda.
\end{equation}
By coercivity of $G$, the first inequality uniformly bounds
$\norm{\AOp f}_2$. By coercivity of $R$ modulo $\Null$, the second inequality uniformly
bounds $\|\P_\M f\|_\B$ by \cref{prop:coercive-mod-N-equivalences}~\ref{item:coercive_projector_equiv}. Hence \cref{eq:existence-main-bound} uniformly
bounds $\norm{f}_{\B}$. Thus $L$ is coercive.

Therefore $L$ is proper, $w^*$-l.s.c., and coercive. By
\cref{prop:coercive_implies_infcpt}, $L$ is $w^*$-inf-compact, and by
\cref{prop:lsc_infcpt_give_min}, $L$ attains its minimum on $\B$. Since
$L(f)<+\infty$ implies $R(f)<+\infty$, every minimizer belongs to $\F$.
\end{proof}

\begin{remark}
The theorem applies to equality-constrained problems. For instance, if $G=\iota_{\curly{y}}$ (indicator function) for some $y\in\R^M$, then minimizing $L$ is the same as minimizing $R$ subject to $\AOp f=y$. More generally, one may take $G=\iota_{\C}$ for any compact set $\C\subset\R^M$.
\end{remark}

\subsubsection{Equality of Infimal Values and Transfer of Minimizers}
\label{sec:equivalence-minimizers}

We now compare the parameter-space problem, the problem over the parametric model class, and the function-space problem. Let $\AOp:\B\to\R^M$ be a $w^*$-continuous linear operator, let $G:\R^M\to[0,+\infty]$, and let $\lambda>0$. We write
\begin{align}
    J_{\rm par}(\theta)
    &\coloneqq
    G(\AOp f_\theta)+\lambda C(\theta),
    \quad \theta\in\Theta, \\
    J_\Theta(f)
    &\coloneqq
    G(\AOp f)+\lambda\Rcirc(f),
    \quad f\in\B, \\
    J(f)
    &\coloneqq
    G(\AOp f)+\lambda R(f),
    \quad f\in\B.
\end{align}
Here $\Rcirc$ denotes the trivial extension of the parametric representation cost to $\B$. Thus, $\Rcirc(f)=+\infty$ for $f\notin\F_\Theta$, while $R(f)=+\infty$ for $f\notin\F$. Consequently, minimizing $J_\Theta$ over $\B$ is the same as minimizing it over $\F_\Theta$, and minimizing $J$ over $\B$ is the same as minimizing it over $\F$.

The parameter-space problem and the problem over $\F_\Theta$ always have the same infimal value. Indeed, by the definition of $\Rcirc$,
\begin{equation}
\label{eq:parametric-model-class-equality}
    \inf_{\theta\in\Theta} J_{\rm par}(\theta)
    =
    \inf_{f\in\F_\Theta} J_\Theta(f).
\end{equation}
Thus, the only nontrivial question is when the extension from $\Rcirc$ to $R=\overline{\Rcirc}$ preserves the infimal value. Once this is known, the transfer of minimizers between problems follows immediately, which is summarized next.

\begin{lemma}
\label{lem:minimizer-transfer}
Suppose that
\begin{equation}
\label{eq:common-infimal-value}
    \inf_{\theta\in\Theta} J_{\rm par}(\theta)
    =
    \inf_{f\in\F_\Theta} J_\Theta(f)
    =
    \inf_{f\in\F} J(f).
\end{equation}
Then the following minimizer transfer statements hold.
\begin{enumerate}[label=(\roman*)]
    \item If $\theta^\star\in\Theta$ minimizes the parameter-space problem,
    then $f_{\theta^\star}\in\F$ minimizes the function-space problem.

    \item \label{item:twoooo} If $f^\star\in\F$ minimizes the function-space problem,
    $f^\star\in\F_\Theta$, and $R(f^\star)=\Rcirc(f^\star)$, then
    $f^\star$ minimizes the problem over $\F_\Theta$.

    \item If, in addition to the hypotheses of \ref{item:twoooo}, there exists $\theta^\star\in\Theta$ such that
    \begin{equation}
        f_{\theta^\star}=f^\star
        \quad\text{and}\quad
        C(\theta^\star)=\Rcirc(f^\star),
    \end{equation}
    then $\theta^\star$ minimizes the parameter-space problem.
\end{enumerate}
\end{lemma}

\begin{proof}
Let $v$ denote the common infimal value in \cref{eq:common-infimal-value}.
Suppose first that $\theta^\star$ minimizes the parameter-space problem.
Since $R\leq\Rcirc$ and $\Rcirc(f_{\theta^\star})\leq C(\theta^\star)$, we
have
\begin{equation}
    J(f_{\theta^\star})
    =
    G(\AOp f_{\theta^\star})+\lambda R(f_{\theta^\star})
    \le
    G(\AOp f_{\theta^\star})+\lambda C(\theta^\star)
    =
    v.
\end{equation}
Since $v=\inf_{f\in\F}J(f)$, this proves that $f_{\theta^\star}$ minimizes
the function-space problem.

Next, suppose $f^\star\in\F$ minimizes the function-space problem,
$f^\star\in\F_\Theta$, and $R(f^\star)=\Rcirc(f^\star)$. Then
\begin{equation}
    J_\Theta(f^\star)
    =
    G(\AOp f^\star)+\lambda\Rcirc(f^\star)
    =
    G(\AOp f^\star)+\lambda R(f^\star)
    =
    v.
\end{equation}
Since $v=\inf_{f\in\F_\Theta}J_\Theta(f)$, this proves that $f^\star$
minimizes the problem over $\F_\Theta$.

Finally, suppose in addition that $f^\star=f_{\theta^\star}$ and
$C(\theta^\star)=\Rcirc(f^\star)$. Then
\begin{equation}
    J_{\rm par}(\theta^\star)
    =
    G(\AOp f_{\theta^\star})+\lambda C(\theta^\star)
    =
    G(\AOp f^\star)+\lambda\Rcirc(f^\star)
    =
    v.
\end{equation}
Since $v=\inf_{\theta\in\Theta}J_{\rm par}(\theta)$, this proves that
$\theta^\star$ minimizes the parameter-space problem.
\end{proof}

We next investigate when \cref{eq:common-infimal-value} holds. We begin with the basic case in which the data-fidelity term is continuous. This covers the usual empirical-risk-minimization problems with continuous losses such as the least-squares loss.

\begin{theorem}
\label{thm:equiv-minimizer}
Suppose that $\AOp:\B\to\R^M$ is $w^*$-continuous, that $G:\R^M\to\R$ is continuous, and that $\lambda>0$. Then
\begin{equation}
\label{eq:equiv-minimizer-infima}
    \inf_{\theta\in\Theta}
    G(\AOp f_\theta)+\lambda C(\theta)
    =
    \inf_{f\in\F_\Theta}
    G(\AOp f)+\lambda\Rcirc(f)
    =
    \inf_{f\in\F}
    G(\AOp f)+\lambda R(f).
\end{equation}
Consequently, the minimizer transfer statements in \cref{lem:minimizer-transfer} apply.
\end{theorem}

\begin{proof}
The first equality is \cref{eq:parametric-model-class-equality}. For the second equality, the map $f\mapsto G(\AOp f)$ is $w^*$-continuous. Therefore, by \cref{prop:lscreg_properties}~\ref{item:lsc-reg-cont-sum}, $J=\overline{J_\Theta}$. Hence, by \cref{prop:lscreg_properties}~\ref{item:inf-preserve},
\begin{equation}
    \inf_{f\in\F}J(f)
    =
    \inf_{f\in\B}J(f)
    =
    \inf_{f\in\B}J_\Theta(f)
    =
    \inf_{f\in\F_\Theta}J_\Theta(f).
\end{equation}
Together with \cref{eq:parametric-model-class-equality}, this proves
\cref{eq:equiv-minimizer-infima}. The minimizer transfer statements then
follow from \cref{lem:minimizer-transfer}.
\end{proof}

The continuity assumption on $G$ excludes equality constraints and other discontinuous data-fidelity terms. We now remove this assumption under an additional structural condition on the representation cost. The result applies to representation costs that are powers of subadditive homogeneous functions. This includes, for example, (quasi-)Banach-norm costs and squared Hilbert-norm costs. The latter are not themselves subadditive, but their square roots are.
\begin{lemma}
\label{lem:fiberwise-equivalence}
Let $q>0$ and $p>0$. Assume that $\Rcirc=Q_\Theta^q$, where
$Q_\Theta:\B\to[0,+\infty]$ is subadditive and absolutely
$p$-homogeneous. Let $Q\coloneqq\overline{Q_\Theta}$, so that $R=Q^q$.
Let $\AOp:\B\to\R^M$ be a $w^*$-continuous linear operator such that
\begin{equation}
\label{eq:A-surjective-on-model-class}
    \AOp(\F_\Theta)=\R^M.
\end{equation}
Then, for every $y\in\R^M$,
\begin{equation}
\label{eq:fiberwise-equivalence}
    \inf_{\substack{\theta\in\Theta\\ \AOp f_\theta=y}} C(\theta)
    =
    \inf_{\substack{f\in\F_\Theta\\ \AOp f=y}} \Rcirc(f)
    =
    \inf_{\substack{f\in\F\\ \AOp f=y}} R(f).
\end{equation}
\end{lemma}

\begin{proof}
The first equality follows directly from the definition of $\Rcirc$. We prove
the second equality by first showing that
\begin{equation}
\label{eq:Q-fiberwise-equivalence}
    \inf_{\substack{f\in\F_\Theta\\ \AOp f=y}} Q_\Theta(f)
    =
    \inf_{\substack{f\in\F\\ \AOp f=y}} Q(f).
\end{equation}
The ``$\geq$'' direction follows from $Q\leq Q_\Theta$. For the reverse
direction, fix $f\in\F$ with $\AOp f=y$ and $Q(f)<+\infty$, and let
$\varepsilon>0$. By \cref{eq:A-surjective-on-model-class}, there exist
$h_1,\ldots,h_M\in\F_\Theta=\dom(Q_\Theta)$ such that
$\AOp h_i=\e_i$, $i=1,\ldots,M$, where $\{\e_i\}_{i=1}^M$ is the standard
basis of $\R^M$. Choose $\delta>0$ small enough so that
\begin{equation}
    \sum_{i=1}^M \delta^p Q_\Theta(h_i)<\varepsilon.
\end{equation}
The set $\mathcal O \coloneqq \curly{ g\in\B \st \abs{[\AOp g]_i-y_i}<\delta,\ i=1,\ldots,M }$ is a $w^*$-neighborhood of $f$. Since $Q=\overline{Q_\Theta}$, there exists $g\in\mathcal O$ such that $Q_\Theta(g)\leq Q(f)+\varepsilon$. In particular, $g\in\dom(Q_\Theta)=\F_\Theta$. Define
\begin{equation}
    \widetilde g
    \coloneqq
    g-\sum_{i=1}^M\paren[\big]{[\AOp g]_i-y_i}h_i.
\end{equation}
Since $\dom(Q_\Theta)$ is a vector space, $\widetilde g\in\F_\Theta$. Moreover,
\begin{equation}
    \AOp\widetilde g
    =
    \AOp g-\sum_{i=1}^M\paren[\big]{[\AOp g]_i-y_i}\e_i
    =
    y.
\end{equation}
By subadditivity and absolute $p$-homogeneity of $Q_\Theta$,
\begin{equation}
    Q_\Theta(\widetilde g)
    \leq
    Q_\Theta(g)
    +
    \sum_{i=1}^M
    \abs{[\AOp g]_i-y_i}^p Q_\Theta(h_i)
    \leq
    Q(f)+2\varepsilon.
\end{equation}
Therefore,
\begin{equation}
    \inf_{\substack{u\in\F_\Theta\\ \AOp u=y}} Q_\Theta(u)
    \leq
    Q(f)+2\varepsilon.
\end{equation}
Taking $\varepsilon\to0$ and then taking the infimum over all
$f\in\F$ with $\AOp f=y$ gives the ``$\leq$'' direction in
\cref{eq:Q-fiberwise-equivalence}. Finally, since $t\mapsto t^q$ is increasing
on $[0,+\infty]$,
\begin{equation}
    \inf_{\substack{f\in\F_\Theta\\ \AOp f=y}} \Rcirc(f)
    =
    \paren[\Bigg]{
        \inf_{\substack{f\in\F_\Theta\\ \AOp f=y}} Q_\Theta(f)
    }^q
    =
    \paren[\Bigg]{
        \inf_{\substack{f\in\F\\ \AOp f=y}} Q(f)
    }^q
    =
    \inf_{\substack{f\in\F\\ \AOp f=y}} R(f).
\end{equation}
This completes the proof.
\end{proof}

\begin{theorem}
\label{thm:equivalence_of_min_quasi_norm_case}
\label{cor:parametric-function-space-equivalence}
Let $q>0$ and $p>0$. Assume that $\Rcirc=Q_\Theta^q$, where
$Q_\Theta:\B\to[0,+\infty]$ is subadditive and absolutely
$p$-homogeneous. Let $\AOp:\B\to\R^M$ be a $w^*$-continuous linear operator
such that
\begin{equation}
    \AOp(\F_\Theta)=\R^M.
\end{equation}
Then, for every $G:\R^M\to[0,+\infty]$ and every $\lambda>0$,
\begin{equation}
\label{eq:equivalence_of_min_quasi_seminorm_case}
    \inf_{\theta\in\Theta}
    G(\AOp f_\theta)+\lambda C(\theta)
    =
    \inf_{f\in\F_\Theta}
    G(\AOp f)+\lambda\Rcirc(f)
    =
    \inf_{f\in\F}
    G(\AOp f)+\lambda R(f).
\end{equation}
Consequently, the minimizer transfer statements in
\cref{lem:minimizer-transfer} apply.
\end{theorem}

\begin{proof}
The first equality is \cref{eq:parametric-model-class-equality}. For the
second equality, observe that
\begin{align}
    \inf_{f\in\F_\Theta}
    G(\AOp f)+\lambda\Rcirc(f)
    &=
    \inf_{y\in\R^M}
    \curly[\Big]{
        G(y)
        +
        \lambda
        \inf_{\substack{f\in\F_\Theta\\ \AOp f=y}}
        \Rcirc(f)
    } \nonumber\\
    &=
    \inf_{y\in\R^M}
    \curly[\Big]{
        G(y)
        +
        \lambda
        \inf_{\substack{f\in\F\\ \AOp f=y}}
        R(f)
    } \nonumber\\
    &=
    \inf_{f\in\F}
    G(\AOp f)+\lambda R(f),
\end{align}
where the middle equality follows from \cref{lem:fiberwise-equivalence}.
Thus the three infimal values agree, and the minimizer transfer statements
follow from \cref{lem:minimizer-transfer}.
\end{proof}

\begin{remark}
\label{rem:how-to-use-parametric-function-space-equivalence}
In \cref{sec:RKHS,sec:wavelet,sec:I-RKBS,sec:shallow,sec:deep-nets}, where we instantiate the representation-cost framework for various data-fitting methods, \cref{cor:parametric-function-space-equivalence} is the tool that allows us to pass between the parameter-space problem and the function-space problem. The hypotheses of \cref{thm:equivalence_of_min_quasi_norm_case} are typically verified through \cref{prop:parameter-root-properties-transfer}. This gives the precise sense in which the function-space problem is the \emph{nonparametric description} of the original parametric method, and establishes the equivalence \cref{eq:equivalence}.
\end{remark}

\subsection{Discussion}

The results in this section make precise the relationship between the parameter-space and function-space problems. The passage from $\Rcirc$ to $R$ is a closure operation. Indeed, we embed $\F_\Theta$ in an ambient dual Banach space $\B$ and replace the trivial extension of $\Rcirc$ by its $w^*$-l.s.c.\ regularization. The native space $\F=\dom(R)$ is the resulting collection of functions with finite extended representation cost. This can be thought of as a \emph{relaxation} of the parametric problem. Indeed, the problem over $\F$ is posed over a typically larger class of functions and uses $R$ instead of $\Rcirc$. The results in \cref{thm:equiv-minimizer,thm:equivalence_of_min_quasi_norm_case,cor:parametric-function-space-equivalence} show that, under the stated hypotheses, this relaxation is tight at the level of infimal values and that parameter-space minimizers induce function-space minimizers. The converse lifting statement requires the additional hypotheses in \cref{lem:minimizer-transfer}.

This passage from parameters to functions is related to the realization-space perspective of~\cite{elbrachter2019degenerate}, which studies the degeneracies of the \emph{realization map} $\theta \mapsto f_\theta$ and the extent to which optimization over parameters can be understood through optimization over realized functions. Our viewpoint is complementary: Rather than attempting to invert the realization map or remove its degeneracies, we quotient out the parameterization through the cost $\Rcirc$ and then pass to its $w^*$-l.s.c.\ regularization. Thus, the central object is the extended representation cost $R$ and its native space $\F$, and not the structure of the realization map itself.

It is important to note that the native space is defined \emph{before} any data-fitting problem is specified. Once the ambient dual Banach space $\B$, the model class $\F_\Theta$, and the parameter cost $C: \Theta \to [0, +\infty)$ are fixed, the construction
\begin{equation}
    \Rcirc \rightsquigarrow R=\overline{\Rcirc},
    \quad
    \F=\dom(R),
\end{equation}
does not depend on the data or the choice of loss function. Thus, the native space $\F$ is induced by the parametric model, parameter cost, and ambient function space.

Finally, our minimizer-transfer results resemble \emph{abstract representer theorems}~\cite{BoyerRepresenter,BrediesRepresenter,UnserUnifyingRepresenter}, but should also be distinguished. A representer theorem gives conditions under which a problem over a function space has a minimizer with a finite parametric representation. The results in this section do not by themselves give such a representation or a bound on the number of parameters. Rather, they provide the bridge between the parametric problem and the function-space problem. Interestingly, this bridge is also not restricted to convex regularizers, which is the setting of~\cite{BoyerRepresenter,BrediesRepresenter,UnserUnifyingRepresenter}. In particular, our results apply to representation costs with nonconvex unit balls, such as those arising for deep neural networks (\cref{prop:deep-unit-ball-nonconvex}), whose native spaces, as we shall see in \cref{sec:deep-nets}, are naturally quasi-Banach.

\section{Lipschitz Spaces and ``Universal'' Topologies} \label{sec:Lip}

The construction of an extended representation cost depends on the choice of an ambient function space and topology. In the preceding section, this topology was taken to be the weak$^*$ topology of a dual Banach space containing the parametric model class. For many models used in data science, a natural common ambient space is a Lipschitz space. The reason is not that Lipschitz regularity is always the main object of study, but rather that Lipschitz spaces provide a common dual Banach space in which many parametric model classes can be embedded.

The key feature of this choice is that, on norm-bounded subsets of Lipschitz spaces, the weak$^*$ topology coincides with pointwise convergence. We follow the presentation of Lipschitz spaces, Arens--Eells spaces, and their weak$^*$ topologies from~\cite[Chapters~2--3]{weaver2018lipschitz}. Thus, whenever the relevant representation cost controls the Lipschitz norm, the weak$^*$ closure of the model class can be understood in terms of pointwise limits. This is useful because pointwise convergence does not depend on the particular parameterization, architecture, or representation cost. It is therefore a natural ``universal'' topology for comparing different data-fitting methods in function space.

Throughout this section, $(\X,\rho)$ denotes a separable metric space with distinguished base point $\e\in\X$. We first treat scalar-valued functions and then pass to finite-dimensional vector-valued functions. Let $\Lip(\X)$ denote the space of Lipschitz functions $f:\X\to\R$, with
Lipschitz seminorm
\begin{equation}
    \abs{f}_{\Lip(\X)}
    \coloneqq
    \sup_{x\neq y}
    \frac{\abs{f(x)-f(y)}}{\rho(x,y)}.
\end{equation}
Let $\Lip_0(\X)\coloneqq\{f\in\Lip(\X):f(\e)=0\}$.
Then $\abs{\dummy}_{\Lip(\X)}$ is a norm on $\Lip_0(\X)$. We endow $\Lip(\X)$
with the norm
\begin{equation}
    \norm{f}_{\Lip(\X)}
    \coloneqq
    \max\{\abs{f}_{\Lip(\X)},\abs{f(\e)}\}.
\end{equation}
With this norm, $\Lip(\X)$ is isometrically identified with $\Lip_0(\X)\times\R$ via the map $f\mapsto (f-f(\e),f(\e))$, where $\Lip_0(\X)\times\R$ is endowed with the norm $\norm{(g,c)} \coloneqq \max\{\abs{g}_{\Lip(\X)},\abs{c}\}$.

The Arens--Eells space $\AE(\X)$, also called the Lipschitz-free space,
is a canonical predual of $\Lip_0(\X)$; see, e.g.,
\cite[Chapter~2]{weaver2018lipschitz}. Concretely, $\AE(\X)$ is the
completion of the linear span of the set of signed measures $\{\delta_x-\delta_y\st x,y\in\X\}$ under the norm
\begin{equation}
\norm{m}_{\AE(\X)}
  \coloneqq \inf\left\{\sum_i \abs{a_i}\rho(x_i,y_i)
   \st m = \sum_i a_i(\delta_{x_i}-\delta_{y_i})\right\}.
\end{equation}
In particular, $\delta_x-\delta_\e\in\AE(\X)$ for every $x\in\X$.
The duality $\Lip_0(\X)\simeq\AE(\X)'$ is isometric, so $\Lip(\X)$
is isometrically identified with the dual of $\AE(\X)\times\R$, where
$\AE(\X)\times\R$ is endowed with the norm
$\norm{(\mu,c)} \coloneqq \norm{\mu}_{\AE(\X)}+\abs{c}$. We denote the corresponding weak$^*$ topology by
\begin{equation}
    w^*
    \coloneqq
    \sigma\paren[\big]{\Lip(\X),\AE(\X)\times\R}.
\end{equation}
Since $\X$ is separable, $\AE(\X)$ is separable; see, e.g., \cite[p.~91]{godefroy2015survey}. Hence $\AE(\X)\times\R$ is separable.

The following proposition is the standard description of the weak$^*$ topology on bounded subsets of Lipschitz spaces.
\begin{proposition}
\label{prop:pointwise-is-wstar-compatible-lip}
The topology of pointwise convergence coincides with the weak$^*$ topology on norm-bounded subsets of
$\Lip(\X)$, i.e., it is weak$^*$-compatible (\cref{def:weakstar_compatible}).  In particular, $f_n\starconverges f$ in $\Lip(\X)$ if and only if
\begin{equation}
    \sup_{n\in\N}\|f_n\|_{\Lip(\X)}<+\infty\quad \text{and}\quad f_n(x)\to f(x) \quad\text{for every }x\in\X.
\end{equation}
\end{proposition}
\begin{proof}
The coincidence of the weak$^*$ and pointwise topologies on norm-bounded
sets is proved in \cite[Theorem~3.3 and Corollary~3.4]{weaver2018lipschitz}.
The result then follows
by \Cref{lem:seq_star_converge_equiv}.
\end{proof}

We now pass to finite-dimensional vector-valued functions. Let $D\in\N$ and endow $\R^D$ with its Euclidean norm $\norm{\dummy}_2$. Let $\Lip(\X;\R^D)$ denote the space of Lipschitz maps $f:\X\to\R^D$, with Lipschitz seminorm
\begin{equation}
    \abs{f}_{\Lip(\X;\R^D)}
    \coloneqq
    \sup_{x\neq y}
    \frac{\norm{f(x)-f(y)}_2}{\rho(x,y)}.
\end{equation}
We endow $\Lip(\X;\R^D)$ with the norm
\begin{equation}
    \norm{f}_{\Lip(\X;\R^D)}
    \coloneqq
    \max\{\abs{f}_{\Lip(\X;\R^D)},\norm{f(\e)}_2\}.
\end{equation}
Let $\Lip_0(\X;\R^D) \coloneqq \{f\in\Lip(\X;\R^D):f(\e)=0\}$.
By Weaver's universal property of the Arens--Eells space (see~\cite[Theorem~3.6 and p.~85]{weaver2018lipschitz}), for every Banach space $V$ there is an isometric identification of $\Lip_0(\X;V)$ with $\mathcal L(\AE(\X),V)$, the Banach space of bounded linear operators from $\AE(\X)$ to $V$ endowed with the operator norm.
Taking $V=\R^D$ and identifying $(\R^D)'$ with $\R^D$ by the Euclidean inner product gives the isometric identification
\begin{equation}
    \Lip_0(\X;\R^D)
    \simeq
    \mathcal L(\AE(\X),\R^D)
    \simeq
    \bigl(\AE(\X)\,\widehat\otimes_\pi\,\R^D\bigr)',
\end{equation}
where $\widehat\otimes_\pi$ denotes the projective tensor product. Since
\begin{equation}
    \Lip(\X;\R^D)
    \simeq
    \Lip_0(\X;\R^D)\oplus_\infty\R^D,
\end{equation}
it follows that $\Lip(\X;\R^D)$ is isometrically the dual of
\begin{equation}
    \bigl(\AE(\X)\,\widehat\otimes_\pi\,\R^D\bigr)\oplus_1\R^D.
\end{equation}
We denote the corresponding weak$^*$ topology by $w^*$. Since $\X$ is separable and $D<+\infty$, this predual is separable. Hence, $\Lip(\X;\R^D)$ has a separable predual.

\begin{remark} \label{rem:vv-lip-con}
The weak$^*$ topology $w^*$ on $\Lip(\X;\R^D)$ is equivalently described componentwise: $f_n \starconverges f$ in $\Lip(\X;\R^D)$ if and only if $f_n^{(i)} \starconverges f^{(i)}$ in $\Lip(\X)$ for each $i=1,\dots,D$. Thus, by \Cref{prop:pointwise-is-wstar-compatible-lip},
$f_n\starconverges f$ in $\Lip(\X;\R^D)$ if and only if
\begin{equation}
    \sup_{n\in\N}\|f_n\|_{\Lip(\X;\R^D)}<+\infty\quad \text{and}\quad f_n(x)\to f(x)~\text{in}~\R^D \quad\text{for every }x\in\X.
\end{equation}
In particular, the topology of pointwise convergence is weak$^*$-compatible on $\Lip(\X;\R^D)$.
\end{remark}

We now apply this observation to representation costs. Let $\F_\Theta\subset\Lip(\X;\R^D)$ and let $R_\Theta:\F_\Theta\to[0,+\infty]$ be a parametric representation cost. We identify $R_\Theta$ with its trivial extension to $\Lip(\X;\R^D)$ by setting $R_\Theta(f)=+\infty$ for $f\not\in\F_\Theta$, and define $R\coloneqq\overline{R_\Theta}$, the $w^*$-l.s.c.\ regularization. Recall that $\const(\Rcirc) \subset \Lip(\X;\R^D)$ denotes the constancy space of $\Rcirc$, i.e., the largest subspace of directions along which $\Rcirc$ is invariant under translation (see \cref{def:constancy_space}).

\begin{theorem} \label{thm:rep_cost_lip_spaces}
Assume $R_\Theta$ is coercive modulo a subspace $\Null \subset \const(R_\Theta)$ that is $w^*$-complemented and sequentially closed in the topology of pointwise convergence. Then, for every $f\in\Lip(\X;\R^D)$,
\begin{equation}\label{eq:lip_lsc_reg}
    R(f)
    =
    \min \{
        \alpha\in[0,+\infty]
        \st
        (f_n)_{n\in\N}\subset\F_\Theta,\ f_n\rightarrow f \text{ pointwise},\ \Rcirc(f_n)\rightarrow \alpha\}.
\end{equation}
\end{theorem}
\begin{proof}
The predual of $\Lip(\X;\R^D)$ is separable because $\X$ is separable and
$D<+\infty$. By \cref{rem:vv-lip-con}, the topology of pointwise convergence is weak$^*$-compatible on $\Lip(\X;\R^D)$. Therefore, the result follows from \cref{prop:coercive_mod_subspace_implies_seqlscreg_compatible}.
\end{proof}

When the input space $\X$ is compact, the topology of pointwise convergence in
\cref{thm:rep_cost_lip_spaces} can be replaced by the topology of uniform
convergence.

\begin{theorem}
\label{cor:uniform-representation-cost}
Suppose $\X$ is compact. Assume $R_\Theta$ is coercive modulo a subspace $\Null \subset \const(R_\Theta)$ that is $w^*$-complemented and sequentially closed in the topology of uniform convergence. Then, for every $f\in\Lip(\X;\R^D)$,
\begin{align}\label{eq:uniform-representation-cost}
    R(f)
    =
    \min \{
        \alpha\in[0,+\infty]
        \st
        (f_n)_{n\in\N}\subset\F_\Theta,\ \norm{f_n-f}_{L^\infty(\X;\R^D)}\to 0,\ \Rcirc(f_n)\to \alpha\}.
\end{align}
\end{theorem}
\begin{proof}
On norm-bounded subsets of $\Lip(\X;\R^D)$, pointwise convergence and uniform convergence agree when $\X$ is compact; this follows componentwise from the scalar-valued result \cite[Proposition~2.39 and Corollary~2.40]{weaver2018lipschitz}. Therefore, the topology of uniform convergence is weak$^*$-compatible on $\Lip(\X;\R^D)$ (since the topology of pointwise convergence is). The result follows from \cref{prop:coercive_mod_subspace_implies_seqlscreg_compatible}.
\end{proof}

Alternatively, the next result shows we may express $R$ in terms of a limiting infimum of $\Rcirc$ over uniform neighborhoods of~$f$.
\begin{corollary}
\label{cor:uniform-envelope-coercive}
Under the assumptions of
\cref{cor:uniform-representation-cost},
for every $f\in\Lip(\X;\R^D)$ we have
\begin{align}\label{eq:lip_lsc_reg_metric_version}
    R(f)
    & =
    \lim_{\varepsilon\to0^+}
    \inf\left\{
        R_\Theta(g):
        g\in\F_\Theta,\
        \norm{f-g}_{L^\infty(\X;\R^D)}<\varepsilon
    \right\} \\
    & =  \lim_{\varepsilon\to0^+}
    \inf\left\{
        C(\theta):
        \theta\in\Theta,\
        \norm{f-f_\theta}_{L^\infty(\X;\R^D)}<\varepsilon
    \right\}.\label{eq:lip_lsc_reg_metric_version_param_version}
\end{align}
\end{corollary}
\begin{proof}
\cref{cor:uniform-representation-cost} shows $R$ is equal to the sequential l.s.c.\ regularization of $\Rcirc$ in the topology of uniform convergence. Since this topology is metrizable
via the metric $d(f,g)=\norm{f-g}_{L^\infty(\X;\R^D)}$, we obtain \cref{eq:lip_lsc_reg_metric_version} by \cref{prop:lscreg_equivalents}~\ref{item:lsc_in_metric_space}. The equality between \cref{eq:lip_lsc_reg_metric_version} and \cref{eq:lip_lsc_reg_metric_version_param_version} is immediate from the definition of $R_\Theta$.
\end{proof}

\begin{remark}
If $\const(\Rcirc)$ is finite dimensional, then in the locally convex spaces considered in this paper, it is closed, topologically complemented, and sequentially closed and hence automatically satisfies the topological hypotheses of \cref{thm:rep_cost_lip_spaces,cor:uniform-representation-cost,cor:uniform-envelope-coercive}.
\end{remark}

\begin{remark}
The expression for the representation cost given in \cref{eq:lip_lsc_reg_metric_version_param_version} is the definition originally given in~\cite{SavareseInfWidth} in the context of shallow ReLU neural networks and the squared Euclidean-norm parameter cost. This definition was then used in~\cite{parkinson2024depth} for deep ReLU neural networks. Thus, our definition coincides with the definitions of~\cite{parkinson2024depth,SavareseInfWidth}, under the assumption that the input domain $\X$ is compact.  A slightly different definition for the representation cost was given in~\cite{OngieFunctionSpace} to account for the technicalities involved when considering the non-compact input domain $\X = \R^d$. In \cref{app:equivalence} we prove that this definition is also equal to the definition given in this paper.
\end{remark}

We conclude this section by characterizing $w^*$-continuous linear maps out of $\Lip(\X;\R^D)$. This criterion will be used in later sections to verify that data-fitting operators, such as point evaluations, are compatible with the weak$^*$ topology.

\begin{proposition}
\label{prop:lipschitz-wstar-continuous-linear-maps}
Let $\AOp:\Lip(\X; \R^D) \to\R^M$ be linear. Then $\AOp$ is $w^*$-continuous if
and only if, for every sequence $(f_n)_{n\in\N}\subset\Lip(\X;\R^D)$ and every
$f\in\Lip(\X;\R^D)$ such that $f_n \to f$ pointwise and $\sup_{n\in\N}\norm{f_n}_{\Lip(\X;\R^D)}<+\infty$, we have $\AOp f_n\to \AOp f$ (in $\R^M$). In particular, if $\AOp$ is sequentially continuous in the topology of pointwise convergence, then $\AOp$ is $w^*$-continuous.
\end{proposition}

\begin{proof}
By a corollary of the Krein--Smulian theorem (see~\cite[Corollary~12.8]{conway2019course}), a linear form on $\Lip(\X;\R^D)$ is $w^*$-continuous if and only if it is $w^*$-sequentially continuous. Applying this componentwise to $\AOp$, and using the equivalence between $w^*$-convergence and pointwise convergence of norm-bounded sequences in  $\Lip(\X;\R^D)$ (\cref{prop:pointwise-is-wstar-compatible-lip,rem:vv-lip-con}), gives the result.
\end{proof}

\begin{remark}
\label{rem:w*-point-evaluation}
For any finite collection of points $\{x_j\}_{j=1}^N\subset\X$, the evaluation
operator
\begin{equation}
    \EOp:\Lip(\X;\R^D)\to(\R^D)^N,
    \quad
    \EOp f\coloneqq (f(x_1),\ldots,f(x_N)),
\end{equation}
is $w^*$-continuous. Indeed, if $(f_n)_{n\in\N}$ converges pointwise to $f$, then $\EOp f_n\to \EOp f$ in $(\R^D)^N$. The claim follows from \cref{prop:lipschitz-wstar-continuous-linear-maps}.
\end{remark}

\begin{remark}\label{rem:lip-sobolev}
If $\X$ is an open subset of $\R^d$ and $f \in \Lip(\X;\R^D)$, then  Rademacher's theorem guarantees that $f$ is differentiable almost everywhere,  and its classical derivative agrees a.e.\ with the weak derivative ${\D} f \in L^\infty(\X;\R^{D \times d})$. Moreover, if $\X$ is convex and endowed  with the Euclidean metric, then we have the identity
\begin{equation}\label{eq:equiv_lip_sobolev_norm}
|f|_{\Lip(\X;\R^D)} = \|{\D} f\|_{L^\infty(\X;\R^{D \times d})}
\coloneqq \esssup_{x\in\X} \|{\D} f(x)\|_{\op}.
\end{equation}
See, e.g.,  \cite[Theorem~4, p.~292]{evans2010partial}. When $\X$ is additionally bounded, the identity above implies that the space $\Lip(\X;\R^D)$ coincides with  the Sobolev space $W^{1,\infty}(\X;\R^D)$, with equivalent norms. Therefore, one could consider the  space $W^{1,\infty}(\X;\R^D)$ in place of $\Lip(\X;\R^D)$ on such domains.
\end{remark}

\section{Kernel Methods and RKHSs} \label{sec:RKHS}
Kernel methods and RKHSs  are arguably among the most classical and well-studied methods from the function-space perspective. In this section, we show that kernel methods and RKHSs are compatible with our abstract construction of representation costs and native spaces. We focus on real kernels and RKHSs defined on a set $\X$.
\begin{definition}
    A \emph{kernel} is a function $\kernel: \X \times \X \to \R$ such that $\kernel(x, z) = \kernel(z, x)$, for all $x, z \in \X$, and 
    \begin{equation}
         \sum_{i=1}^K \sum_{j=1}^K a_i a_j \kernel(x_i, x_j) \geq 0,
    \end{equation}
    for all $K \in \mathbb{N}$, $\curly{x_j}_{j=1}^K \subset \X$, and $\curly{a_j}_{j=1}^K \subset \R$.
\end{definition}
\begin{definition}
    A real Hilbert space $(\Hil, \ang{\dummy, \dummy})$ of functions $\X \to \R$ is said to be a \emph{reproducing kernel Hilbert space} (RKHS) if point evaluations are bounded, i.e., for every $x \in \X$, there exists a constant $C_x > 0$ such that
    \begin{equation}
        \abs{f(x)} \leq C_x \norm{f}_\Hil, \quad f \in \Hil,
    \end{equation}
    where $\norm{f}_\Hil^2 \coloneqq \ang{f, f}$.
\end{definition}
A fundamental fact is that every RKHS $\Hil$ on $\X$ has a unique kernel $\kernel: \X \times \X \to \R$ that satisfies the \emph{reproducing property}: For every $x \in \X$,
\begin{equation}
    \ang{\kernel(x, \dummy), f} = f(x), \quad f \in \Hil.
\end{equation}
The reverse is also true in that every kernel $\kernel:\X \times \X \to \R$ induces a unique RKHS $\Hil$ on $\X$ such that $\kernel(\dummy, \dummy)$ is the reproducing kernel for $\Hil$. This one-to-one correspondence between RKHSs and kernels can be seen from the Riesz representation theorem and the Moore--Aronszajn theorem.

\subsection{Parametric Models and the Construction of RKHSs} \label{sec:param-RKHS}
Fix a kernel $\kernel: \X \times \X \to \R$ and let $\Hil$ denote the corresponding RKHS with inner product $\ang{\dummy, \dummy}$ and norm $\norm{\dummy}_\Hil$. Let
\begin{equation}
\Theta \coloneqq \{(a_1,x_1),\ldots,(a_K,x_K)\st a_j \in \R, x_j \in \X, K\in \mathbb{N}\}
\end{equation}
denote the parameter space and for a parameter $\theta = ((a_1,x_1),\ldots,(a_K,x_K))\in\Theta$ define the parametric model
\begin{equation}
f_{\theta} = \sum_{j=1}^K a_j \kernel(\dummy, x_j). \label{eq:kernel-machine}
\end{equation}
Note that due to the nature of the definition of $\Theta$, \cref{eq:kernel-machine} is a \emph{kernel machine} with an \emph{arbitrary} number of kernels $K \in \mathbb{N}$ and \emph{arbitrary} kernel centers $\{x_j\}_{j=1}^K$. Thus, we see that the parametric model class $\F_\Theta$ coincides with the (unclosed) span of the kernels, i.e., $\F_\Theta = \spn\{\kernel(x, \dummy)\}_{x \in \X} \eqqcolon \Hil_0$, which is simply the set of all kernel machines with kernel $\kernel(\dummy, \dummy)$.

The parameter cost for training kernel machines in, e.g., kernel ridge regression, is
\begin{equation} \label{eq:kernel-cost}
C(\theta) = \sum_{i=1}^K\sum_{j=1}^K a_ia_j \kernel(x_i,x_j), \quad \theta \in \Theta.
\end{equation}
The parametric representation cost is then
\begin{equation}
\Rcirc(f) = \norm{f}_{\Hil_0}^2, \quad f \in \Hil_0, \label{eq:Rcirc-kernel}
\end{equation}
where $\norm{\dummy}_{\Hil_0}$ is the norm induced by the inner product
\begin{equation}
\ang*{\sum_{i=1}^K a_i \kernel(\dummy, x_i), \sum_{j=1}^J b_j \kernel(\dummy, z_j)}_{\Hil_0} \coloneqq \sum_{i=1}^K\sum_{j=1}^J a_ib_j \kernel(x_i, z_j), \label{eq:H0-ip}
\end{equation}
which immediately implies the reproducing property: For any $f \in \Hil_0$, $\ang{\kernel(x, \dummy), f}_{\Hil_0} = f(x)$ for all $x \in \X$. It is a standard exercise to show that \cref{eq:H0-ip} defines a \textit{bona fide} inner product on $\Hil_0$.\footnote{In particular, this exercise arises when proving the Moore--Aronszajn theorem (cf.~\cite{theory-reproducing-kernels}).} Let $(f_n)_{n\in \N} \subset \Hil_0$ be Cauchy with respect to $\norm{\dummy}_{\Hil_0}$. For every $x \in \X$ and $n, \ell \in \N$, we have that
\begin{equation}
    \abs{f_n(x) - f_\ell(x)} = \abs{\ang{\kernel(x, \dummy), f_n - f_\ell}_{\Hil_0}} \leq \norm{\kernel(x, \dummy)}_{\Hil_0} \norm{f_n - f_\ell}_{\Hil_0} = \sqrt{\kernel(x, x)} \norm{f_n - f_\ell}_{\Hil_0}. \label{eq:H0-Cauchy}
\end{equation}
Thus, we can identify the limit of a Cauchy sequence $(f_n)_{n\in \N} \subset \Hil_0$ by the pointwise limit
\begin{equation}
    f(x) \coloneqq \lim_{n \to \infty} f_n(x), \quad x \in \X, \label{eq:H0-pointwise}
\end{equation}
which is guaranteed to exist since $\R$ is complete. 

Recall that the RKHS $\Hil$ can also be \emph{defined as} the completion of $\Hil_0$ in $\norm{\dummy}_{\Hil_0}$ (see~\cite[Definition~2.9]{ScholkopfKernels}). Thanks to \cref{eq:H0-Cauchy,eq:H0-pointwise}, we thus have that $\Hil$ is given by the collection of all pointwise limits of Cauchy sequences in $\Hil_0$, where the norm $\norm{\dummy}_\Hil$ and inner product $\ang{\dummy, \dummy}$ can be equivalently defined via limits of $\norm{\dummy}_{\Hil_0}$ and $\ang{\dummy, \dummy}_{\Hil_0}$, respectively. An RKHS has a lot of structure with respect to modes of convergence, summarized in the next lemma.

\begin{lemma} \label{lemma:RKHS-convergence}
Let $\Hil$ be an RKHS on $\X$ and let $(f_n)_{n\in\mathbb N}\subset \Hil$ and $f\in \Hil$. Then,
\begin{enumerate}[label=(\roman*)]
\item $f_n$ weakly converges to $f$ if and only if $f_n(x)\to f(x)$ for all $x\in \X$ and $\sup_{n\in\mathbb N}\norm{f_n}_\Hil<+\infty$. \label{item:weak-convergence-H}
\item $f_n$ strongly converges to $f$ if and only if $f_n(x)\to f(x)$ for all $x\in \X$ and $\norm{f_n}_\Hil\to \norm{f}_\Hil$. \label{item:strong-convergence-H}
\end{enumerate}
\end{lemma}
\begin{proof}
Let $\kernel: \X \times \X \to \R$ denote the reproducing kernel of $\Hil$ and
recall that $\Hil_0 = \spn\curly{\kernel(x,\dummy)}_{x \in \X}$ is dense in $\Hil$. By the reproducing property, $f_n(x)=\ang{f_n,\kernel(x, \dummy)}$ and $f(x)=\ang{f,\kernel(x, \dummy)}$ for all $x\in \X$. Hence, if $f_n(x) \to f(x)$ for all $x \in \X$, then $\ang{g, f_n} \to \ang{g, f}$ for every $g\in \Hil_0$ (since $g$ is a finite combination of the kernels $\kernel(x, \dummy)$). \Cref{item:weak-convergence-H} then follows from~\cite[Theorem~8.40]{hunter2001applied}.

If $f_n$ strongly converges to $f$, then $f_n(x) \to f(x)$ for all $x \in \X$ as
\begin{equation}
    \abs{f(x) - f_n(x)} = \abs{\ang{\kernel(x, \dummy), f - f_n}} \leq \norm{\kernel(x, \dummy)}_\Hil \norm{f - f_n}_\Hil = \sqrt{\kernel(x, x)} \norm{f - f_n}_\Hil.
\end{equation}
The fact that $\norm{f_n}_\Hil \to \norm{f}_\Hil$ is immediate by the reverse triangle inequality. Conversely, if $f_n(x) \to f(x)$ for all $x \in \X$ and $\norm{f_n}_\Hil \to \norm{f}_\Hil$, then $\sup_{n\in \N} \norm{f_n}_\Hil$ is bounded. By \cref{item:weak-convergence-H}, this implies that $f_n$ weakly converges to $f$. Combining weak convergence with the convergence of norms $\norm{f_n}_\Hil \to \norm{f}_\Hil$ and the fact that all Hilbert spaces are uniformly convex Banach spaces implies that $f_n$ strongly converges to $f$ by~\cite[Proposition~3.32]{brezis2011functional}. Hence, \ref{item:strong-convergence-H} holds.
\end{proof}

\subsection{Representation Costs and Representer Theorems}
We now establish that, given a kernel, the associated parametric model class (i.e., the set of kernel machines), and the usual parameter cost, the induced extended representation cost and native space coincide with the squared norm in the corresponding RKHS and the RKHS itself, respectively. This shows that kernel methods and RKHSs are a special case of the abstract representation-cost and function-space perspective developed in this paper. Finally, we show that the celebrated RKHS representer theorem is compatible with our abstract framework.

For compatibility with the universal topology introduced in \cref{sec:Lip}, we now restrict our attention to the class of \emph{Lipschitz kernels}. In particular, in the sequel, let $(\X,\rho)$ be a separable metric space and let $\e \in \X$ be some base point.

\begin{definition}
A kernel $\kernel:\X \times \X \to \R$ is said to be \emph{$(L, \rho)$-Lipschitz} if the map $x \mapsto \kernel(x, \dummy)$ is $L$-Lipschitz continuous with respect to the metric $\rho(\dummy, \dummy)$, i.e., 
\begin{equation}
\norm{\kernel(x, \dummy) - \kernel(z, \dummy)}_\Hil  \leq L \, \rho(x,z), \quad \text{for all } x,z\in \X.
\end{equation}
\end{definition}

\begin{lemma}\label{lem:lipschitz_kernels_embed}
If $\kernel: \X \times \X \to \R$ is an $(L, \rho)$-Lipschitz kernel, then its RKHS $\Hil$ continuously embeds into $\Lip(\X)$. In particular, $\norm{f}_{\Lip(\X)} \leq \max\curly{L, \sqrt{\kernel(\e, \e)}}\norm{f}_\Hil$, for all $f \in \Hil$.
\end{lemma}
\begin{proof}
For any $f \in \Hil$ and $x, z \in \X$, we have that
\begin{equation}
    \abs{f(x) - f(z)} = \abs{\ang{\kernel(x, \dummy) - \kernel(z, \dummy), f}} \leq \norm{\kernel(x, \dummy) - \kernel(z, \dummy)}_\Hil \norm{f}_\Hil \leq L \, \norm{f}_\Hil \, \rho(x, z).
\end{equation}
Therefore, $\abs{f}_{\Lip(\X)} \leq L \, \norm{f}_\Hil$. Next,
\begin{equation}
    \abs{f(\e)} = \abs{\ang{f, \kernel(\e, \dummy)}} \leq \norm{f}_\Hil \norm{\kernel(\e, \dummy)}_\Hil = \sqrt{\kernel(\e, \e)} \norm{f}_\Hil.
\end{equation}
Thus, for any $f \in \Hil$, we have $\norm{f}_{\Lip(\X)} = \max\curly{\abs{f(\e)}, \abs{f}_{\Lip(\X)}} \leq \max\curly{L, \sqrt{\kernel(\e, \e)}}\norm{f}_\Hil$.
\end{proof}
\begin{remark} \label{remark:RKHS-Lip-kernels}
    Many common kernels used in machine learning fall under the class of Lipschitz kernels. For example, if $\X \subset \R^d$ with the metric $\rho(x, z) \coloneqq \norm{x - z}_2$, then the Gaussian kernel given by
    \begin{equation}
        \kernel_\mathrm{Gauss}(x, z) \coloneqq \exp\paren*{-\frac{1}{2\sigma^2}\norm{x - z}_2^2}, \quad x, z \in \X, \quad \sigma > 0,
    \end{equation}
    is a Lipschitz kernel.
    On $\X \subset \R^d$, the $\ell^p$-Laplacian kernel, $1 \leq p \leq 2$, given by
    \begin{equation}
        \kernel_\mathrm{Laplace}(x, z) \coloneqq \exp\paren*{-\frac{1}{\sigma}\norm{x - z}_p},\quad x, z \in \X, \quad \sigma > 0,
    \end{equation}
    is also a Lipschitz kernel, but with respect to the metric $\rho(x, z) \coloneqq \norm{x - z}_p^{1/2}$.\footnote{Recall that for any metric $\rho(\dummy, \dummy)$, the bivariate function $\rho(\dummy, \dummy)^q$ for $0 < q \leq 1$ is also a metric. In fact, the two metric topologies coincide.}
\end{remark}

\begin{theorem} \label{thm:rep-cost-kernel-RKHS}
    Let $\kernel: \X \times \X \to \R$ be an $(L, \rho)$-Lipschitz kernel and let $(\Hil, \norm{\dummy}_\Hil)$ be the corresponding RKHS. With the parametric model setup as in \cref{sec:param-RKHS}, if we construct the extended representation cost with $\B = \Lip(\X)$, we have for any $f \in \Lip(\X)$ that
    \begin{equation}
        R(f) = \begin{cases}
            \norm{f}_\Hil^2, & \text{if $f \in \Hil$}, \\
            +\infty, & \text{if $f \in \Lip(\X) \setminus \Hil$}.
        \end{cases}
    \end{equation}
    In this case, the native space is
    \begin{equation}
        \F = \curly{f \in \Lip(\X) \st R(f) < +\infty} = \Hil.
    \end{equation}
\end{theorem}
\begin{proof}
    From \cref{eq:Rcirc-kernel}, the parametric representation cost $\Rcirc: \Lip(\X) \to [0, +\infty]$ is
    \begin{equation}
        \Rcirc(f) = \begin{cases}
            \norm{f}_{\Hil_0}^2 = \norm{f}_\Hil^2 & \text{if $f \in \F_\Theta = \Hil_0$}, \\
            +\infty & \text{if $f \in \Lip(\X) \setminus \Hil_0$}.
        \end{cases}
    \end{equation}
    By \cref{lem:lipschitz_kernels_embed}, $\Rcirc: \Lip(\X) \to [0, +\infty]$ is coercive. Hence, by \cref{thm:rep_cost_lip_spaces}, the extended representation cost is given by the sequential characterization
    \begin{equation}
        R(f) = \min \curly{\alpha \in [0,+\infty] \st (f_n)_{n \in \N} \subset \Hil_0,\ f_n \to f \text{ pointwise},\ \norm{f_n}^2_\Hil \rightarrow \alpha}.
    \end{equation}
    
    Fix $f \in \Hil$ and choose $(f_n)_{n \in \N} \subset \Hil_0$ such that $f_n \to f$ strongly (such a sequence always exists since $\Hil_0$ is dense in $\Hil$). Then, by \cref{lemma:RKHS-convergence}~\ref{item:strong-convergence-H}, $f_n \to f$ pointwise and $\norm{f_n}_\Hil \to \norm{f}_\Hil$. Therefore, $R(f) \leq \norm{f}_\Hil^2$.

    Conversely, suppose that $f \in \F$ (i.e., $R(f) < +\infty$). Then, there exists a sequence $(f_n)_{n \in \N} \subset \Hil_0$ such that $f_n \to f$ pointwise with $R(f) = \lim_{n \to \infty} \Rcirc(f_n) = \lim_{n \to \infty} \norm{f_n}_\Hil^2$. Since $\sup_{n\in\N}\norm{f_n}_\Hil<+\infty$ and $\Hil$ is reflexive,
    there exists a subsequence $(f_{n_j})_{j\in\N}$ and some $h\in \Hil$ such that $f_{n_j}$ weakly converges to $h$. For each $x\in\X$, point evaluation is weakly continuous on $\Hil$, so $f_{n_j}(x)\to h(x)$. Since $f_n\to f$ pointwise, we also have $f_{n_j}(x)\to f(x)$, and hence $h=f$ pointwise on $\X$. Thus $f\in \Hil$. 
    
    Since the Hilbert norm is weakly l.s.c., we have that
    \begin{equation}
    \norm{f}_\Hil^2
    \leq \liminf_{j\to\infty}\norm{f_{n_j}}_\Hil^2
    = \lim_{n\to\infty}\norm{f_n}_\Hil^2
    = R(f).
    \end{equation}
    Thus, $R(f) = \norm{f}_\Hil^2$ for all $f \in \Hil$ and $R(f) = +\infty$ for all $f \in \Lip(\X) \setminus \Hil$, in which case $\F = \Hil$.
\end{proof}

This theorem verifies that the inductive bias of kernel methods is exactly captured by the associated RKHS and that the representation cost is exactly captured by the squared RKHS norm. We can now investigate the celebrated RKHS representer theorem through the lens of our abstract framework.

\begin{theorem}\label{thm:rkhs-representer}
Let $\kernel:\X\times \X\to \R$ be an $(L,\rho)$-Lipschitz kernel with associated RKHS $\Hil$. Fix the distinct data sites $x_1,\dots,x_N\in \X$ and let $G:\R^N\to [0, +\infty]$ be proper, l.s.c., and coercive. For $\lambda>0$, consider the problem
\begin{equation}\label{eq:rkhs-prob}
\min_{f\in \Hil} G\big(f(x_1),\dots,f(x_N)\big) + \lambda \norm{f}_{\Hil}^2.
\end{equation}
Then, if the objective is proper, the solution set is nonempty and every minimizer $f^\star$ of \cref{eq:rkhs-prob} is of the form
\begin{equation}\label{eq:representer}
f^\star=\sum_{i=1}^N a_i \kernel(\dummy,x_i),
\end{equation}
for some $a_1, \ldots, a_N \in \R$. 
If we further suppose that the Gram matrix $K\in\R^{N\times N}$ with $K_{ij}=\kernel(x_i,x_j)$ is invertible, then the objective in \cref{eq:rkhs-prob} is always proper and the problem can be recast as the finite-dimensional optimization problem  (with nonempty solution set)
\begin{equation}\label{eq:finite-dim}
\min_{a\in\R^N} G(Ka) + \lambda a^\T K a,
\end{equation}
in the sense that any solution $a^\star \in \R^N$ to \cref{eq:finite-dim} induces a solution to \cref{eq:rkhs-prob} via the map $a^\star \mapsto \sum_{i=1}^N a^\star_i \kernel(\dummy, x_i)$. In particular, although the parametric model class $\F_\Theta = \Hil_0 = \spn\curly{\kernel(x, \dummy)}_{x \in \X}$ allows for an arbitrary number of kernels and arbitrary kernel centers, it suffices to have $N$ kernels with fixed centers at the data sites $\curly{x_i}_{i=1}^N$.
\end{theorem}
\begin{proof}
    With $\B = \Lip(\X)$, we have by \cref{rem:w*-point-evaluation} that the evaluation operator $\EOp: \Lip(\X) \to \R^N$ given by
    \begin{equation}
        \EOp f = \big(f(x_1), \ldots, f(x_N)\big) \in \R^N
    \end{equation}
    is $w^*$-continuous. There exists at least one minimizer of \cref{eq:rkhs-prob} by \cref{thm:existence_of_min}. The form of the solution \cref{eq:representer} follows from the classical RKHS representer theorem~\cite{deBoorSmoothingSplines,WahbaSmoothingSplines3,scholkopf2001generalized,ScholkopfKernels,WahbaSplineModels,WendlandBook}. In particular, the RKHS representer theorem reveals that it suffices to optimize over $\spn\{\kernel(\dummy, x_i)\}_{i=1}^N$ as opposed to $\spn\{\kernel(\dummy, x)\}_{x \in \X}$.
    
    The reduction to \cref{eq:finite-dim} follows from \cref{cor:parametric-function-space-equivalence}, applied to the kernel-machine parameterization \cref{eq:kernel-machine} with parameter cost \cref{eq:kernel-cost}, where we note that the hypothesis $\EOp(\F_\Theta) = \R^N$ is satisfied thanks to the invertibility of the Gram matrix $K$. In this case, the square root of the representation cost is the RKHS norm, so the hypotheses are verified with $q=2$ and $p=1$ by \cref{prop:parameter-root-properties-transfer}. Thus, the parameter-space problem and the function-space problem have the same infimal value, and minimizers of the parameter-space problem induce minimizers of the function-space problem.
\end{proof}

\section{Wavelets and Besov Spaces} \label{sec:wavelet}
Wavelet methods provide a canonical sequence-space example of the function-space perspective. In contrast to kernel methods and RKHSs (\cref{sec:RKHS}), where the representation cost is quadratic, wavelet methods naturally lead to weighted $\ell^1$-penalties on coefficients across scales. These penalties capture Besov regularity. Besov spaces attracted substantial interest in signal processing and statistics due to their tight connections with wavelets. More specifically, the key finding was that wavelet-thresholding/shrinkage methods automatically adapt to unknown smoothness in a minimax sense, where smoothness was quantified by Besov regularity~\cite{donoho1994ideal,donoho1995adapting,DonohoWaveletShrinkage}. These preliminary results led to the ``wavelet boom'' in the 1990s and early 2000s. There are now standard references on wavelets and their connections to application areas~\cite{daubechies1992ten,Mallat2008WaveletTour,MeyerWaveletsOperators,strang1996wavelets}. Part of the reason is that Besov regularity is able to model real-world signals (at least low-dimensional ones). In this section, we show that wavelet thresholding/shrinkage methods and Besov spaces are compatible with our abstract construction of representation costs and native spaces.

To avoid technicalities related to bounded domains and boundary-adapted wavelets, we work on $\R^d$ and consider a real Meyer wavelet system \cite{LemarieOndelettes}. In particular, Meyer wavelets have infinitely many vanishing moments and form an orthonormal basis of $L^2(\R^d)$. The Besov space $B^s_{p,q}(\R^d)$ is the space of functions on $\R^d$ of smoothness order $s$ in $L^p(\R^d)$, where the second index $q$ provides finer control of regularity. In this section, we focus on Besov spaces with $p = q = 1$ due to their tight connections with sparse wavelet expansions.

\subsection{Parametric Models and the Construction of Besov Spaces} \label{sec:param-wav}

We recall the classical definition of $B^s_{1,1}(\mathbb{R}^d)$ based on finite differences. We refer the reader to the books of DeVore \& Lorentz~\cite{DeVoreLorentz1993Constructive} and Triebel~\cite{Triebel1983Theory} for more details. Fix $s>0$ and a positive integer $r$. For $h\in\mathbb{R}^d$, define the finite-difference operators
\begin{equation}
\Delta_h f(x) \coloneqq f(x+h)-f(x)
\quad\text{and}\quad
\Delta_h^r \coloneqq \underbrace{\Delta_h\circ\cdots\circ\Delta_h}_{r\text{ times}}.
\end{equation}
The $L^1$-modulus of smoothness is given by
\begin{equation}
\omega_r(f,t)_{L^1} \coloneqq \sup_{\|h\|_2\le t}\norm{\Delta_h^r f}_{L^1},
\quad t>0.
\end{equation}
The \emph{inhomogeneous} Besov space $B^s_{1,1}(\mathbb{R}^d)$ is the subspace of (measurable) functions $f: \R^d \to \R$ such that
\begin{equation}
\norm{f}_{B^s_{1,1}} \coloneqq \norm{f}_{L^1(\mathbb{R}^d)} + \int_0^1 t^{-s}\omega_r(f,t)_{L^1} \, \frac{\dsym t}{t} < +\infty,
\end{equation}
where $r > s$ is any integer (any such choice would yield an equivalent norm~\cite{DeVoreLorentz1993Constructive}). In particular $(B^s_{1,1}(\R^d), \norm{\dummy}_{B^s_{1,1}})$ is a Banach space.

Crucially, the wavelet boom hinged on the fact that Besov norms have equivalent characterizations via \emph{sequence-space norms on wavelet coefficients}~\cite[Chapter~6]{MeyerWaveletsOperators}. Let $\phi$ be the Meyer scaling function on $\R$, and let $\psi$ be the associated Meyer
wavelet. To construct the $d$-dimensional tensor-product system, define
\begin{equation}
\eta_0 \coloneqq \phi \quad\text{and}\quad \eta_1 \coloneqq\psi.
\end{equation}
For each binary multi-index $\varepsilon=(\varepsilon_1,\dots,\varepsilon_d)\in\{0,1\}^d$, define
\begin{equation}
\psi^\varepsilon(x) \coloneqq \prod_{i=1}^d \eta_{\varepsilon_i}(x_i),
\quad x=(x_1,\ldots,x_d)\in\R^d.
\end{equation}
Thus, $\psi^{(0,\dots,0)} \eqqcolon \phi^{(d)}$ is the $d$-dimensional tensor-product scaling function and the remaining $2^d-1$ functions $\curly{\psi^\varepsilon \st \varepsilon\in\{0,1\}^d\setminus\{0\}}$ are the associated tensor-product wavelets. For $k\in\Z^d$, $j\in\mathbb{N}_0$, and $\varepsilon\in\{0,1\}^d\setminus\{0\}$, define
\begin{equation}
\phi_k(x) \coloneqq \phi^{(d)}(x-k)
\quad\text{and}\quad
\psi^\varepsilon_{j,k}(x) \coloneqq 2^{jd/2}\psi^\varepsilon(2^j x-k).
\end{equation}
Furthermore, let
\begin{equation}
\Lambda_0 \coloneqq \Z^d,
\quad
\Lambda_+ \coloneqq \mathbb{N}_0\times\Z^d\times\paren[\big]{\{0,1\}^d\setminus\{0\}},
\quad\text{and}\quad
\Lambda \coloneqq \Lambda_0\cup\Lambda_+.
\end{equation}
The $d$-dimensional Meyer wavelet system is then the union of the scaling functions and the wavelet functions: $\curly{\phi_k}_{k \in \Lambda_0} \cup \curly{\psi^\varepsilon_{j,k}}_{(j, k, \varepsilon) \in \Lambda_+}$.

Next, define the weighted wavelet atoms by
\begin{equation}
\kernel_\lambda \coloneqq \begin{cases}
\phi_k, & \lambda=k\in\Lambda_0,\\
2^{-j(s-d/2)}\psi^\varepsilon_{j,k},
& \lambda=(j,k,\varepsilon)\in\Lambda_+.
\end{cases} \label{eq:weighted-Meyer-wavelets}
\end{equation}
Since $\phi, \psi \in \Sch(\R)$ (the Schwartz space of smooth and rapidly decreasing functions), we have for every $a=(a_\lambda)_{\lambda\in\Lambda}\in\ell^1(\Lambda)$ the series $\sum_{\lambda\in\Lambda} a_\lambda \kernel_\lambda$ converges in the space of tempered distributions $\Sch'(\R^d)$. Define the operator
\begin{equation}
\TOp:\ell^1(\Lambda)\to\Sch'(\R^d): a \mapsto \sum_{\lambda\in\Lambda} a_\lambda \kernel_\lambda \label{eq:wavelet-synthesis}
\end{equation}
and define the set
\begin{equation}
\U \coloneqq \curly{\TOp(a) \st a \in \ell^1(\Lambda)}.
\end{equation}
By construction, $\TOp:\ell^1(\Lambda)\to\U$ is linear and surjective. It is also injective. Indeed, if $\TOp a=0$ in $\Sch'(\R^d)$, pair the distributional series with any Meyer scaling function or wavelet. Since these test functions belong to $\Sch(\R^d)$ and the unweighted Meyer system is orthonormal in $L^2(\R^d)$, the pairing isolates the corresponding coefficient multiplied by the nonzero weight in \cref{eq:weighted-Meyer-wavelets}. Hence every coefficient of $a$ vanishes. Thus, there exists an inverse $\TOp^{-1}:\U\to\ell^1(\Lambda)$. Hence, if we endow $\U$ with the norm
\begin{equation}
    \norm{f}_\U = \norm{\TOp^{-1} f}_{\ell^1}, \quad f \in \U,
\end{equation}
we have that $\ell^1(\Lambda)$ and $\U$ are isometrically isomorphic and, in particular, $\U$ is a Banach space. Furthermore, every $f \in \U$ admits a \emph{wavelet series representation}
\begin{equation}
    f = \sum_{\lambda \in \Lambda} a_\lambda \kernel_\lambda, \quad a = \TOp^{-1} f.
\end{equation}

From~\cite[Chapter~6]{MeyerWaveletsOperators}, every $f \in B^s_{1,1}(\R^d)$ admits a wavelet series representation of the form
\begin{equation}
f = \sum_{k\in\Z^d} c_k \phi_k + \sum_{\varepsilon\in\{0,1\}^d\setminus\{0\}} \sum_{j \in \mathbb{N}_0} \sum_{k\in\Z^d} d^\varepsilon_{j,k}\psi^\varepsilon_{j,k}
\end{equation}
with the property that
\begin{equation}
\norm{f}_{B^s_{1,1}} \asymp \norm{(c_k)_{k\in\Z^d}}_{\ell^1} + \sum_{\varepsilon\in\{0,1\}^d\setminus\{0\}}
\sum_{j \in \mathbb{N}_0} 2^{j(s-d/2)} \norm{(d^\varepsilon_{j,k})_{k\in\Z^d}}_{\ell^1}.
\end{equation}
This implies that $\U = B^s_{1,1}(\R^d)$ (as sets) and that $\norm{\dummy}_\U$ is an equivalent norm to $\norm{\dummy}_{B^s_{1,1}}$.

This motivates the wavelet parametric model. Let
\begin{equation}
\Theta \coloneqq \{(a_1,\lambda_1),\ldots,(a_K,\lambda_K)\st a_j \in \R, \lambda_j \in \Lambda, K\in \mathbb{N}\}
\end{equation}
denote the parameter space and for a parameter $\theta = ((a_1,\lambda_1),\ldots,(a_K,\lambda_K))\in\Theta$ define the parametric model
\begin{equation}
f_{\theta} = \sum_{j=1}^K a_j \kernel_{\lambda_j}. \label{eq:wavelet-model}
\end{equation}
Note that due to the nature of the definition of $\Theta$, \cref{eq:wavelet-model} is a \emph{wavelet series} with an \emph{arbitrary} number of wavelets $K \in \mathbb{N}$ and \emph{arbitrary} wavelet parameters $\{\lambda_j\}_{j=1}^K$. Thus, we see that the parametric model class $\F_\Theta$ can be written as
\begin{equation}
    \F_\Theta = \curly*{ \sum_{\lambda \in \Lambda} a_\lambda \kernel_\lambda \st a \in c_{00}(\Lambda)},
\end{equation}
where $c_{00}(\Lambda)$ is the space of finitely supported sequences on $\Lambda$. The parameter cost that corresponds to wavelet thresholding/shrinkage is
\begin{equation}
    C(\theta) = \sum_{j=1}^K \abs{a_j}, \quad \theta \in \Theta,
\end{equation}
which is the $\ell^1$-norm on the coefficients of the (weighted) wavelet system atoms.

\subsection{Representation Costs and Representer Theorems}

We now establish that, given the weighted Meyer wavelet system, the associated parametric model class, and the $\ell^1$-norm parameter cost, the induced extended representation cost and native space coincide with a norm on a Besov space and the
Besov space itself, respectively. This shows that wavelet-thresholding methods and Besov spaces are
a special case of the abstract representation-cost and function-space perspective developed in this
paper. Finally, we show that sparse representer theorems for wavelet thresholding/shrinkage are compatible
with our abstract framework.

For compatibility with the universal topology introduced in \cref{sec:Lip}, we restrict attention to
the smoothness regime
\begin{equation}
s>d+1. \label{eq:Besov-Lip-embedding}
\end{equation}
This ensures that the weighted wavelet atoms are uniformly Lipschitz and, in particular, $B^s_{1,1}(\R^d)$ embeds continuously into $\Lip(\R^d)$, where we take the base point $\e = 0 \in \R^d$ and endow $\R^d$ with the Euclidean metric.

\begin{lemma} \label{lemma:Besov-Lip-embed}
Let $s>d+1$. Then,
$\norm{f}_{\Lip(\mathbb{R}^d)}\lesssim \norm{f}_{B_{1,1}^s}$, for all $f \in B^s_{1,1}(\R^d)$.
\end{lemma}

\begin{proof}
Every $f\in B^s_{1,1}(\R^d)$ admits a wavelet series representation
\begin{equation}
f=\sum_{\lambda\in\Lambda} a_\lambda \kernel_\lambda, \quad (a_\lambda)_{\lambda\in\Lambda}\in \ell^1(\Lambda), \label{eq:wavelet-rep-lemma}
\end{equation}
with $\norm{f}_\U = \norm{a}_{\ell^1}$ as an equivalent norm to $\norm{f}_{B^s_{1,1}}$. For $\lambda=(j,k,\varepsilon)\in\Lambda_+$, we have that
\begin{equation}
\norm{\kernel_\lambda}_{L^\infty}
\lesssim
2^{-j(s-d)}
\quad\text{and}\quad
\abs{\kernel_\lambda}_{\Lip(\mathbb{R}^d)}
\lesssim
2^{-j(s-d-1)}.
\end{equation}
Next, the scaling
functions satisfy
\begin{equation}
\sup_{k\in\mathbb{Z}^d}\norm{\phi_k}_{\Lip(\mathbb{R}^d)}<+\infty.
\end{equation}
Hence,
\begin{equation}
C_s \coloneqq \sup_{\lambda\in\Lambda}\norm{\kernel_\lambda}_{\Lip(\mathbb{R}^d)}<+\infty.
\end{equation}
For $f \in B^s_{1,1}(\R^d)$ with the representation \cref{eq:wavelet-rep-lemma}, we have that
\begin{equation}
\norm{f}_{\Lip(\R^d)}
\leq
\sum_{\lambda\in\Lambda}|a_\lambda|\,\|\kernel_\lambda\|_{\Lip(\mathbb{R}^d)}
\leq
C_s \|(a_\lambda)_{\lambda\in\Lambda}\|_{\ell^1} \lesssim \norm{f}_{B^s_{1,1}},
\end{equation}
which completes the proof.
\end{proof}

\begin{theorem}
Let $s > d+1$. With the parametric model setup as in \cref{sec:param-wav}, if we construct the extended representation cost with $\B = \Lip(\R^d)$, we have for any $f \in \Lip(\R^d)$ that
\begin{equation}
    R(f) = \begin{cases}
        \norm{f}_\U, & \text{if $f \in B^s_{1,1}(\R^d)$}, \\
        +\infty, & \text{if $f \in \Lip(\R^d) \setminus B^s_{1,1}(\R^d)$}.
    \end{cases}
\end{equation}
In this case, the native space is
\begin{equation}
    \F = \curly{f \in \Lip(\R^d) \st R(f) < +\infty} = B^s_{1,1}(\R^d)
\end{equation}
and
\begin{equation}
    R(f) \asymp \norm{f}_{B^s_{1,1}}, \quad f \in \F.
\end{equation}
\end{theorem}

\begin{proof}
From \cref{sec:param-wav}, the parametric representation cost is
\begin{equation}
\Rcirc(f)=
\begin{cases}
\norm{\TOp^{-1} f}_{\ell^1} = \norm{f}_{\U} & \text{if } f\in \F_\Theta, \\
+\infty & \text{if } f\in \Lip(\R^d) \setminus \F_\Theta,
\end{cases}
\end{equation}
where
\begin{equation}
    \F_\Theta = \curly*{ \sum_{\lambda \in \Lambda} a_\lambda \kernel_\lambda \st a \in c_{00}(\Lambda)}.
\end{equation}
By \cref{lemma:Besov-Lip-embed}, $\Rcirc:\Lip(\R^d)\to [0,+\infty]$ is coercive. Hence, by \cref{thm:rep_cost_lip_spaces}, the extended representation cost is given by the sequential characterization
\begin{equation}
    R(f) = \min \curly{\alpha \in [0,+\infty] \st (f_n)_{n \in \N} \subset \F_\Theta,\      f_n \to f \text{ pointwise},\ \norm{f_n}_\U \rightarrow \alpha}.
\end{equation}

Fix $f\in B^s_{1,1}(\R^d) = \U$ and set $a \coloneqq \TOp^{-1}f \in \ell^1(\Lambda)$ so that $\norm{f}_\U = \norm{a}_{\ell^1}$. Since $c_{00}(\Lambda)$ is dense in $\ell^1(\Lambda)$, there exists a sequence
$(a^{(n)})_{n\in\mathbb{N}}\subset c_{00}(\Lambda)$ such that $\norm{a^{(n)} - a}_{\ell^1} \to 0$. Define $f_n \coloneqq \TOp a^{(n)}$. Since $\U$ is isometrically isomorphic to $\ell^1(\Lambda)$, we have that
\begin{equation}
    \norm{f_n - f}_\U = \norm{a^{(n)} - a}_{\ell^1} \to 0. \label{eq:U-ell1-converge}
\end{equation}
By \cref{lemma:Besov-Lip-embed}, we have that $\norm{f_n - f}_{\Lip(\R^d)} \lesssim \norm{f_n - f}_\U \to 0$. Moreover, for any $x \in \R^d$, we have that
\begin{equation}
    \abs{f_n(x) - f(x)} \leq \abs{f_n(0) - f(0)} + \abs{f_n - f}_{\Lip(\R^d)} \norm{x}_2 \leq (1 + \norm{x}_2) \norm{f_n - f}_{\Lip(\R^d)}.
\end{equation}
Hence, $f_n(x) \to f(x)$ for all $x \in \R^d$. The reverse triangle inequality on \cref{eq:U-ell1-converge} yields that $\norm{f_n}_\U \to \norm{f}_\U$. Thus, $R(f) \leq \lim_{n \to \infty} \norm{f_n}_\U = \norm{f}_\U$.

Conversely, suppose that $f \in \F$ (i.e., $R(f) < +\infty$). Then, by the sequential characterization, there exists a sequence $(f_n)_{n \in \N} \subset \F_\Theta$ such that $f_n \to f$ pointwise and $\norm{f_n}_\U \to R(f)$. Write $a^{(n)} \coloneqq \TOp^{-1} f_n$, so that $a^{(n)} \in c_{00}(\Lambda)$. Then, $\sup_{n \in \N} \norm{a^{(n)}}_{\ell^1} < +\infty$. Let $c_0(\Lambda)$ denote the space of sequences tending to $0$ on $\Lambda$ endowed with the $\ell^\infty(\Lambda)$-norm.
Since $\Lambda$ is countable, $c_0(\Lambda)$ is separable. Hence closed and bounded subsets of
$\ell^1(\Lambda)=c_0(\Lambda)'$ are weak$^*$ compact and metrizable. The Banach--Alaoglu theorem then guarantees the existence of a subsequence (which we do not relabel) such that $a^{(n)}$ weak$^*$ converges to $a$ for some $a \in \ell^1(\Lambda)$. For every $x\in\R^d$, we have $(\kernel_\lambda(x))_{\lambda\in\Lambda}\in c_0(\Lambda)$. Indeed, for $\lambda=(j,k,\varepsilon)\in\Lambda_+$,
\begin{equation}
    \sup_{k,\varepsilon}|\kernel_{j,k,\varepsilon}(x)|
    \leq \sup_{k,\varepsilon}\|\kernel_{j,k,\varepsilon}\|_{L^\infty}
    \lesssim 2^{-j(s-d)}\to 0
\end{equation}
as $j\to\infty$. For the finitely many remaining scales, and for the scaling functions, Schwartz decay implies convergence to zero as $\|k\|_2\to\infty$. Thus,
\begin{equation}
    f_n(x) = \sum_{\lambda \in \Lambda} a_\lambda^{(n)} \kernel_\lambda(x) \to \sum_{\lambda \in \Lambda} a_\lambda \kernel_\lambda(x) = (\TOp a)(x).
\end{equation}
Since $f_n \to f$ pointwise, we immediately see that $f = \TOp a \in \U$. Since all dual Banach norms are weak$^*$ l.s.c., we have that
\begin{equation}
    \norm{f}_\U = \norm{a}_{\ell^1} \leq \liminf_{n \to \infty} \norm{a^{(n)}}_{\ell^1} = \liminf_{n \to \infty} \norm{f_n}_\U = \lim_{n \to \infty} \norm{f_n}_\U = R(f).
\end{equation}
Thus, $R(f) = \norm{f}_\U \asymp \norm{f}_{B^s_{1,1}}$ for all $f \in B^s_{1,1}(\R^d) = \U$ and $R(f) = +\infty$ for all $f \in \Lip(\R^d) \setminus B^s_{1,1}(\R^d)$, in which case $\F = B^s_{1,1}(\R^d)$.
\end{proof}

This theorem verifies that the inductive bias of weighted wavelet thresholding/shrinkage methods is exactly captured by the associated Besov space and that the representation cost is exactly captured by the associated $\ell^1$-wavelet norm. We can now investigate representer theorems for Besov regularization through the lens of our abstract framework.

\begin{lemma} \label{lemma:wavelet-interpolation}
Let $s>d+1$, and let $x_1,\ldots,x_N\in\R^d$ be distinct. Then $\EOp(\F_\Theta)=\R^N$ where $\EOp f \coloneqq \bigl(f(x_1),\ldots,f(x_N)\bigr)$.
\end{lemma}
\begin{proof}
Let $C_c^\infty(\R^d)$ denote the space of smooth and compactly supported functions on $\R^d$. Since the data sites are distinct, for each $i=1,\ldots,N$ there exists $h_i\in C_c^\infty(\R^d)$ such that $h_i(x_m)=\delta_{im}$ (Kronecker delta), for $m=1,\ldots,N$. Moreover, $C_c^\infty(\R^d)\subset B^s_{1,1}(\R^d)$. By the wavelet characterization of $B^s_{1,1}(\R^d)$ and the density of $c_{00}(\Lambda)$ in $\ell^1(\Lambda)$, the finite wavelet sums in $\F_\Theta$ are dense in $B^s_{1,1}(\R^d)$. Since $s>d+1$, the embedding $B^s_{1,1}(\R^d)\hookrightarrow \Lip(\R^d)$ is continuous, and therefore $\EOp:\F\to\R^N$ is continuous. It follows that each canonical basis vector $\e_i\in\R^N$ belongs to the closure of $\EOp(\F_\Theta)$, because $h_i$ can be approximated in $B^s_{1,1}(\R^d)$ by finite wavelet sums. Hence $\EOp(\F_\Theta)$ is dense in $\R^N$. But $\EOp(\F_\Theta)$ is a subspace of the finite-dimensional space $\R^N$, and is therefore closed. Consequently, $\EOp(\F_\Theta)=\R^N$.
\end{proof}

\begin{theorem}
Let $s>d+1$ and let $\kernel_\lambda$, $\lambda \in \Lambda$ denote the weighted Meyer wavelet atoms \cref{eq:weighted-Meyer-wavelets}. Fix distinct data
sites $x_1,\dots,x_N\in\R^d$ and let $G:\R^N\to [0,+\infty]$ be proper, l.s.c., and
coercive. For $\lambda>0$, consider the problem
\begin{equation}
\min_{f\in B^s_{1,1}(\R^d)}
G\paren[\big]{f(x_1),\dots,f(x_N)}+\lambda \norm{f}_\U. \label{eq:Besov-prob}
\end{equation}
Then, the solution set is nonempty and there exists a minimizer $f^\star$ of \cref{eq:Besov-prob} of the form
\begin{equation}
f^\star=\sum_{j=1}^K a_j \kernel_{\lambda_j},
\quad
K\leq N, \label{eq:wavelet-form}
\end{equation}
for some $a_1,\dots,a_K\in\R$ and (adaptive) wavelet parameters
$\lambda_1,\dots,\lambda_K\in\Lambda$. The problem in \cref{eq:Besov-prob} can be recast as the finite-atom optimization problem (with nonempty solution set)
\begin{equation} \label{eq:wavelet-finite-dim}
\min_{\substack{a_1,\ldots,a_K\in\R\\ \lambda_1,\dots,\lambda_K\in\Lambda}}
G \paren*{\sum_{j=1}^K a_j\kernel_{\lambda_j}(x_1),\ldots, \sum_{j=1}^K a_j\kernel_{\lambda_j}(x_N)} + \lambda \sum_{j=1}^K \abs{a_j}, \quad K \geq N,
\end{equation}
in the sense that any solution $(a^\star, \lambda^\star) \in \R^K \times \Lambda^K$ induces a solution to \cref{eq:Besov-prob} via the map  $(a^\star, \lambda^\star) \mapsto \sum_{j=1}^K a_j^\star \kernel_{\lambda_j^\star}$. In
particular, although the parametric model class $\F_\Theta$ allows for an arbitrary number of wavelet atoms
at arbitrary locations and scales, it suffices to have at most $N$ active wavelet atoms.
\end{theorem}

\begin{proof}
With $\B = \Lip(\R^d)$, we have by \cref{rem:w*-point-evaluation} that the evaluation operator $\EOp: \Lip(\R^d) \to \R^N$ given by
\begin{equation}
    \EOp f = \big(f(x_1), \ldots, f(x_N)\big) \in \R^N
\end{equation}
is $w^*$-continuous. Furthermore, by \cref{lemma:wavelet-interpolation}, the objective in \cref{eq:Besov-prob} is proper. There exists at least one minimizer of \cref{eq:Besov-prob} by \cref{thm:existence_of_min}. Let $\bar{f} \in B^s_{1,1}(\R^d)$ be any minimizer of \cref{eq:Besov-prob} and write $\bar{a} \coloneqq \TOp^{-1} \bar{f} \in \ell^1(\Lambda)$. Then,
\begin{equation}
     \bar{a} \in \argmin_{a \in \ell^1(\Lambda)} \norm{a}_{\ell^1} \quad\subj\quad z_i = \sum_{\lambda \in \Lambda} a_\lambda \kernel_\lambda(x_i), \quad i = 1, \ldots, N, \label{eq:ell1-norm-prob}
\end{equation}
where $z_i = \bar{f}(x_i)$, $i = 1, \ldots, N$. Indeed, if this were not the case, it would contradict the optimality of $\bar{f}$. In particular, any solution $\tilde{a} \in \ell^1(\Lambda)$ to \cref{eq:ell1-norm-prob} induces a solution to \cref{eq:Besov-prob} via the map $\tilde{a} \mapsto \TOp \tilde{a}$. Since $(\kernel_\lambda(x_i))_{\lambda \in \Lambda} \in c_0(\Lambda)$, $i = 1, \ldots, N$, the problem \cref{eq:ell1-norm-prob} is a special case of the classical $\ell^1$-norm minimization problem with finitely many weak$^*$ continuous linear constraints. This problem has a representer theorem~\cite{BoyerRepresenter,BrediesRepresenter,unser2016representer} that implies that there exists a solution of the form $a^\star \in \ell^1(\Lambda)$ such that $\abs{\spt(a^\star)} \leq N$. Thus, the sequence $a^\star$ induces a solution to \cref{eq:Besov-prob} of the form \cref{eq:wavelet-form}.

The reduction to \cref{eq:wavelet-finite-dim} follows from \cref{cor:parametric-function-space-equivalence}, applied to the wavelet dictionary parameterization $(a_j,\lambda_j)_{j=1}^K \mapsto \sum_{j=1}^K a_j\kappa_{\lambda_j}$ with parameter cost $\sum_{j=1}^K\abs{a_j}$, where we note that the hypothesis $\EOp(\F_\Theta) = \R^N$ is satisfied by \cref{lemma:wavelet-interpolation} since the $x_1, \ldots, x_N$ are distinct. The hypotheses are verified with $q=p=1$ by \cref{prop:parameter-root-properties-transfer}. Thus, the parameter-space problem and the function-space problem have the same infimal value, and minimizers of the parameter-space problem induce minimizers of the function-space problem.
\end{proof}

\begin{remark}
In the wavelet-coefficient domain, the regularizer in \cref{eq:Besov-prob} is a weighted $\ell^1$-penalty. In iterative methods, the associated proximal step is coefficient-wise soft-thresholding. Equivalently, if one solves the corresponding penalized problem by a proximal or forward--backward method, then each wavelet coefficient is updated by a shrinkage rule with threshold proportional to its weight; in particular, because the weights depend on the scale $j$, this yields scale-dependent thresholding. This is precisely what underlies classical wavelet shrinkage methods for Besov-regularized estimation~\cite{donoho1994ideal,donoho1995adapting,DonohoWaveletShrinkage}. See also~\cite{DaubechiesDefriseDeMol2004} for iterative thresholding algorithms with weighted
coefficient penalties, and~\cite[Section~6.5]{ParikhBoyd2014} for the proximal-operator perspective.
\end{remark}

\section{Sparse Kernel Methods and RKBSs Parameterized by Measures} \label{sec:I-RKBS}
Compared to Hilbert spaces, Banach spaces possess much richer \emph{geometric} structure. This can be seen, for example, from the fact that any two Hilbert spaces of the same dimension are isometrically isomorphic. Thus, once the dimension is fixed, there is essentially only one Hilbert space. This line of work led to the development of RKBSs~\cite{lin2022reproducing,zhang2009reproducing}. Furthermore, the study of learning methods over Banach spaces isometrically isomorphic to the sequence space $\ell^1(\Z)$ or the space of Radon measures $\M(\Xi)$, where $\Xi$ is a Hausdorff space, gained significant popularity due to the \emph{sparsity-inducing} effect of these norms~\cite{bach2025optimization,wang2024sparse}. Such norms also have tight connections to the representation costs of shallow neural networks, as we shall see in \cref{sec:shallow}. In this section, we show that \emph{sparse} kernel methods and RKBSs parameterized by Radon measures are compatible with our abstract construction of representation costs and native spaces. We focus on real RKBSs defined on a set $\X$.

\begin{definition}
    A real Banach space $(\U, \norm{\dummy}_\U)$ of functions $\X \to \R$ is said to be a \emph{reproducing kernel Banach space} (RKBS) if point evaluations are bounded, i.e., for every $x \in \X$, there exists a constant $C_x > 0$ such that
    \begin{equation}
        \abs{f(x)} \leq C_x \norm{f}_\U, \quad f \in \U.
    \end{equation}
\end{definition}

In this section, we focus exclusively on real RKBSs parameterized by Radon measures (cf.~\cite{BartolucciRKBS,lin2022reproducing,spek2025duality,wang2024sparse}).  Let $\Xi$ be a \emph{locally compact} Hausdorff space. Then, these are real RKBSs of functions $\X \to \R$ that are isometrically isomorphic to $\M(\Xi)$, where the isometry is specified by an integral operator with a kernel $\kernel: \X \times \Xi \to \R$. Such RKBSs are called \emph{integral RKBSs} (I-RKBSs). Crucial to this setup is the \emph{Riesz--Markov--Kakutani representation theorem}, which states that the topological dual of the Banach space of continuous functions (vanishing at infinity if $\Xi$ is not compact), denoted by $C_0(\Xi)$ (endowed with the $L^\infty$-norm) is the space of Radon measures $\M(\Xi)$. In particular, the \emph{variation norm} of $\nu \in \M(\Xi)$ is
\begin{equation}
    \norm{\nu}_\M = \sup_{\substack{\varphi \in C_0(\Xi) \\ \norm{\varphi}_{L^\infty} = 1}} \abs{\ang{\varphi, \nu}} = \sup_{\substack{\varphi \in C_0(\Xi) \\ \norm{\varphi}_{L^\infty} = 1}} \abs*{\int_{\Xi} \varphi(\xi) \dd\nu(\xi)}. 
\end{equation}

\begin{definition} \label{defn:I-RKBS-kernel}
    An \emph{I-RKBS kernel} is a function $\kernel: \X \times \Xi \to \R$ such that $\kernel(x, \dummy) \in C_0(\Xi)$ for all $x \in \X$ and $\spn\{\kernel(x, \dummy)\}_{x \in \X}$ is dense in $C_0(\Xi)$, i.e., the $L^\infty$-closure of the span is $C_0(\Xi)$.
\end{definition}

Kernels that satisfy this density assumption are related to the concept of \emph{universal kernels}~\cite{micchelli2006universal}, with the key twist that $\kernel(x, \dummy)$ is a function defined on $\Xi$ and not $\X$. The next lemma provides ways to check if a kernel is an I-RKBS kernel and bears resemblance to the injectivity conditions for the embedding of Radon measures into Hilbert~\cite{sriperumbudur2011universality} and Banach~\cite{fukumizu2011learning} spaces.

\begin{lemma} \label{lemma:I-RKBS}
    Let $\X$ be a set and $\Xi$ be a locally compact Hausdorff space. Let $\kernel: \X \times \Xi \to \R$ be a function such that $\kernel(x, \dummy) \in C_0(\Xi)$ for every $x \in \X$ and define the integral operator
    \begin{equation}
        \TOp_\kernel: \M(\Xi) \to \R^\X: \nu \mapsto \int_{\Xi} \kernel(\dummy, \xi) \dd\nu(\xi). \label{eq:integral-op}
    \end{equation}
    Then, the following are equivalent:
    \begin{enumerate}[label=(\roman*)]
        \item $\kernel: \X \times \Xi \to \R$ is an I-RKBS kernel. \label{item:I-RKBS-kernel}
        \item $\spn\{\kernel(x, \dummy)\}_{x \in \X}$ is dense in $C_0(\Xi)$. \label{item:dense-C0}
        \item $\TOp_\kernel$ is an injection. \label{item:integral-op-injective}
    \end{enumerate}
\end{lemma}
\begin{proof}
    \ref{item:I-RKBS-kernel} $\Leftrightarrow$ \ref{item:dense-C0} is by definition. Given $\nu \in \M(\Xi)$, define
    \begin{equation}
        \Lambda_\nu: C_0(\Xi) \to \R: \varphi \mapsto \int_{\Xi} \varphi(\xi) \dd\nu(\xi).
    \end{equation}
    Then, the annihilator of $V \coloneqq \spn\{\kernel(x, \dummy)\}_{x \in \X}$ with respect to the dual pair $(C_0(\Xi), \M(\Xi))$ is
    \begin{equation}
        V^\perp = \curly{\nu \in \M(\Xi) \st \Lambda_\nu(\varphi) = 0 \text{ for all } \varphi \in V}.
    \end{equation}
    Recall that $V$ is dense in $C_0(\Xi)$ if and only if $V^\perp = \{0\}$. Since $V$ is the linear span of $\{\kernel(x, \dummy)\}_{x \in \X}$, we have that $V^\perp = \ker \TOp_\kernel$. Hence, we have the chain of equivalences
    \begin{equation}
        V \text{ is dense in } C_0(\Xi) \quad\Leftrightarrow\quad \ker \TOp_\kernel = \{0\} \quad\Leftrightarrow\quad \TOp_\kernel \text{ is injective},
    \end{equation}
    in which case \ref{item:dense-C0} $\Leftrightarrow$ \ref{item:integral-op-injective}, which completes the proof.
\end{proof}

Given an I-RKBS kernel $\kernel: \X \times \Xi \to \R$, define the set
\begin{equation}
    \U \coloneqq \curly{\TOp_\kernel(\nu): \nu \in \M(\Xi)},
\end{equation}
with $\TOp_\kernel$ defined as in \cref{eq:integral-op}. By construction, we have that $\TOp_\kernel: \M(\Xi) \to \U$ is linear and surjective. By \cref{lemma:I-RKBS}~\ref{item:integral-op-injective}, it is also injective. Thus, there exists an inverse $\TOp_\kernel^{-1}: \U \to \M(\Xi)$. Hence, if we endow $\U$ with the norm
\begin{equation} \label{eq:UUUUUU}
    \norm{f}_\U = \norm{\TOp_\kernel^{-1} f}_\M, \quad f \in \U,
\end{equation}
we have that $\M(\Xi)$ and $\U$ are isometrically isomorphic and, in particular, $\U$ is a Banach space. Going one step further, we see that every $f \in \U$ admits an \emph{integral representation}
\begin{equation} \label{eq:integral-rep}
    f = \int_\Xi \kernel(\dummy, \xi) \dd\nu(\xi), \quad \nu = \TOp_\kernel^{-1}f,
\end{equation}
and $\U$ is an RKBS since
\begin{equation}
    \abs{f(x)} = \abs{\ang{\kernel(x, \dummy), \TOp_\kernel^{-1} f}} \leq \norm{\kernel(x, \dummy)}_{L^\infty} \norm{\TOp_\kernel^{-1}f}_\M = \norm{\kernel(x, \dummy)}_{L^\infty} \norm{f}_\U. \label{eq:RKBS-CS-proof}
\end{equation}
This motivates the definition of an I-RKBS.
\begin{definition}
    A real Banach space $(\U, \norm{\dummy}_\U)$ of functions $\X \to \R$ is called an \emph{integral reproducing kernel Banach space} (I-RKBS) if there exists an I-RKBS kernel $\kernel:\X \times \Xi \to \R$ such that
    \begin{equation}
        \U = \curly*{\int_\Xi \kernel(\dummy, \xi) \dd\nu(\xi): \nu \in \M(\Xi)} \quad\text{and}\quad \norm*{\int_\Xi \kernel(\dummy, \xi) \dd\nu(\xi)}_\U = \norm{\nu}_\M.
    \end{equation}
\end{definition}
\begin{remark}
    Many works that study I-RKBSs define the norm as an infimum over all representations (see~\cite{BachConvexNN,BartolucciRKBS,spek2025duality}), i.e., they consider the norm
    \begin{equation}
        \norm*{f}_\U = \inf_{\nu \in \M(\Xi)} \norm{\nu}_\M \quad\subj\quad f = \int_\Xi \kernel(\dummy, \xi) \dd\nu(\xi) = \TOp_\kernel\nu.
    \end{equation}
    We do not need an infimum thanks to the density assumption in \cref{defn:I-RKBS-kernel}, which consequently ensures that $\TOp_\kernel: \M(\Xi) \to \U$ is a bijection.
\end{remark}
Parallel to RKHS theory, every I-RKBS $\U$ on $\X$ with kernel $\kernel: \X \times \Xi \to \R$ satisfies a \emph{reproducing property}. If we define the pairing $[\dummy, \dummy]: C_0(\Xi) \times \U \to \R$ by
\begin{equation}
    [\varphi, f] \coloneqq \ang{\varphi, \TOp_\kernel^{-1} f}, \quad (\varphi, f) \in C_0(\Xi) \times \U, \label{eq:RKBS-pairing}
\end{equation}
where $\ang{\dummy, \dummy}$ is the canonical pairing between $C_0(\Xi)$ and $\M(\Xi)$, we have the following reproducing property: For every $x \in \X$,
\begin{equation} \label{eq:I-RKBS-reproducing}
    [\kernel(x, \dummy), f] = f(x), \quad f \in \U,
\end{equation}
and, by \cref{eq:UUUUUU}, we have the Cauchy--Schwarz-type inequality
\begin{equation}
    \abs{[\kernel(x, \dummy), f]} \leq \norm{\kernel(x, \dummy)}_{L^\infty} \norm{f}_\U.  \label{eq:RKBS-CS}
\end{equation}

\subsection{Parametric Models and the Construction of I-RKBSs} \label{sec:RKBS-param}
Fix an I-RKBS kernel $\kernel: \X \times \Xi \to \R$ and let $\U$ denote the corresponding I-RKBS with norm $\norm{\dummy}_\U$. Let
\begin{equation}
\Theta \coloneqq \{(a_1,\xi_1),\ldots,(a_K,\xi_K)\st a_j \in \R, \xi_j \in \Xi, K\in \mathbb{N}\}
\end{equation}
denote the parameter space and for a parameter $\theta = ((a_1,\xi_1),\ldots,(a_K,\xi_K))\in\Theta$ define the parametric model
\begin{equation}
f_{\theta} = \sum_{j=1}^K a_j \kernel(\dummy,\xi_j). \label{eq:Banach-kernel-machine}
\end{equation}
Note that due to the nature of the definition of $\Theta$, \cref{eq:Banach-kernel-machine} is a \emph{kernel machine} with an \emph{arbitrary} number of kernels $K \in \mathbb{N}$ and \emph{arbitrary} kernel parameters $\{\xi_j\}_{j=1}^K$. Thus, we see that the parametric model class $\F_\Theta$ coincides with the (unclosed) span of the kernels, i.e., $\F_\Theta = \spn\{\kernel(\dummy, \xi)\}_{\xi \in \Xi} \eqqcolon \U_0$, which is simply the set of all kernel machines with kernel $\kernel(\dummy, \dummy)$. This setup departs from the RKHS setting since the parameters need not be ``centers'' that lie in the same space as the function input.

Another key difference is in the parameter cost for training \emph{sparse} kernel methods, which takes the form
\begin{equation}
    C(\theta) = \sum_{j=1}^K \abs{a_j}, \quad \theta \in \Theta.
\end{equation}
Observe that this is the $\ell^1$-norm of the kernel coefficients and therefore promotes sparsity on the number of active kernels. The parametric representation cost is then
\begin{equation}
    \Rcirc(f) = \norm{\TOp_\kernel^{-1}f}_\M = \norm*{\sum_{j=1}^K a_j \delta_{\xi_j}}_\M = \sum_{j=1}^K \abs{a_j}, \quad f \in \U_0, \label{eq:I-RKBS-param-rep-cost}
\end{equation}
where $\sum_{j=1}^K a_j \delta_{\xi_j} = \TOp_\kernel^{-1}f$ is the \emph{unique} measure such that $f$ has the integral representation \cref{eq:integral-rep}, where we assume that the $\xi_j$ are pairwise distinct without loss of generality (since otherwise, terms could be combined by modifying the coefficients).

Since $\TOp_\kernel:\M(\Xi)=C_0(\Xi)'\to\U$ is an isometric isomorphism, we transport the canonical duality via the pairing $[\dummy, \dummy]$ from \cref{eq:RKBS-pairing}. Under this pairing, $\U=C_0(\Xi)'$ isometrically. We endow $\U$ with the corresponding weak$^*$ topology $\sigma(\U,C_0(\Xi))$. Similar to the RKHS setting (see \cref{lemma:RKHS-convergence}), an I-RKBS has a useful characterization of weak$^*$-convergent sequences.

\begin{lemma} \label{lemma:I-RKBS-convergence}
Let $\U$ be an I-RKBS on $\X$ and let $(f_n)_{n\in\mathbb N}\subset \U$ and $f\in \U$. Then, $f_n$ weak$^*$ converges to $f$ if and only if $f_n(x)\to f(x)$ for all $x\in \X$ and $\sup_{n\in\mathbb N}\norm{f_n}_\U<+\infty$.
\end{lemma}
\begin{proof}
Let $\kernel:\X\times\Xi\to\R$ be an I-RKBS kernel for $\U$. If $f_n$ weak$^*$ converges to $f$, then
\begin{equation}
    f_n(x)=[\kernel(x,\dummy),f_n]
    \to
    [\kernel(x,\dummy),f]=f(x)
\end{equation}
for every $x\in\X$, and every weak$^*$-convergent sequence is norm-bounded.

Conversely, suppose that $f_n(x)\to f(x)$ for all $x\in\X$ and that $\sup_n\|f_n\|_\U<+\infty$. Pointwise convergence gives
\begin{equation}
    [\varphi,f_n]\to[\varphi,f]
\end{equation}
for every $\varphi\in\spn\{\kernel(x,\dummy):x\in\X\}$. This span is dense in $C_0(\Xi)$ by the definition of an I-RKBS kernel. The uniform norm bound and a standard density argument therefore extend the convergence to every $\varphi\in C_0(\Xi)$. Hence $f_n$ weak$^*$ converges to $f$ in $\U$.
\end{proof}

\subsection{Representation Costs and Representer Theorems}
We now establish that, given an I-RKBS kernel, the associated parametric model class (i.e., the set of kernel machines), and the $\ell^1$-norm parameter cost, the induced extended representation cost and native space coincide with the norm in the corresponding I-RKBS and the I-RKBS itself, respectively. This shows that sparse kernel methods and I-RKBSs are a special case of the abstract representation-cost and function-space perspective developed in this paper. Finally, we show that the sparse representer theorems for I-RKBS-norm regularization are compatible with our abstract framework.

For compatibility with the universal topology introduced in \cref{sec:Lip}, we now restrict our attention to the class of \emph{Lipschitz I-RKBS kernels}. In particular, in the sequel, let $(\X,\rho)$ be a separable metric space and let $\e \in \X$ be some base point. We also assume that $\Xi$ is a \emph{second-countable} locally compact Hausdorff space.

\begin{definition}
An I-RKBS kernel $\kernel:\X \times \Xi \to \R$ is said to be \emph{$(L, \rho)$-Lipschitz} if the map $x \mapsto \kernel(x, \dummy)$ is $L$-Lipschitz continuous with respect to the metric $\rho(\dummy, \dummy)$, i.e., 
\begin{equation}
\norm{\kernel(x, \dummy) - \kernel(z, \dummy)}_{L^\infty}  \leq L \, \rho(x,z), \quad \text{for all } x,z\in \X.
\end{equation}
\end{definition}

\begin{lemma}\label{lem:lipschitz_kernels_embed-RKBS}
If $\kernel: \X \times \Xi \to \R$ is an $(L, \rho)$-Lipschitz I-RKBS kernel, then its I-RKBS $\U$ continuously embeds into $\Lip(\X)$. In particular, $\norm{f}_{\Lip(\X)} \leq \max\curly{L, \norm{\kernel(\e, \dummy)}_{L^\infty}}\norm{f}_\U$, for all $f \in \U$.
\end{lemma}
\begin{proof}
For any $f \in \U$ and $x, z \in \X$, by \cref{eq:RKBS-CS}, we have that
\begin{equation}
    \abs{f(x) - f(z)} = \abs{[\kernel(x, \dummy) - \kernel(z, \dummy), f]} \leq \norm{\kernel(x, \dummy) - \kernel(z, \dummy)}_{L^\infty} \norm{f}_\U  \leq L \, \norm{f}_\U \, \rho(x, z).
\end{equation}
Therefore, $\abs{f}_{\Lip(\X)} \leq L \, \norm{f}_\U$. Next, again by \cref{eq:RKBS-CS},
\begin{equation}
    \abs{f(\e)} = \abs{[\kernel(\e, \dummy), f]} \leq \norm{\kernel(\e, \dummy)}_{L^\infty} \norm{f}_\U.
\end{equation}
Thus, for any $f \in \U$, we have $\norm{f}_{\Lip(\X)} = \max\curly{\abs{f(\e)}, \abs{f}_{\Lip(\X)}} \leq \max\curly{L, \norm{\kernel(\e, \dummy)}_{L^\infty}}\norm{f}_\U$.
\end{proof}
\begin{remark}
    Many common kernels used in machine learning fall under the class of Lipschitz I-RKBS kernels, including the RKHS kernels such as the Gaussian and Laplacian kernels discussed in \cref{remark:RKHS-Lip-kernels}, where $\Xi = \X = \R^d$. Allowing $\X$ and $\Xi$ to be different spaces also permits kernels that take the form of neurons, as we shall see in \cref{sec:shallow}.
\end{remark}

\begin{theorem} \label{thm:I-RKBS-rep-cost}
    Let $\Xi$ be a second-countable locally compact Hausdorff space, let $\kernel: \X \times \Xi \to \R$ be an $(L, \rho)$-Lipschitz I-RKBS kernel, and let $(\U, \norm{\dummy}_\U)$ be the corresponding I-RKBS. With the parametric model setup as in \cref{sec:RKBS-param}, if we construct the extended representation cost with $\B = \Lip(\X)$, we have for any $f \in \Lip(\X)$ that
    \begin{equation}
        R(f) = \begin{cases}
            \norm{f}_\U, & \text{if $f \in \U$}, \\
            +\infty, & \text{if $f \in \Lip(\X) \setminus \U$}.
        \end{cases}
    \end{equation}
    In this case, the native space is
    \begin{equation}
        \F = \curly{f \in \Lip(\X) \st R(f) < +\infty} = \U.
    \end{equation}
\end{theorem}

Before proving the theorem, we will need the following result regarding the convergence of atomic measures to arbitrary measures in $\M(\Xi)$. 
\begin{lemma}\label{lem:approx_by_discrete_measures}
Let $\Xi$ be a second-countable locally compact Hausdorff space. For every $\nu \in \M(\Xi)$, there exists a sequence of atomic measures $(\nu_n)_{n \in \N}$ (i.e., each $\nu_n \in \spn\curly{\delta_\xi}_{\xi \in \Xi}$) such that $\nu_n$ weak$^*$~converges to $\nu$ and $\norm{\nu_n}_\M \to \norm{\nu}_\M$.
\end{lemma}
\begin{proof}
Let $\nu = \nu^+ - \nu^-$ be the Jordan decomposition of $\nu$. Since $\Xi$ is second-countable, by
\cite[Theorem~6.9, p.~99]{malliavin1995integration}, there exist sequences $(\nu_n^+)_{n\in\N}$,
$(\nu_n^-)_{n\in\N}$ of atomic measures converging
narrowly\footnote{Recall that a sequence $(\mu_n)_{n\in\N} \subset \M(\Xi)$ is said to \emph{converge narrowly} to $\mu \in \M(\Xi)$ if $\int_\Xi \phi\dd \mu_n \rightarrow \int_\Xi\phi\dd\mu$ for all continuous and bounded functions $\phi:\Xi\rightarrow\R$.} to $\nu^+$ and $\nu^-$, respectively. Set $\nu_n = \nu_n^+ - \nu_n^-$. Then  $\nu_n$ converges narrowly to $\nu$, and since narrow convergence implies weak$^*$ convergence in $\M(\Xi) = C_0(\Xi)'$, $\nu_n$ weak$^*$ converges to $\nu$, as well. Also, since $\nu^\pm$ are positive measures, narrow
convergence gives $\|\nu_n^\pm\|_\M = \nu_n^\pm(\Xi) \to \nu^\pm(\Xi) = \|\nu^\pm\|_\M$. Therefore, by the fact that the $\M$-norm is weak$^*$ l.s.c.\ and using the triangle inequality, we have
\[
  \|\nu\|_\M
  \leq \liminf_{n\to\infty} \|\nu_n\|_\M
  \leq \limsup_{n\to\infty} \|\nu_n\|_\M
  \leq \lim_{n\to\infty} \bigl(\|\nu_n^+\|_\M + \|\nu_n^-\|_\M\bigr)
  = \|\nu^+\|_\M + \|\nu^-\|_\M
  = \|\nu\|_\M,
\]
where the last equality holds because $\nu^+$ and $\nu^-$ are mutually singular.
\end{proof}

\begin{proof}[Proof of \Cref{thm:I-RKBS-rep-cost}]
    From \cref{eq:I-RKBS-param-rep-cost}, the parametric representation cost $\Rcirc: \Lip(\X) \to [0, +\infty]$ is
    \begin{equation}
        \Rcirc(f) = \begin{cases}
            \norm{\TOp_\kernel^{-1} f}_\M = \norm{f}_\U & \text{if $f \in \F_\Theta = \U_0$}, \\
            +\infty & \text{if $f \in \Lip(\X) \setminus \U_0$}.
        \end{cases}
    \end{equation}
    By \cref{lem:lipschitz_kernels_embed-RKBS}, $\Rcirc: \Lip(\X) \to [0, +\infty]$ is coercive. Hence, by \cref{thm:rep_cost_lip_spaces}, the extended representation cost is given by the sequential characterization
    \begin{equation}
        R(f) = \min \curly{\alpha \in [0,+\infty] \st (f_n)_{n \in \N} \subset \U_0,\ f_n \to f \text{ pointwise},\ \norm{f_n}_\U \rightarrow \alpha}.
    \end{equation}

    Fix $f \in \U$ and set $\nu \coloneqq \TOp_\kernel^{-1}f \in \M(\Xi)$ so that $\norm{f}_\U = \norm{\nu}_\M$. By \cref{lem:approx_by_discrete_measures}, there exists a sequence of atomic measures $(\nu_n)_{n \in \N}$ such that $\nu_n$ weak$^*$ converges to $\nu$ and
    $\norm{\nu_n}_\M \to \norm{\nu}_\M$. Define $f_n \coloneqq \TOp_\kernel \nu_n$. Since $\U$ is isometrically isomorphic to $\M(\Xi)$, we have that $f_n$ weak$^*$ converges to $f$ and that $\norm{f_n}_\U \to \norm{f}_\U$. By \cref{lemma:I-RKBS-convergence}, we have that $f_n \to f$ pointwise and $\sup_{n \in \N} \norm{f_n}_\U < +\infty$.
    Thus, $R(f) \leq \lim_{n \to \infty} \norm{f_n}_\U = \norm{f}_\U$.

    Conversely, suppose that $f \in \F$ (i.e., $R(f) < +\infty$). Then, there exists a sequence $(f_n)_{n \in \N} \subset \U_0$ such that $f_n \to f$ pointwise with $R(f) = \lim_{n \to \infty} \Rcirc(f_n) = \lim_{n \to \infty} \norm{f_n}_\U$.  We have that $\sup_{n\in\N}\norm{f_n}_U<+\infty$. Since $\Xi$ is second-countable, $C_0(\Xi)$ is separable. Therefore, since $\U$ is isometrically isomorphic to $\M(\Xi)=C_0(\Xi)'$, closed and bounded subsets of $\U$ are weak$^*$-compact and metrizable. Hence, after passing to a subsequence, there exists $h\in \U$ such that $f_{n_j}$ weak$^*$ converges to $h$ in $\U$. For each $x\in\X$, point evaluation is weak$^*$-continuous on $\U$, so $f_{n_j}(x)\to h(x)$. Since $f_n\to f$ pointwise, we also have $f_{n_j}(x)\to f(x)$, and hence $h=f$ pointwise on $\X$. Thus $f\in \U$.
    
    Since the dual norm is weak$^*$ l.s.c., we have that
    \begin{equation}
      \norm{f}_\U \leq \liminf_{j\to\infty}\norm{f_{n_j}}_\U
= \lim_{n\to\infty}\norm{f_n}_\U = R(f).
    \end{equation}
    Thus, $R(f) = \norm{f}_\U$ for all $f \in \U$ and $R(f) =+\infty$ for all $f \in \Lip(\X) \setminus \U$, in which case $\F = \U$.
\end{proof}

This theorem verifies that the inductive bias of sparse kernel methods is exactly captured by the associated I-RKBS and that the representation cost is exactly captured by the I-RKBS norm. We can now investigate representer theorems for I-RKBSs through the lens of our abstract framework.

\begin{theorem} \label{thm:I-RKBS-rep-thm}
    Let $\Xi$ be a second-countable locally compact Hausdorff space, let $\kernel: \X \times \Xi \to \R$ be an $(L, \rho)$-Lipschitz I-RKBS kernel, and let $(\U, \norm{\dummy}_\U)$ be the corresponding I-RKBS. Fix the distinct data sites $x_1, \ldots, x_N \in \X$ and let $G: \R^N \to [0, +\infty]$ be proper, l.s.c., and coercive. For $\lambda > 0$, consider the problem
    \begin{equation}
        \min_{f \in \U} G\paren[\big]{f(x_1), \ldots, f(x_N)} + \lambda \norm{f}_\U. \label{eq:I-RKBS-prob}
    \end{equation}
    Then, if the objective is proper, the solution set is nonempty and there exists a minimizer $f^\star$ of \cref{eq:I-RKBS-prob} of the form
    \begin{equation}
        f^\star = \sum_{j=1}^K a_j \kernel(\dummy, \xi_j), \quad K \leq N, \label{eq:representer-I-RKBS}
    \end{equation}
    for some $a_1, \ldots, a_K \in \R$ and (adaptive) kernel parameters $\xi_1, \ldots, \xi_K \in \Xi$. 
    If we further suppose that
    \begin{equation}
        {\spn}\big\{(\kernel(x_1, \xi), \ldots, \kernel(x_N, \xi)) \st \xi \in \Xi\big\} = \R^N, \label{eq:interpolation-assump-rkbs}
    \end{equation}
    then the objective in \cref{eq:I-RKBS-prob} is always proper and the problem can be recast as the finite-atom optimization problem  (with nonempty solution set)
    \begin{equation}
        \min_{\substack{a_1, \ldots, a_K \in \R \\ \xi_1, \ldots, \xi_K \in \Xi}} G\paren*{\sum_{j=1}^K a_j \kernel(x_1, \xi_j), \ldots, \sum_{j=1}^K a_j \kernel(x_N, \xi_j)} + \lambda \sum_{j=1}^K \abs{a_j}, \quad K \geq N, \label{eq:finite-dim-I-RKBS}
    \end{equation}
    in the sense that any solution $(a^\star, \xi^\star) \in \R^K \times \Xi^K$ induces a solution to \cref{eq:I-RKBS-prob} via the map $(a^\star, \xi^\star) \mapsto \sum_{j=1}^K a^\star_j \kernel(\dummy, \xi_j^\star)$. In particular, although the parametric model class $\F_\Theta = \U_0 = \spn\{\kernel(\dummy,\xi)\}_{\xi \in \Xi}$ allows for an arbitrary number of kernels and arbitrary kernel parameters, it suffices to use at most $N$ active kernels.
\end{theorem}
\begin{proof}
    With $\B = \Lip(\X)$, we have by \cref{rem:w*-point-evaluation} that the evaluation operator $\EOp: \Lip(\X) \to \R^N$ given by
    \begin{equation}
        \EOp f = \big(f(x_1), \ldots, f(x_N)\big) \in \R^N
    \end{equation}
    is $w^*$-continuous. There exists at least one minimizer of \cref{eq:I-RKBS-prob} by \cref{thm:existence_of_min}. Let $\bar{f} \in \U$ be any minimizer of \cref{eq:I-RKBS-prob} and define $\bar{\nu} \coloneqq \TOp_\kernel^{-1} \bar{f} \in \M(\Xi)$. Then,
    \begin{equation}
         \bar{\nu} \in \argmin_{\nu \in \M(\Xi)} \norm{\nu}_\M \quad\subj\quad z_i = \int_\Xi \kernel(x_i, \xi) \dd\nu(\xi), \quad i = 1, \ldots, N, \label{eq:M-norm-prob}
    \end{equation}
    where $z_i = \bar{f}(x_i)$, $i = 1, \ldots, N$. Indeed, if this were not the case, it would contradict the optimality of $\bar{f}$. In particular, any solution $\tilde{\nu} \in \M(\Xi)$ to \cref{eq:M-norm-prob} induces a solution to \cref{eq:I-RKBS-prob} via the map $\tilde{\nu} \mapsto \TOp_\kernel \tilde{\nu}$. Since $\kernel(x_i, \dummy) \in C_0(\Xi)$, $i = 1, \ldots, N$, the problem \cref{eq:M-norm-prob} is a special case of the classical $\M$-norm minimization problem with finitely many weak$^*$ continuous linear constraints. This problem has a representer theorem~\cite{BoyerRepresenter,BrediesRepresenter,bredies2013inverse,FisherJerome,flinth2019exact,UnserSplinesgTV,zuhovickii1948remarks,zuhovickii1962approximation} that implies that there exists a solution of the form $\nu^\star = \sum_{j=1}^K a_j \delta_{\xi_j}$, for some $K \leq N$. Thus, the measure $\nu^\star$ then induces a solution to \cref{eq:I-RKBS-prob} of the form $f^\star = \sum_{j = 1}^K a_j \kernel(\dummy, \xi_j)$. 
    
    The reduction to \cref{eq:finite-dim-I-RKBS} follows from \cref{cor:parametric-function-space-equivalence}, applied to the atomic kernel parameterization $(a_j,\xi_j)_{j=1}^K \mapsto \sum_{j=1}^K a_j\kappa(\dummy,\xi_j)$ with parameter cost $\sum_{j=1}^K\abs{a_j}$, where we note that the hypothesis $\EOp(\F_\Theta) = \R^N$ is satisfied by \cref{eq:interpolation-assump-rkbs}. The hypotheses are verified with $q=p=1$ by \cref{prop:parameter-root-properties-transfer}. Thus, the parameter-space problem and the function-space problem have the same infimal value, and minimizers of the parameter-space problem induce minimizers of the function-space problem.
\end{proof}

\section{Shallow Neural Networks and Banach Spaces} \label{sec:shallow}
We now show that \emph{shallow} neural networks, i.e., networks with \emph{one} hidden layer or, equivalently, \emph{two} layers, can be studied with the same representation-cost machinery as kernel methods and RKHSs/I-RKBSs. The primary building block of a neural network is a \emph{neuron}, which is the composition of an affine map with a nonlinearity. We focus on the ReLU nonlinearity in this section. Then, a neuron is a function of the form
\begin{equation}
    \R^d \ni x \mapsto (w^\T x - b)_+ \in \R, \quad (w, b) \in \R^d \times \R,
\end{equation}
where $(\dummy)_+ \coloneqq \max\curly{0, \dummy}$ is the ReLU. Then, a shallow ReLU neural network is a function of the form
\begin{equation}
    f_\theta(x) = \sum_{j=1}^K v_j (w_j^\T x - b_j)_+ + w_0^\T x - b_0, \quad \theta = \{(v_j, w_j, b_j)\}, \label{eq:ReLU-form}
\end{equation}
where $v_j \in \R$ are the \emph{output weights}, $w_j \in \R^d$ are the \emph{input weights}, $b_j \in \R$ are the \emph{biases}, and $K$ is the \emph{width} of the network. The affine term $x \mapsto w_0^\T x - b_0$ is called a \emph{skip connection} or \emph{residual connection} in neural network parlance. The parameter space is
\begin{equation}
\Theta \coloneqq \{(w_0, b_0), (v_1,w_1,b_1),\ldots,(v_K,w_K,b_K)\st v_j \in \R, w_j \in \R^d, b_j \in \R, K\in \mathbb{N}\}.
\end{equation}
\begin{remark}
    The point of the skip connection in \cref{eq:ReLU-form} is to provide a crisp function-space description. In particular, it allows for the separation of the affine part of the model, so that the representation cost measures only the genuinely non-affine component of the neural network.
\end{remark}

\subsection{Representation Costs and Representer Theorems}
The parameter cost for training shallow ReLU neural networks is the \emph{weight-decay} cost
\begin{equation} \label{eq:shallow-relu-weight-decay}
    C(\theta) = \frac{1}{2}\sum_{j=1}^K \abs{v_j}^2 + \norm{w_j}_2^2,
\end{equation}
which is the squared Euclidean norm of the vectorization of the penalized hidden and output weights~\cite{hanson1988comparing,krogh1991simple}. Let $\Aff(\R^d) \coloneqq \curly{w^\T(\dummy) - b \st w \in \R^d, b \in \R}$ denote the space of affine functions. This is a norm-closed subspace of $\Lip(\R^d)$, where we fix the base point $\e = 0$ in this section. We have the following result.

\begin{lemma} \label{lemma:shallow-rep-cost-coercive-const}
    If we construct the extended representation cost $R$ with $\B = \Lip(\R^d)$, then, for all $f \in \Lip(\R^d)$, we have
    \begin{equation}
        \inf_{g \in \Aff(\R^d)} \norm{f + g}_{\Lip(\R^d)} \leq R(f) \leq \Rcirc(f). \label{eq:shallow-Lip-bound}
    \end{equation}
    In particular, $\Rcirc$ and $R$ are coercive modulo $\Aff(\R^d)$. Furthermore, $\const(\Rcirc) = \const(R) = \Aff(\R^d)$.
\end{lemma}
\begin{proof}
    The space $\Aff(\R^d)$ is a norm-closed $(d+1)$-dimensional subspace of $\Lip(\R^d)$ and is hence $w^*$-closed. For any $\theta \in \Theta$, observe that
    \begin{equation}
        \inf_{g \in \Aff(\R^d)} \norm{f_\theta + g}_{\Lip(\R^d)} \leq \sum_{j=1}^K \abs{v_j} \norm{w_j}_2 \leq \frac{1}{2} \sum_{j=1}^K \abs{v_j}^2 + \norm{w_j}_2^2,
    \end{equation}
    where the first inequality holds since the Lipschitz constant of each term $v_j(w_j^\T(\dummy) - b_j)_+$ is upper bounded by $\abs{v_j}\norm{w_j}_2$, and the second inequality holds since $ab \leq (a^2 + b^2)/2$. Hence, by \cref{lem:suff_cond_R_coercive}, $\Rcirc$ and $R$ are both coercive modulo $\Aff(\R^d)$ and \cref{eq:shallow-Lip-bound} holds. Finally, it is clear that $\Aff(\R^d) \subset \const(\Rcirc)$ since any affine function can be added to any $f \in \F_\Theta$ without changing the parametric representation cost (as the skip connection is not penalized by the parameter cost). By \cref{lem:nesting_of_const_spaces}, $\Aff(\R^d) \subset \const(\Rcirc) \subset \const(R)$. By \cref{prop:coercive_mod_determines_constancy}, $\Aff(\R^d) = \const(\Rcirc) = \const(R)$.
\end{proof}

Observe that the ReLU is \emph{positively homogeneous}, i.e., $(\gamma \dummy)_+ = \gamma (\dummy)_+$ for any $\gamma > 0$. This homogeneity implies that the parametric representation cost of $f \in \F_\Theta$ can be written as
\begin{align}
    \Rcirc(f) &= \inf\curly*{\frac{1}{2}\sum_{j=1}^K \abs{v_j}^2 + \norm{w_j}_2^2 \st f = \sum_{j=1}^K v_j (w_j^\T (\dummy) - b_j)_+ + w_0^\T (\dummy) - b_0} \nonumber \\
    &= \inf\curly*{\sum_{j=1}^K \abs{v_j} \norm{w_j}_2 \st f = \sum_{j=1}^K v_j (w_j^\T (\dummy) - b_j)_+ + w_0^\T (\dummy) - b_0} \nonumber \\
    &= \inf\curly*{\sum_{j=1}^K \abs{v_j} \st f = \sum_{j=1}^K v_j (w_j^\T (\dummy) - b_j)_+ + w_0^\T (\dummy) - b_0, \quad w_j \in \Sph^{d-1}}, \label{eq:weight-to-1-norm}
\end{align}
where $\Sph^{d-1} \coloneqq \curly{x \in \R^d \st \norm{x}_2 = 1}$ denotes the unit sphere and $\Rcirc(f) = +\infty$ for any $f \in \Lip(\R^d) \setminus \F_\Theta$. 
The second equality follows by balancing each pair of input and output weights and applying the equality case of the arithmetic mean--geometric mean inequality. The third equality follows by normalizing each nonzero input weight and absorbing its norm into the corresponding output weight.
For more details about this reduction, see~\cite[\S2]{shenouda2024variation} (see also~\cite{grandvalet1998least,neyshabur2015search,OngieFunctionSpace,parhi2021banach,parhi2023deep,SavareseInfWidth}).

With the reduction in \cref{eq:weight-to-1-norm}, observe that the neurons parameterized by $(w_j, b_j) \in \cyl$ and $(-w_j, -b_j) \in \cyl$ only differ by an affine function. Hence, we can further simplify the parametric representation cost as
\begin{align}
    \Rcirc(f) 
    &= \inf\curly*{\sum_{j=1}^K \abs{v_j} \st f = \sum_{j=1}^K v_j \frac{\abs{w_j^\T (\dummy) - b_j}}{2} + w_0^\T (\dummy) - b_0, \quad \curly{w_j}_{j \geq 1} \subset \Sph^{d-1}} \nonumber \\
    &= \inf\curly*{\sum_{j=1}^K \abs{v_j} \st f = \sum_{j=1}^K v_j \frac{\abs{w_j^\T (\dummy) - b_j}}{2} + w_0^\T (\dummy) - b_0, \quad \curly{(w_j, b_j)}_{j \geq 1} \subset \mathbb{P}^d}, \label{eq:shallow-Rcirc-simplified}
\end{align}
where $\mathbb{P}^d$ denotes the space of unoriented affine hyperplanes in $\R^d$.\footnote{This space can be identified with the quotient $\mathbb{P}^d \cong (\cyl) / \sim$ where $(w, b) \sim (-w, -b)$.} The second equality holds since the absolute value is an even function. This reduction reveals the utility of including an unpenalized skip connection in the architecture and parameter cost.

We can now immediately recover the sharp characterization of the representation cost of shallow ReLU neural networks via the Radon transform due to~\cite{OngieFunctionSpace} using the machinery developed in this paper. Recall that the Radon transform of $f \in L^1(\R^d)$ is given by
\begin{equation}
    \RadonOp\curly{f}(u, t) = \int_{\{x \in \R^d \st u^\T x = t\}} f(x) \dd x, \quad (u, t) \in \mathbb{P}^d,
\end{equation}
where $\dd x$ denotes integration against the $(d-1)$-dimensional Lebesgue measure on the hyperplane $\curly{x \in \R^d \st u^\T x = t}$. In the sequel, the Radon transform is understood in the \emph{sense of distributions}, see, e.g.,~\cite{LudwigRadon,parhi2024distributional,RammRadonDist}. 

\begin{theorem} \label{thm:Radon-BS}
    If we construct the extended representation cost with $\B = \Lip(\R^d)$, we have for any $f \in \Lip(\R^d)$ that
    \begin{equation}
    R(f) = \begin{cases}
        c_d\norm{\RadonOp (-\Delta)^{\frac{d+1}{2}} f}_\M, & \text{if $f \in \F$}, \\
        +\infty, &\text{if $f \in \Lip(\R^d) \setminus \F$},
    \end{cases} \label{eq:Radon-seminorm-rep-cost}
    \end{equation}
    where $c_d \coloneqq 1/(2(2\pi)^{d-1})$, $\Delta = \sum_{i=1}^d \partial_{x_i}^2$ denotes the Laplacian\footnote{Fractional powers of the Laplacian are understood via the Fourier transform.} and
    \begin{equation}
        \F = \curly{f \in \Lip(\R^d) \st c_d\RadonOp (-\Delta)^{\frac{d+1}{2}} f \in \M(\mathbb{P}^d)}
    \end{equation}
    denotes the native space. Furthermore, the native space is a Banach space when endowed with the norm
    \begin{equation}
        \norm{f}_\F = \max\{R(f), \abs{f(0)}, \norm{(f(\e_1) - f(0), \ldots, f(\e_d) - f(0))}_2\}, \label{eq:Radon-norm}
    \end{equation}
    where $\e_i$ denotes the $i$th canonical basis vector.
\end{theorem}
\begin{proof}
From~\cite[Appendix~A]{parhi2022kinds}, consider the projector $\P_\Aff: \Lip(\R^d) \to \Lip(\R^d)$ defined by
\begin{equation}
    \P_\Aff f = f(0) + \sum_{i=1}^d (f(\e_i) - f(0)) x_i, \quad f \in \Lip(\R^d), \label{eq:Aff-proj}
\end{equation}
where $x = (x_1, \ldots, x_d) \in \R^d$. This projector is $w^*$-continuous since it is defined by a finite number of point evaluations. Let $\Aff(\R^d)^\comp$ denote the corresponding topological complement of $\Aff(\R^d)$. The corresponding complementary projector is $\P_{\Aff^\comp} \coloneqq (\Id - \P_\Aff): \Lip(\R^d) \to \Lip(\R^d)$. Let
\begin{equation}
    f \coloneqq \sum_{j=1}^K v_j \frac{\abs{w_j^\T (\dummy) - b_j}}{2} + w_0^\T (\dummy) - b_0, \quad \curly{(w_j, b_j)}_{j \geq 1} \subset \mathbb{P}^d,
\end{equation}
a direct calculation (cf.~\cite[pp.~482--483]{parhi2022kinds}) reveals that
\begin{equation}
    \P_{\Aff^\comp} f = \sum_{j=1}^K v_j \kernel_{\ReLU}(\dummy, (w_j, b_j)), \label{eq:Aff-comp-proj}
\end{equation}
with
\begin{equation} \label{eq:ReLU-kernel}
    \kernel_{\ReLU}(x, (w, b)) = \frac{\abs{w^\T x - b}}{2} - \frac{\abs{b}}{2} \paren*{1 - \sum_{i=1}^d x_i} - \sum_{i=1}^d x_i \frac{\abs{w_i - b}}{2}, \quad (w, b) \in \mathbb{P}^d,
\end{equation}
where, in particular, $\kernel_{\ReLU}(x, \dummy) \in C_0(\mathbb{P}^d)$ for all $x \in \R^d$~\cite[p.~483]{parhi2022kinds}. By \cref{lemma:shallow-rep-cost-coercive-const}, $\Aff(\R^d) \subset \const(\Rcirc)$. Hence, we can apply \cref{lem:chatacterizing_R_mod_N} to find that the extended representation cost admits the factorization
\begin{equation}
    R = \overline{\Rcirc|_{\Aff(\R^d)^\comp}}^{\,w^*_{\Aff^\comp}} \circ \P_{\Aff^\comp},
\end{equation}
where $w^*_{\Aff^\comp}$ is the subspace topology of $w^*$ on $\Aff^\comp(\R^d)$. From \cref{eq:Aff-comp-proj}, we have that
\begin{equation}
    \Rcirc|_{\Aff(\R^d)^\comp}(f) = \inf\curly*{\sum_{j=1}^K \abs{v_j} \st f = \sum_{j=1}^K v_j \kernel_{\ReLU}(\dummy, (w_j, b_j)), \quad \curly{(w_j, b_j)}_{j \geq 1} \subset \mathbb{P}^d}
\end{equation}
for $f \in \spn\{\kernel_{\ReLU}(\dummy, (w, b))\}_{(w, b) \in \mathbb{P}^d}$ and $+\infty$ otherwise.

By \cite[Theorem~3.6]{parhi2025function} with $k = (d-1)$ and $\LOp = \Delta$, the integral operator
\begin{equation}
    \TOp_{\ReLU}: \M(\mathbb{P}^d) \to \R^{\R^d}: \nu \mapsto \int_{\mathbb{P}^d} \kernel_{\ReLU}(\dummy, (w, b)) \dd \nu(w, b)
\end{equation}
is an injection (see also~\cite[Lemma~1]{OngieFunctionSpace}). Hence, by \cref{lemma:I-RKBS}, $\kernel_{\ReLU}:  \R^d \times \mathbb{P}^d \to \R$ is an I-RKBS kernel. Let $(\U, \norm{\dummy}_\U)$ denote the corresponding I-RKBS. Next, the map $x \mapsto \kernel_{\ReLU}(x, \dummy)$ is $(1, \norm{\dummy}_2)$-Lipschitz~\cite[Equation~(C.3)]{parhi2026compositional}. Finally, $\mathbb{P}^d$ is a second-countable locally compact Hausdorff space. Thus, by \cref{thm:I-RKBS-rep-cost},\footnote{Since the span of the kernels and its corresponding I-RKBS are contained in
the weak$^*$-closed space $\Aff(\mathbb R^d)^\comp$, the l.s.c.\
regularizations computed in $\Aff(\mathbb R^d)^\comp$ and in
$\Lip(\mathbb R^d)$ agree on $\Aff(\mathbb R^d)^\comp$.} we have for $f \in \Lip(\R^d)$ that
\begin{equation}
    R(f) = \begin{cases}
        \norm{\P_{\Aff^\comp} f}_\U, &\text{if $\P_{\Aff^\comp} f \in \U$}, \\
        +\infty, &\text{else}.
    \end{cases}
\end{equation}

Observe that
\begin{equation}
    \norm{\P_{\Aff^\comp} f}_\U = \norm{\TOp_{\ReLU}^{-1} \P_{\Aff^\comp} f}_\M.
\end{equation}
From \cite[Theorem~3.6]{parhi2025function} with $k = (d-1)$ and $\LOp = \Delta$, we have that $\TOp_{\ReLU}^{-1} = -c_d \RadonOp (-\Delta)^{\frac{d+1}{2}}$. Thus,
\begin{equation}
    \TOp_{\ReLU}^{-1} \P_{\Aff^\comp} f = -c_d \RadonOp (-\Delta)^{\frac{d+1}{2}} \P_{\Aff^\comp} f = -c_d \RadonOp (-\Delta)^{\frac{d+1}{2}} f, \quad \P_{\Aff^\comp} f \in \U.
\end{equation}
Hence,
\begin{equation}
    R(f) = \begin{cases}
        c_d \norm{\RadonOp (-\Delta)^{\frac{d+1}{2}} f}_\M, & \text{if $f \in \F$}, \\
        +\infty, &\text{if $f \in \Lip(\R^d) \setminus \F$}.
    \end{cases} \label{eq:R-form-proof-Radon-BS}
\end{equation}

Since $\Rcirc$ is subadditive and $1$-homogeneous, as can be seen directly from \cref{eq:shallow-Rcirc-simplified}, we have that $\F$ is a Banach space when endowed with \cref{eq:Radon-norm} by \cref{cor:F_is_quasi_banach} with $\Null = \Aff(\R^d)$ and the projector $\P_\Aff$ from \cref{eq:Aff-proj}.
\end{proof}

This theorem reveals that the inductive bias of shallow ReLU neural networks with the weight-decay parameter cost is exactly captured by the Radon-domain \emph{seminorm} \cref{eq:Radon-seminorm-rep-cost}. By the intertwining properties of the Radon transform, if we define the operator $\KOp \coloneqq c_d (-\partial_t^2)^{\frac{d-1}{2}}$, that acts on the $t$-variable of the Radon domain, we have the equality
\begin{equation}
    R(f) = \norm{\KOp \RadonOp \Delta f}_\M. \label{eq:FBP-version}
\end{equation}
The operator $\KOp$ is the ``backprojection filter'' from the filtered-backprojection algorithm in computed tomography. Due to the equality \cref{eq:FBP-version}, the native space $\F$ for this representation cost is referred to as the space of functions of \emph{second-order Radon bounded variation} and denoted by $\RBV^2(\R^d) \coloneqq \F$. Furthermore, this representation cost is referred to as the second-order Radon total variation of a function, denoted by $\RTV^2(f) \coloneqq R(f)$~\cite{parhi2021banach,parhi2022kinds,parhi2023deep,parhi2022near}. This is because when $d = 1$, $\RTV^2(\dummy) = \TV^2(\dummy)$, which is the classical notion of second-order total variation of a univariate function. Indeed, $f: \R \to \R$ is said to have bounded second-order total variation if $f'' \in \M(\R)$. Furthermore, the second-order total variation is precisely given by $\TV^2(f) = \norm{f''}_\M$. 

\begin{theorem} \label{thm:Radon-BS-rep-thm}
    Let $x_1, \ldots, x_N \in \R^d$ be distinct and suppose that the restriction of the evaluation map
    \begin{equation}
    \EOp: f \mapsto \paren[\big]{f(x_1), \ldots, f(x_N)}
    \end{equation}
    to $\Aff(\R^d)$ is injective. Let $G: \R^N \to [0, +\infty]$ be proper, l.s.c., and coercive. For $\lambda > 0$, consider the problem
    \begin{equation}
        \min_{f \in \RBV^2(\R^d)} G\paren[\big]{f(x_1), \ldots, f(x_N)} + \lambda \RTV^2(f). \label{eq:RTV2-prob}
    \end{equation}
    Then, the solution set is nonempty and there exists a minimizer $f^\star$ of \cref{eq:RTV2-prob} of the form
    \begin{equation}
        f^\star(x) = \sum_{j=1}^K v_j (w_j^\T x - b_j)_+ + w_0^\T x - b_0, \quad x \in \R^d, \quad K \leq N, \label{eq:ReLU-representer}
    \end{equation}
    for some $v_1, \ldots, v_K \in \R$, $w_0, \ldots, w_K \in \R^d$, and $b_0, \ldots, b_K \in \R$. The problem \cref{eq:RTV2-prob} can be recast as the finite-dimensional optimization problem (with nonempty solution set)
    \begin{equation}
        \min_{\theta = \{(v_j, w_j, b_j)\} \cup \{(w_0, b_0)\}} G\paren[\big]{f_\theta(x_1), \ldots, f_\theta(x_N)} + \frac{\lambda}{2} \sum_{j=1}^K \abs{v_j}^2 + \norm{w_j}_2^2, \quad K \geq N, \label{eq:weight-decay-representer}
    \end{equation}
    where $\theta$ and $f_\theta$ are as in \cref{eq:ReLU-form}, in the sense that any solution $\theta^\star$ induces a solution to \cref{eq:RTV2-prob} via the map $\theta^\star \mapsto f_{\theta^\star}$. In particular, although the parametric model class allows for an arbitrary number of neurons, it suffices to use at most $N$ active neurons.
\end{theorem}
\begin{proof}
    With $\B = \Lip(\R^d)$, we have by \cref{rem:w*-point-evaluation} that the evaluation operator $\EOp: \Lip(\R^d) \to \R^N$ given by
    \begin{equation}
        \EOp f = \big(f(x_1), \ldots, f(x_N)\big) \in \R^N
    \end{equation}
    is $w^*$-continuous. By \cref{lemma:shallow-rep-cost-coercive-const}, $R$ is coercive modulo $\Aff(\R^d)$. The injectivity assumption on $\EOp$ is equivalent to $\ker \EOp \cap \Aff(\R^d) = \{0\}$.  Furthermore, we have $\EOp(\F_\Theta) = \R^N$ by~\cite[Theorem~5.1]{pinkus1999approximation} since $x_1, \ldots, x_N$ are distinct and so the objective in \cref{eq:RTV2-prob} is proper. There exists at least one minimizer of \cref{eq:RTV2-prob} by \cref{thm:existence_of_min}. Let $\bar{f} \in \RBV^2(\R^d)$ be any minimizer of \cref{eq:RTV2-prob} and write $\bar{f} = \bar{g} + \bar{q}$ with $\bar{g} \coloneqq \P_{\Aff^\comp} \bar{f}$ and $\bar{q} \coloneqq \P_{\Aff} \bar{f}$. Then,
    \begin{equation}
        \bar{g} \in \argmin_{g \in \U} \norm{g}_\U \quad\subj\quad g(x_i) = \bar{f}(x_i) - \bar{q}(x_i), \quad i = 1, \ldots, N, \label{eq:ReLU-I-RKBS-prob}
    \end{equation}
    where $\U$ is the I-RKBS of $\kernel_{\ReLU}$ \cref{eq:ReLU-kernel}. Indeed, if this were not the case, it would contradict the optimality of $\bar{f}$. In particular, any solution $\tilde{g} \in \U$ to \cref{eq:ReLU-I-RKBS-prob} induces a solution to \cref{eq:RTV2-prob} via the map $\tilde{g} \mapsto \tilde{g} + \bar{q}$. The form of the solution \cref{eq:ReLU-representer} follows from \cref{thm:I-RKBS-rep-thm} by noting that $\kernel_{\ReLU}(\dummy, (w, b))$ only differs from a ReLU neuron $x \mapsto (w^\T x - b)_+$ by an affine function. 
    
    The reduction to \cref{eq:weight-decay-representer} follows from \cref{cor:parametric-function-space-equivalence}, applied to the shallow ReLU network parameterization \cref{eq:ReLU-form} with the weight-decay cost \cref{eq:shallow-relu-weight-decay}. The hypotheses are verified with $q=p=1$ by \cref{prop:parameter-root-properties-transfer}. Indeed, the subadditivity follows by concatenating neurons, while the homogeneity follows from the positive homogeneity of the ReLU and the standard rescaling of input and output weights. Thus, the parameter-space problem and the function-space problem have the same infimal value, and minimizers of the parameter-space problem induce minimizers of the function-space problem. Conversely, the representer theorem above produces a function-space minimizer in the shallow ReLU model class with at most $N$ active neurons. Choosing a balanced parameterization gives no gap between $R$ and $\Rcirc$ and attains the representation cost. Hence, the corresponding parameters solve the finite-dimensional problem \cref{eq:weight-decay-representer}.
\end{proof}
\begin{remark}
    The assumption that $\EOp|_{\Aff(\R^d)}$ is injective is rather mild. Indeed, it is equivalent to the matrix
    \begin{equation}
        \begin{bmatrix}
        1 & x_1^\T \\ \vdots & \vdots \\ 1 & x_N^\T
        \end{bmatrix} \in \R^{N \times (d+1)}
    \end{equation}
    having full column rank. This happens, in particular, when $N \geq d + 1$ and the $\{x_j\}_{j=1}^N$ are in general position. This occurs, with probability one, for instance, when $N \geq d+1$ and the $\{x_j\}_{j=1}^N$ are drawn i.i.d.\ with respect to a probability measure that is absolutely continuous with respect to the Lebesgue measure on $\R^d$, a common assumption in statistics and machine learning.
\end{remark}
\begin{remark}
    Although this section focused on scalar-valued shallow ReLU networks, analogous results hold in the vector-valued case. In that case, the native space would correspond to the vector-valued space $\RBV^2(\R^d; \R^D)$ when the parameter cost is the weight-decay regularizer. See~\cite[Section~2.2]{parhi2026compositional} for a precise definition of this space. Recent work has also investigated sharper characterizations of the function-space bias specific to the vector-valued scenario such as conditions for uniqueness of minimizers~\cite{nakhleh2024new}.
\end{remark}

\subsection{Functional Analysis of Shallow Neural Networks}
The functional analysis of shallow neural networks has been a fruitful area of research since at least the 1990s, perhaps due to the seminal work of Barron~\cite{BarronApprox,BarronApproxEst}. In the pre-deep-learning era, this led to the line of work studying the \emph{variation spaces} with respect to dictionaries~\cite{BachConvexNN,barron2008approximation,kurkova2001bounds,kurkova2002comparison,MatouvsekZonotope,mhaskar2004tractability}, where the dictionaries were often composed of neurons. The interest in these spaces is that they admit nonlinear approximation rates that do not grow with input dimension. Thus, these spaces are sometimes said to be \emph{immune to the curse of dimensionality}. A key attribute of members of variation spaces is that the only irregular functions are those that vary in a few directions. This led Donoho to coin the term \emph{mixed variation} to describe these kinds of functions~\cite{DonohoMixedVariation}.

In the post-deep-learning era, this perspective was reinvigorated by several works coming from different perspectives such as the mean-field formulation~\cite{mei2018mean}, gradient flows in spaces of measures~\cite{chizat2018global,ma2022barron}, early work on representation costs~\cite{OngieFunctionSpace,SavareseInfWidth}, and nonlinear approximation theory~\cite{devore2025weighted,stephan1,siegel2025optimal,siegel2020approximation,siegel2023characterization,yang2025optimal} for shallow neural networks. This then led to investigations on representer theorems for shallow neural networks. A variant of \cref{thm:Radon-BS-rep-thm} first appeared in~\cite{parhi2021banach} and was then followed up by a series of papers~\cite{bartolucci2024lipschitz,BartolucciRKBS,najaf2026sampling,nakhleh2025global, pieper2022nonconvex,spek2025duality,unser2023ridges,unser2024kernel}, where Banach-space representer theorems were proven for shallow neural networks. The perspectives taken in these works were either building upon the characterization due to~\cite{OngieFunctionSpace} of the native space via the Radon transform~\cite{BartolucciRKBS,parhi2021banach,unser2023ridges,unser2024kernel} or from the perspective of I-RKBSs~\cite{bartolucci2024lipschitz,BartolucciRKBS,korolev2022two,najaf2026sampling,spek2025duality}. \Cref{thm:Radon-BS-rep-thm} reveals that this representer theorem naturally emerges from the representation-cost framework and simultaneously unifies the research program on the functional analysis of shallow neural networks.

\subsubsection{Spectral Barron Spaces}
Depending on the community, the native spaces of shallow neural networks are often discussed under the term ``Barron spaces'' (cf.~\cite{ma2022barron}) to honor Barron's seminal work~\cite{BarronApprox,BarronApproxEst} on shallow neural network approximation theory. Historically, however, it is important to distinguish between native spaces and the original spaces studied by Barron. His original work studied a class of functions defined through weighted integrability of the Fourier transform of the function. In modern terminology, these spaces are called \emph{spectral Barron spaces}~\cite{chen2021representation,siegel2022high} to avoid confusion with referring to the native spaces as Barron spaces.

For $s>0$, define
\begin{equation}
\mathscr{B}^s(\R^d) \coloneqq
\curly{f\in \Sch'(\R^d) \st (1+\norm{\dummy}_2^2)^{s/2} \widehat f(\dummy) \in \M(\R^d)}, \label{eq:Barron-space}
\end{equation}
where we recall that $\Sch'(\R^d)$ denotes the space of tempered distributions on $\R^d$, $\hat{\dummy}$ denotes the (distributional) Fourier transform,\footnote{Under the convention that for $f \in L^1(\R^d)$, the Fourier transform is $\hat{f}(\xi) = \int_{\R^d} f(x) \e^{-\imag2\pi \xi^\T x} \dd x$, $\xi \in \R^d$.} and $\M(\R^d)$ is the space of finite Radon measures on $\R^d$. This is a Banach space when endowed with the norm
\begin{equation} \label{eq:Barron-norm}
\norm{f}_{\mathscr{B}^s} \coloneqq  \norm{(1+\norm{\dummy}_2^2)^{s/2} \widehat f(\dummy)}_{\M}.
\end{equation}
This is the \emph{spectral Barron space of order $s$}. When $s=1$, it is a measure-valued extension of the first-order spectral class used in Barron's original approximation theorem~\cite{BarronApprox}.\footnote{Most works consider an $L^1(\R^d)$ condition and an $L^1$-norm rather than working with $\M(\R^d)$ and the $\M$-norm in \cref{eq:Barron-space,eq:Barron-norm}. On bounded domains, this does not change the definition of the space or norm~\cite{siegel2023characterization}, but on $\R^d$ working with $\M(\R^d)$ yields a space that is strictly larger~\cite{parhi2022continuous}.}

Although the spectral Barron spaces do not neatly correspond to the native spaces of shallow neural networks with common activation functions, they do fit naturally into the representation-cost framework and, in particular, the I-RKBS perspective of \cref{sec:I-RKBS} once one chooses the
right kernel. Indeed, let $\Xi=\R^d$ and consider the complex-valued\footnote{Although \cref{sec:I-RKBS} focused on real I-RKBSs, the complex case is identical.} kernel
\begin{equation}
\kernel_s(x,\xi) \coloneqq \frac{\e^{\imag 2\pi \xi^\T x}}{(1+\|\xi\|_2^2)^{s/2}},
\qquad (x,\xi)\in\R^d \times \R^d.
\end{equation}
Since $s > 0$, we have that $\kernel_s(x, \dummy) \in C_0(\R^d)$ for all $x \in \R^d$. Thus, for a finite complex-valued Radon measure $\nu\in \M(\R^d)$, the operator
\begin{equation}
\TOp_s\nu \coloneqq \int_{\R^d} \kernel_s(\dummy,\xi)\dd\nu(\xi)
\end{equation}
is well defined. Then
\begin{equation}
\hat{\TOp_s\nu} =
(1+\norm{\dummy}_2^2)^{-s/2}\nu,
\end{equation}
so that
\begin{equation}
(1+\norm{\dummy}_2^2)^{s/2} \,\hat{\TOp_s\nu} = \nu.
\end{equation}
Thus, since the Fourier transform is an automorphism of $\Sch'(\R^d)$, we immediately see that the operator $\TOp_s: \M(\R^d) \to \Sch'(\R^d)$ is an injection whose range is precisely $\mathscr{B}^s(\R^d)$. By \cref{lemma:I-RKBS} combined with the fact that
\begin{equation}
\norm{\TOp_s\nu}_{\mathscr{B}^s} = \norm{\nu}_{\M},
\end{equation}
we see that the spectral Barron space $\mathscr{B}^s(\R^d)$ is an I-RKBS whose kernel is a weighted Fourier mode rather than a neuron.

\section{Deep Neural Networks and Quasi-Banach Spaces} \label{sec:deep-nets}

In this section, we apply our abstract construction of representation costs and native spaces to \emph{deep} neural networks. Our main goal is to characterize the function-space geometry induced by depth. Unlike the settings considered previously in \cref{sec:RKHS,sec:I-RKBS,sec:wavelet,sec:shallow}, the representation costs associated with depth-$L$ ReLU networks have homogeneity degree $2/L$. We show that they induce a $p$-normed quasi-Banach structure with
\begin{equation}
p = \frac{2}{L}.
\end{equation}
More precisely, the depth-$L$ representation cost is a power of a $(2/L)$-seminorm and the native space is complete under an associated $(2/L)$-norm. In addition, under suitable hypotheses, the representation cost has a nonconvex unit ball (\cref{prop:deep-unit-ball-nonconvex}). This geometry is analogous to that of $x \mapsto \|x\|_p^p$, $0<p<1$, in finite-dimensional regularization. 

The results of this section should be viewed as \emph{nonlinear} analogs of the well-known \mbox{Schatten-$(2/L)$} geometry associated with depth-$L$ deep linear networks. We begin by reviewing the simpler deep linear network setting before turning to deep ReLU networks.

\subsection{Warm-Up: Deep Linear Neural Networks}

As a warm-up, we recall the representation costs induced by deep linear networks with $\ell^2$-parameter (i.e., weight-decay) regularization. This provides the linear analogue of the depth-dependent regularizers that appear in the sequel. In particular, the deep linear setting illustrates three phenomena that persist in the nonlinear setting: \begin{enumerate*}[label=(\roman*)]
    \item depth changes the induced representation cost,
    \item for $L>2$ the representation cost is nonconvex, and
    \item the cost is a power of the Schatten-$(2/L)$ quasi-norm.
\end{enumerate*}

Let $\Lin(\R^d;\R^D) \subset \Lip(\R^d;\R^D)$ denote the space of linear maps from $\R^d$ to $\R^D$. Every $f \in \Lin(\R^d;\R^D)$ is represented uniquely by a matrix $J_f \in \R^{D \times d}$ such that
\begin{equation}
    f(x) = J_f x, \quad x \in \R^d.
\end{equation}
Let ${\D}f: \R^d \to \R^{D \times d}$ denote the Jacobian of $f$. Observe that ${\D}f(x) = J_f$ for any $x \in \R^d$.

Fix $L \geq 2$. We parameterize $\Lin(\R^d;\R^D)$ by depth-$L$ linear networks. Let $\ThetaL$ denote the collection of all tuples $\theta = (W_1,\ldots,W_L)$ of compatible matrices such that the product $W_L \cdots W_1$ is well-defined, and set
\begin{equation}
    f_\theta(x) = W_L W_{L-1}\cdots W_1 x.
\end{equation}
Clearly, under this parameterization, the parametric model class is $\FThetaL = \Lin(\R^d; \R^D)$. The corresponding parameter cost is
\begin{equation}
    C_L(\theta) = \frac{1}{L} \sum_{\ell=1}^L \norm{W_\ell}_F^2.
\end{equation}
Thus, the induced parametric representation cost is
\begin{equation}
    \RcircL(f) = \inf \curly[\Bigg]{ \frac{1}{L} \sum_{\ell=1}^L \norm{W_\ell}_F^2 \st (W_1,...,W_L) \in \ThetaL, W_L \cdots W_1 = J_f }.
\end{equation}
A result from the matrix factorization literature (see, e.g.,~\cite{dai2021representation,ShangSchatten,MertParallel}) states that
\begin{equation}
    \RcircL(f) = \norm{J_f}_{S^{2/L}}^{2/L} \coloneqq \sum_{i=1}^{\min\{d, D\}} \sigma_i(J_f)^{2/L},
\end{equation}
where $\norm{\dummy}_{S^p}$ denotes the Schatten-$p$ (quasi-)norm, and $\sigma_i(\dummy)$ denotes the $i$th singular value. In particular, when $L=2$, the induced representation cost coincides with the \emph{nuclear norm}, i.e., $R_{\Theta_2}(f) = \norm{J_f}_{S^1} = \norm{J_f}_*$.

The relationship between low-rank matrix factorization, Schatten norms, and convex relaxations of rank minimization has a long history in optimization, inverse problems, and machine learning; see, e.g.,~\cite{arora2019implicit,fazel2001rank,gunasekar2017implicit,recht2010guaranteed,srebro2004maximum,srebro2005rank}. A closely related line of work studies nonconvex low-rank factorizations of matrix and semidefinite optimization problems, beginning with the Burer--Monteiro factorization and its later generalizations~\cite{boumal2020deterministic,burer2003nonlinear,burer2005local}. From the representation-cost perspective, depth-$2$ linear networks induce nuclear-norm regularization, while deeper linear networks induce regularization by the $(2/L)$-th power of the Schatten-$(2/L)$ quasi-norm.

More recently, these induced geometries have been studied through the lens of implicit regularization and optimization bias in deep linear and convolutional networks~\cite{gunasekar2018implicit}. Furthermore,~\cite{dai2021representation} systematically studies representation costs induced by a broad class of linear neural architectures, including fully connected, convolutional, diagonal, and residual networks. In particular, different architectures (i.e., \emph{parameterizations}) induce different \emph{geometries} on function space even when the underlying model class remains linear. For example, convolutional architectures induce sparsity-promoting regularization in the discrete Fourier transform domain, while diagonal architectures induce vector and group quasi-norm regularization.

More generally, when $L>2$, the representation cost is nonconvex. Furthermore, as $L \to \infty$, one recovers the rank penalty since
\begin{equation}
    \lim_{L\to\infty} \RcircL(f) = \lim_{L \to \infty} \norm{J_f}_{S^{2/L}}^{2/L} = \rank(J_f).
\end{equation}
Thus, in the linear setting, increasing depth interpolates between convex low-rank regularization and rank minimization.

Finally, $\RcircL$ already coincides with its $w^*$-l.s.c.\ regularization. Indeed, if $(f_n)_{n\in\N} \subset \Lin(\R^d;\R^D)$ converges pointwise to $f$, then $f \in \Lin(\R^d;\R^D)$ and $J_{f_n} \to J_f$ in operator norm. Since singular values depend continuously on the matrix entries,
\begin{equation}
    \RcircL(f_n) = \norm{J_{f_n}}_{S^{2/L}}^{2/L} \to \norm{J_f}_{S^{2/L}}^{2/L} = \RcircL(f).
\end{equation}
Hence $\RcircL$ is continuous on $\Lin(\R^d;\R^D)$ with respect to the subspace topology of $w^*$. Since $\Lin(\R^d;\R^D)$ is $w^*$-closed (since it is finite-dimensional), the trivial extension of $\RcircL$ to $\Lip(\R^d;\R^D)$ is $w^*$-l.s.c. Consequently, 
\begin{equation}
 R_L = \RcircL \quad\text{and}\quad \F_L = \FThetaL = \Lin(\R^d;\R^D).
\end{equation}
The next subsection shows that the same depth-dependent structure persists for nonlinear ReLU networks. In particular, depth-$L$ ReLU representation costs are absolutely $(2/L)$-homogeneous and subadditive, and the costs induce $(2/L)$-normed quasi-Banach structures.

\subsection{Representation Costs and Representer Theorems}\label{subsec:deep_nets_rep_cost}
Let $\X \subset \R^d$ be any nonempty subset (which could be all of $\R^d$) and fix any basepoint $\e\in\X$. Consider the class of all $L$-layer fully-connected ReLU networks with inputs in $\X \subset \R^d$ and outputs in $\R^D$, having arbitrary intermediate-layer widths. More explicitly, this is the set of all functions having the form
\begin{equation}
f_\theta = a_L \circ \sigma \circ a_{L-1} \circ \sigma \circ \cdots \circ \sigma \circ a_1
\end{equation}
where each $a_\ell:\R^{K_\ell}\rightarrow\R^{K_{\ell+1}}$ is an affine function, i.e., $a_\ell(z) = W_\ell z + b_\ell$ for some
$W_\ell \in \R^{K_{\ell+1}\times K_\ell}$ and some $b_\ell \in \R^{K_{\ell+1}}$, and $\sigma = \max\{0, \dummy\}$ denotes the ReLU activation applied entrywise. Here we assume $K_1 = d$, $K_{L+1} = D$, while otherwise the hidden-layer widths $K_\ell \in \N$ of layers $\ell=2,...,L$ are arbitrary. Note that any such  $f_\theta$ belongs to $\Lip(\X;\R^D)$ since it is the composition of Lipschitz-continuous functions.

The corresponding parameter space $\ThetaL$ consists of all $L$-tuples of weight/bias pairs
\begin{equation}
\theta = \big((W_1,b_1),...,(W_L,b_L)\big),
\end{equation}
and we define the $L$-layer weight-decay parameter cost by 
\begin{equation}
C_L(\theta) = \frac{1}{L}\sum_{\ell=1}^L \|W_\ell\|_F^2,
\end{equation}
i.e., the averaged sum of squares of all \emph{non-bias} parameters in the network. Thus, the induced parametric representation cost is
\begin{equation}
    \RcircL(f) = \inf \curly[\Bigg]{ \frac{1}{L} \sum_{\ell=1}^L \norm{W_\ell}_F^2 \st \theta = \big((W_1,b_1),...,(W_L,b_L)\big) \in \ThetaL,~f=f_\theta}.
\end{equation}

The purpose of this section is to provide characterizations of the extended representation cost $R_L = \overline{\RcircL}$ and the associated native space 
\begin{equation}
\F_L \coloneqq \dom(R_L) = \{ f \in \Lip(\X;\R^D) \st R_L(f) < +\infty \}. 
\end{equation}
Let $\Const(\X;\R^D) \subset \Lip(\X;\R^D)$ denote the norm-closed subspace of constant functions. 

\begin{lemma}\label{lem:deep_relu_coercive_bnd}
Fix $L \geq 2$. For all $f\in\Lip(\X;\R^D)$, we have
\begin{equation}\label{eq:deep_rep_cost_lip_bounds}
\inf_{g \in \Const(\X;\R^D)}\|f+g\|_{\Lip(\X;\R^D)} \leq R_L(f)^{L/2} \leq \RcircL(f)^{L/2}.
\end{equation}
In particular, $\RcircL$ and $R_L$ are coercive modulo $\Const(\X;\R^D)$. Furthermore, $\const(\RcircL) = \const(R_L) = \Const(\X;\R^D)$.
\end{lemma}
\begin{proof}
The ReLU activation function is $1$-Lipschitz and each affine layer $a_\ell(z) = W_\ell z + b_\ell$ is $\|W_\ell\|_{\op}$-Lipschitz. Hence,
\begin{equation}
|f_\theta|_{\Lip(\X;\R^D)} \leq \prod_{\ell=1}^L \|W_\ell\|_{\op} \leq \prod_{\ell=1}^L \|W_\ell\|_F \leq \left(\frac{1}{L}\sum_{\ell=1}^L \|W_\ell\|_F^2\right)^{L/2} \leq C_L(\theta)^{L/2} 
\end{equation}
where the penultimate inequality is by the inequality of arithmetic and geometric means.
Further,
\begin{align}
\inf_{g\in \Const(\X;\R^D)}\|f_\theta+g\|_{\Lip(\X;\R^D)}
 &\leq \|f_\theta-f_\theta(\e)\|_{\Lip(\X;\R^D)} \nonumber \\
 &= |f_\theta-f_\theta(\e)|_{\Lip(\X;\R^D)} \nonumber \\
 &= |f_\theta|_{\Lip(\X;\R^D)} \leq C_L(\theta)^{L/2}.
\end{align}
Therefore, $C_L$ is parameter coercive modulo $\Const(\X;\R^D)$. Applying \cref{lem:suff_cond_R_coercive} (where $\gamma(t) = t^{L/2}$) yields \cref{eq:deep_rep_cost_lip_bounds}. Finally, it is obvious that $\Const(\X;\R^D) \subset \const(\RcircL)$ (since the final-layer bias is not penalized by $C_L$).  By \cref{lem:nesting_of_const_spaces}, $\Const(\X;\R^D) \subset \const(\RcircL) \subset \const(R_L)$. By \cref{prop:coercive_mod_determines_constancy}, $\Const(\X;\R^D) = \const(\RcircL) = \const(R_L)$.
\end{proof}

\begin{lemma}\label{lem:deep_rep_costs_are_quasi_seminorms}
$\RcircL$ and $R_L$ are subadditive and absolutely $2/L$-homogeneous.
\end{lemma}
\begin{proof}
By \cref{prop:parameter-root-properties-transfer}, it suffices to show that $C_L$ is parameter $1$-subadditive and parameter $(1,2/L)$-homogeneous. We first prove parameter $1$-subadditivity. Let $\theta,\vartheta \in \ThetaL$ be given. Write the corresponding affine layers as
\begin{equation}
    a_{\ell}^{\theta}(z) = W_{\ell}^{\theta}z + b_{\ell}^{\theta}
    \quad\text{and}\quad
    a_{\ell}^{\vartheta}(z) = W_{\ell}^{\vartheta}z + b_{\ell}^{\vartheta}.
\end{equation}
We construct a parameter vector $\eta \in \ThetaL$ such that $f_{\eta} = f_{\theta} + f_{\vartheta}$. The construction runs the two networks in parallel and adds their outputs in the final layer. Define
\begin{equation}
    W_1^\eta =
    \begin{bmatrix}
        W_1^\theta \\
        W_1^\vartheta
    \end{bmatrix}
    \quad\text{and}\quad
    b_1^\eta =
    \begin{bmatrix}
        b_1^\theta \\
        b_1^\vartheta
    \end{bmatrix}.
\end{equation}
For $2 \leq \ell \leq L-1$, define 
\begin{equation}
    W_\ell^\eta =
    \begin{bmatrix}
        W_\ell^\theta & 0 \\
        0 & W_\ell^\vartheta
    \end{bmatrix}
    \quad\text{and}\quad
    b_\ell^\eta =
    \begin{bmatrix}
        b_\ell^\theta \\
        b_\ell^\vartheta
    \end{bmatrix}.
\end{equation}
Finally, define
\begin{equation}
    W_L^\eta =
    \begin{bmatrix}
        W_L^\theta & W_L^\vartheta
    \end{bmatrix}
    \quad\text{and}\quad
    b_L^\eta = b_L^\theta + b_L^\vartheta.
\end{equation}
Since the ReLU is applied entrywise, the first $L-1$ layers of $\eta$ compute the hidden representations of the two networks independently, and the final layer adds their outputs. Hence $f_\eta = f_\theta + f_\vartheta$. Moreover, by construction, $\norm{W_\ell^\eta}_F^2 = \norm{W_\ell^\theta}_F^2 + \norm{W_\ell^\vartheta}_F^2$,  $\ell=1,\ldots,L$. Therefore,
\begin{equation}
    C_L(\eta)
    =
    \frac{1}{L}\sum_{\ell=1}^L \norm{W_\ell^\eta}_F^2
    =
    C_L(\theta)+C_L(\vartheta).
\end{equation}
This proves parameter $1$-subadditivity.

We next prove parameter $(1,2/L)$-homogeneity. Let $\alpha \in \R \setminus \{0\}$ and let
\begin{equation}
    \theta = ((W_1,b_1),\ldots,(W_L,b_L)) \in \ThetaL.
\end{equation}
Define
\begin{equation}
\theta_\alpha \coloneqq \Big((\gamma W_1,\gamma b_1), (\gamma W_2,\gamma^2 b_2), \ldots, (\gamma W_{L-1},\gamma^{L-1} b_{L-1}), (\operatorname{sgn}(\alpha)\gamma W_L,\alpha b_L) \Big), \quad
\gamma \coloneqq |\alpha|^{1/L}.
\end{equation}
Thus, we have that $f_{\theta_\alpha} = \alpha f_\theta$. Furthermore,
\begin{equation}
    C_L(\theta_\alpha)
    =
    \frac{1}{L}\sum_{\ell=1}^L |\alpha|^{2/L}\norm{W_\ell}_F^2
    =
    |\alpha|^{2/L}C_L(\theta).
\end{equation}
Finally, the parameter vector with all weights and biases equal to zero realizes the zero function and has cost zero. Thus $C_L$ is parameter $(1,2/L)$-homogeneous.
\end{proof}

\begin{theorem} \label{thm:quasi-Banach-geometry} 
Let $L\geq 2$. Then $R_L^{L/2}$ is a $(2/L)$-seminorm. Furthermore, $\F_L$ is a $(2/L)$-normed quasi-Banach space under the quasi-norm $\|f\|_{\F_L} \coloneqq \max\{R_L(f)^{L/2}, \|f(\e)\|_2\}$.
\end{theorem}
\begin{proof}
\Cref{lem:deep_relu_coercive_bnd,lem:deep_rep_costs_are_quasi_seminorms} show that $\RcircL$ is subadditive, absolutely $(2/L)$-homogeneous, and coercive modulo $\Const(\X;\R^D)$, which is a $w^*$-complemented subspace in $\Lip(\X;\R^D)$ (since it is finite dimensional). Furthermore, a projector onto $\Const(\X;\R^D)$ that is $w^*$-continuous on $\Lip(\X;\R^D)$ is given by $\P_{\Const}f \equiv f(\e)$, and we have $\|\P_{\Const}f\|_{\Lip(\X;\R^D)} = \|f(\e)\|_2$. Therefore, the result follows by \Cref{cor:F_is_quasi_banach}.
\end{proof}

On its own, \cref{thm:quasi-Banach-geometry} does \emph{not} rule out the possibility that the ``rooted'' representation cost $R_L^{L/2}$ is a \textit{bona fide} seminorm on $\F_L$, in which case the augmented quasi-norm $\norm{\dummy}_{\F_L}$ would be a \textit{bona fide} norm. Under the nondegeneracy hypotheses below, it turns out that this is not the case, which we show in \cref{prop:deep-unit-ball-nonconvex}. In particular, we show that, when $L>2$, $d,D \geq 2$, and $\X$ has nonempty interior, the unit representation cost ball, unit rooted representation cost ball, and unit quasi-norm ball are all \emph{nonconvex} and therefore do not induce a \textit{bona fide} (semi)norm.

The proof of this result relies on the following lemma, which bounds products of singular values of the Jacobian of any function in terms of its representation cost. Let $\operatorname{int}(\X)$ denote the interior of $\X$ and, for $A\in\R^{D\times d}$ we let $\sigma_1(A)\geq\sigma_2(A)\geq\cdots$ denote the singular values of $A$, with the convention that $\sigma_m(A)\coloneqq 0$ whenever $m$ exceeds the rank of $A$. Recall that, by Rademacher's theorem, every $f\in\Lip(\X;\R^D)$ is differentiable at almost every $x\in\operatorname{int}(\X)$ and the Jacobian ${\D} f$ satisfies ${\D} f\in L^\infty_{\mathrm{loc}}(\operatorname{int}(\X);\R^{D\times d})$ (cf.\ \cref{rem:lip-sobolev}).

\begin{lemma}\label{lem:jacobian-minor-bound}
Let $L\geq 2$ and $1\leq k\leq\min\{d,D\}$, and suppose that
$\operatorname{int}(\X)\neq\varnothing$. Then, for every $f\in\Lip(\X;\R^D)$,
\begin{equation}
    \operatorname*{ess\,sup}_{x\in\operatorname{int}(\X)}\;
    \prod_{m=1}^{k}\sigma_m\paren{\D f(x)}
    \;\leq\;
    \paren*{\frac{R_L(f)}{k}}^{kL/2}.
    \label{eq:jacobian-minor-bound}
\end{equation}
\end{lemma}
The proof of \cref{lem:jacobian-minor-bound} is given in \cref{app:singval-lower-bnd}. This lemma is also later used to prove \cref{prop:rank-lower-bound}, which provides a lower bound on the extended representation cost $R_L(f)$ in the limiting regime as $L \to \infty$.

\begin{proposition}
\label{prop:deep-unit-ball-nonconvex}
Let $L>2$, let $d,D\geq 2$, and suppose that $\operatorname{int}(\X)\neq\varnothing$. Then the two unit balls
\begin{equation}
    \curly{f\in \F_L \st R_L(f)\leq 1}
    \quad\text{and}\quad
    \curly{f\in \F_L \st R_L(f)^{L/2}\leq 1}
\end{equation}
coincide and are nonconvex. Consequently, the unit ball
\begin{equation}
    \curly{f\in\F_L\st \norm{f}_{\F_L}\leq 1}
\end{equation}
is nonconvex.
In particular, $R_L^{L/2}$ does not define a \textit{bona fide} seminorm on $\F_L$ and $\norm{\dummy}_{\F_L}$ does not define a \textit{bona fide} norm on $\F_L$.
\end{proposition}

The proof of \cref{prop:deep-unit-ball-nonconvex} is given in \cref{app:nonconvex}. The proof of \cref{prop:deep-unit-ball-nonconvex} boils down to averaging two specially constructed ridge functions with unit representation cost. Then, the $k=2$ case of \cref{lem:jacobian-minor-bound} forces the representation cost of the average to exceed one. This proposition 
shows that the function-space complexity of deep ReLU neural networks as defined by the representation cost is generally \emph{not} a (semi)norm. However, we do remark that \cref{prop:deep-unit-ball-nonconvex} does not preclude the existence of a \textit{bona fide} norm that is \emph{equivalent} to $\norm{\dummy}_{\F_L}$. Exploring this would be an interesting open problem.

To conclude this section, the next result is a representer theorem for data-fitting problems regularized with the representation cost $R_L$ for all $L\geq 2$. To the best of our knowledge, this is the first deep neural network representer theorem (with a vanilla architecture) ``matched'' to the weight-decay regularizer over a complete function space. Interestingly, it suffices to have the widths of the deep neural network solutions be bounded by the number of training data squared.

\begin{theorem}
    \label{thm:deep-relu-representer}
    Fix $L \geq 2$ and let $x_1,\ldots,x_N \in \X$ be distinct. Let
    $G: (\R^D)^N \to [0,+\infty]$ be proper, l.s.c., and coercive. For
    $\lambda > 0$, consider the function-space problem
    \begin{equation}
        \min_{f \in \F_L}
        G\big(f(x_1),\ldots,f(x_N)\big) + \lambda R_L(f).
        \label{eq:depth-L-native-prob}
    \end{equation}
    Then, the solution set of \cref{eq:depth-L-native-prob} is nonempty and
    there exists a minimizer $f^\star$ of \cref{eq:depth-L-native-prob} that is
    realized by a finite-width depth-$L$ ReLU network. Moreover, $f^\star$ may
    be chosen so that
    \begin{equation}
        f^\star = f_{\theta^\star}
    \end{equation}
    for some $\theta^\star \in \Theta_L(N^2)$, where $\Theta_L(N^2)$ denotes
    the finite-dimensional parameter space of depth-$L$ ReLU networks with
    hidden-layer widths at most $N^2$.

    Moreover, for every $K \geq N^2$, the function-space problem
    \cref{eq:depth-L-native-prob} can be recast as the finite-dimensional
    optimization problem
    \begin{equation}
        \min_{\theta \in \Theta_L(K)}
        G\big(f_\theta(x_1),\ldots,f_\theta(x_N)\big)
        +
        \frac{\lambda}{L}
        \sum_{\ell=1}^L \norm{W_\ell}_F^2,
        \label{eq:depth-L-param}
    \end{equation}
    in the sense that \cref{eq:depth-L-param} has a nonempty solution set and
    any solution $\theta^\star$ of \cref{eq:depth-L-param} induces a solution
    of \cref{eq:depth-L-native-prob} via the map
    $\theta^\star\mapsto f_{\theta^\star}$. In particular, although the
    parametric model class allows for arbitrary finite hidden-layer widths, it
    suffices to use hidden-layer widths at most $N^2$.
\end{theorem}

\begin{proof}
Let
\begin{equation}
    \EOp:\Lip(\X;\R^D)\to(\R^D)^N,
    \quad
    \EOp f \coloneqq \big(f(x_1),\ldots,f(x_N)\big)
\end{equation}
be the evaluation operator. By \cref{rem:w*-point-evaluation}, $\EOp$ is $w^*$-continuous. By \cref{lem:deep-relu-finite-sample-surjectivity} we have that $\EOp(\FThetaL) = (\R^D)^N$. By \cref{lem:deep_rep_costs_are_quasi_seminorms}, $\RcircL$ is subadditive and absolutely $(2/L)$-homogeneous. Equivalently, the hypotheses of \cref{prop:parameter-root-properties-transfer} hold with $q=1$ and $p=2/L$.  Thus, we can apply \cref{cor:parametric-function-space-equivalence} to find that
\begin{equation}
    \inf_{\theta\in\ThetaL}
    G\big(f_\theta(x_1),\ldots,f_\theta(x_N)\big)
    +
    \lambda C_L(\theta)
    =
    \inf_{f\in\F_L}
    G\big(f(x_1),\ldots,f(x_N)\big)
    +
    \lambda R_L(f),
    \label{eq:deep-param-native-same-inf}
\end{equation}
and that minimizers of the parameter-space problem induce minimizers of the function-space problem.  By \cref{thm:parametric-deep-existence-width-bound}, the parameter-space problem over $\ThetaL$ admits a minimizer $\theta^\star\in\Theta_L(N^2)$. Hence, by \cref{cor:parametric-function-space-equivalence}, $f_{\theta^\star}$ is a minimizer of \cref{eq:depth-L-native-prob}. This proves that the function-space solution set is nonempty and contains a finite-width depth-$L$ ReLU network with hidden-layer widths at most $N^2$.

Finally, let $K\geq N^2$. By \cref{thm:parametric-deep-existence-width-bound}, the restricted problem \cref{eq:depth-L-param} has a nonempty solution set and has the same minimum value as the parameter-space problem over $\ThetaL$. Therefore any minimizer of \cref{eq:depth-L-param} is also a minimizer of the parameter-space problem over $\ThetaL$. Applying \cref{cor:parametric-function-space-equivalence} once more shows that any such minimizer induces a minimizer of \cref{eq:depth-L-native-prob}. This completes the proof.
\end{proof}

\begin{remark}
    \Cref{thm:deep-relu-representer} resolves an open problem posed by~\cite[Remark~9]{parkinson2024depth} which asked about the existence of minimizers for \cref{eq:depth-L-native-prob}. This theorem resolves this question in the affirmative and provides a finite-width deep-neural-network representation as a solution.
\end{remark}

\subsection{Relationships Between Native Spaces at Different Depths}

Any function realized by a deep ReLU network with finite widths is necessarily a \emph{continuous piecewise linear} (CPwL) function. Describing the exact space of CPwL functions realizable by a fixed-depth ReLU network is an active area of research. We now describe some of the fundamental results in this literature, which motivate analogous questions for the associated native spaces.

Let $\CPWL(\X;\R^D)$ denote the restrictions to $\X$ of continuous piecewise linear functions $\R^d\to\R^D$ with finitely many pieces. Recall that $\FThetaL$ is the space of all functions realizable as an $L$-layer fully-connected ReLU network with arbitrarily large but finite hidden-layer widths. As mentioned above, we have
\begin{equation}
    \FThetaL \subset \CPWL(\X;\R^D)
\end{equation}
for all $L \geq 2$.

Since every function realized by an $L$-layer ReLU network can also be realized by an $(L+1)$-layer ReLU network, we have the \emph{nesting property}
\begin{equation}
    \FThetaL \subset \F_{\Theta_{L+1}},
\end{equation}
for all $L \geq 2$.  In particular, the function $\R^2 \ni x \mapsto \max\{0,x_1,x_2\} \in \R$ belongs to $\F_{\Theta_3}$ but not to $\F_{\Theta_2}$ when $d \geq 2$~\cite[Proposition~2.2]{mukherjee2017lower}. This example is also discussed in the context of the \emph{depth
hierarchy problem} in~\cite{hertrich2021towards}. Therefore, when $\X = \R^d$ and $d \geq 2$, we have that
\begin{equation}
    \F_{\Theta_2} \subsetneq \F_{\Theta_3}.
\end{equation}

The precise number of layers needed to represent arbitrary CPwL functions remains largely open. Let $L_{\CPWL}(d)$ denote the smallest number of layers such that every function in $\CPWL(\R^d;\R^D)$ is representable by an $L_{\CPWL}(d)$-layer ReLU network with arbitrary finite widths. This number is independent of the output dimension $D$ since we can simply apply the scalar constructions coordinatewise and run the resulting networks in parallel. More specifically, we have that
\begin{equation}
    L_{\CPWL}(d) = \min \curly{L \geq 2 \st \F_{\Theta_L} = \CPWL(\R^d;\R^D)}.
\end{equation}

The classical construction based on representing CPwL functions as linear combinations of maxima of $(d+1)$ affine functions and implementing the maximum by a binary tree, gives
\begin{equation}
    L_{\CPWL}(d) \leq \lceil \log_2(d+1)\rceil + 1;
\end{equation}
see, e.g.,~\cite{AroraCPWL,he2020relu,wang2005generalization}.
This bound is sharp in low dimensions in the sense that
\begin{equation}
    L_{\CPWL}(1)=2
    \quad\text{and}\quad
    L_{\CPWL}(2)=L_{\CPWL}(3)=3.
\end{equation}
Indeed, in dimensions $d=2,3$, two hidden layers, equivalently three layers in
our convention, are necessary and sufficient to represent arbitrary CPwL
functions~\cite{he2020relu}. 

The optimal value of $L_{\CPWL}(d)$ is not known for all $d$. In particular, no CPwL function is currently known to require more than three layers to represent. Recent work~\cite{bakaev2025better} improved the general upper bound to
\begin{equation}
    L_{\CPWL}(d) \leq \lceil \log_3(d-1)\rceil + 2,
    \quad d\geq 3.
\end{equation}
This was then improved again in~\cite{ruess2026shallower} to 
\begin{equation}
    L_{\CPWL}(d) \leq \left\lceil \log_6\paren*{\frac{d+1}{2}}\right\rceil + 2,
    \quad d\geq 3.
\end{equation}
This, in particular, reveals that $L_{\CPWL}(d) = 3$ for $2 \leq d \leq 11$. We now investigate analogous properties for the native spaces $\F_L$.
\begin{theorem}
\label{thm:deep-native-space-nesting}
For every $L \geq 2$, we have the nesting property
\begin{equation}
    \F_L \subset \F_{L+1}.
\end{equation}
In particular, for every $f \in \F_L$,
\begin{equation}
    R_{L+1}(f) \leq \frac{2L}{L+1} R_L(f) + \frac{2D}{L+1}.
\end{equation}
\end{theorem}

\begin{proof}
Let $f \in \F_L$. By the sequential characterization of $R_L$ from \cref{thm:rep_cost_lip_spaces}, there exists a sequence $(f_n)_{n\in\N} \subset \FThetaL$ such that $f_n \to f$ pointwise and $\lim_{n\to\infty} \RcircL(f_n) = R_L(f) < +\infty$. For each $n \in \N$, let $\theta_n = ((W_{\ell,n},b_{\ell,n}))_{\ell=1}^L \in \ThetaL$ be such that $f_n = f_{\theta_n}$ and $C_L(\theta_n) \leq \RcircL(f_n) + 1/n$. We construct an $(L+1)$-layer network that realizes the same function. Write
\begin{equation}
    f_{\theta_n} = a_{L,n} \circ \sigma \circ a_{L-1,n} \circ \cdots \circ \sigma \circ a_{1,n}.
\end{equation}
Define $\widetilde{\theta}_n \in \Theta_{L+1}$ by keeping the first $L-1$ affine layers unchanged, replacing the final affine map $a_{L,n}$ by
\begin{equation}
    \widetilde a_{L,n}(z)
    =
    \begin{bmatrix}
        a_{L,n}(z) \\
        -a_{L,n}(z)
    \end{bmatrix},
\end{equation}
and adding the final affine map
\begin{equation}
    \widetilde a_{L+1}(u,v) = u - v.
\end{equation}
Using the identity $\sigma(t)-\sigma(-t)=t$, applied entrywise, we have $f_{\widetilde{\theta}_n} = f_{\theta_n} = f_n$. Moreover,
\begin{equation}
    C_{L+1}(\widetilde{\theta}_n) = \frac{1}{L+1} \sq*{\paren*{\sum_{\ell=1}^{L-1} \norm{W_{\ell,n}}_F^2 + 2\norm{W_{L,n}}_F^2} + 2D} \leq \frac{2L}{L+1} C_L(\theta_n) + \frac{2D}{L+1}.
\end{equation}
Therefore,
\begin{equation}
    R_{L+1}(f) \leq \liminf_{n\to\infty} R_{\Theta_{L+1}}(f_n) \leq \liminf_{n\to\infty} C_{L+1}(\widetilde{\theta}_n) \leq \frac{2L}{L+1} R_L(f) + \frac{2D}{L+1} < +\infty,
\end{equation}
where the first inequality follows from \cref{thm:rep_cost_lip_spaces}, since $f_n\to f$
pointwise and the $R_{\Theta_{L+1}}(f_n)$ are uniformly bounded.
Hence $f \in \F_{L+1}$, which proves $\F_L \subset \F_{L+1}$.
\end{proof}

\begin{corollary} \label{cor:deep-native-space-nesting}
For every $L \geq 2$, the inclusion map $\F_L \hookrightarrow \F_{L+1}$ is continuous with respect to the quasi-Banach topologies induced by the quasi-norms $\norm{\dummy}_{\F_L}$, $L \geq 2$.
In particular, there exists a constant $C_{L,D}<+\infty$ such that
\begin{equation}
\norm{f}_{\F_{L+1}}
\leq
C_{L,D}\norm{f}_{\F_L},
\quad f\in \F_L.
\end{equation}
\end{corollary}

\begin{proof}
Write $A \coloneqq 2L/ (L+1)$ and $B \coloneqq 2D /(L+1)$. By \cref{thm:deep-native-space-nesting}, for every $g\in \F_L$, we have $R_{L+1}(g) \leq A R_L(g)+B$. Applying this inequality to $g=tf$ and using the homogeneity of $R_L$ and $R_{L+1}$ gives for every $t>0$,
\begin{equation}
t^{2/(L+1)}R_{L+1}(f) = R_{L+1}(tf) \leq A t^{2/L}R_L(f)+B. \label{eq:t-ineq}
\end{equation}
If $R_L(f)>0$, choose $t=R_L(f)^{-L/2}$. Then
\begin{equation}
R_{L+1}(f) \leq (A+B)R_L(f)^{L/(L+1)}. \label{eq:same-bound}
\end{equation}
If $R_L(f)=0$, letting $t\to+\infty$ in \cref{eq:t-ineq} gives
$R_{L+1}(f)=0$, so the same bound \cref{eq:same-bound} holds. Thus,
\begin{equation}
R_{L+1}(f)^{(L+1)/2} \leq (A+B)^{(L+1)/2}R_L(f)^{L/2}.
\end{equation}
Finally,
\begin{equation}
\norm{f}_{\F_{L+1}} = \max\{R_{L+1}(f)^{(L+1)/2},\norm{f(\e)}_2\} \leq C_{L,D}\norm{f}_{\F_L},
\end{equation}
where
$C_{L,D} \coloneqq \max\{ (A+B)^{(L+1)/2},1 \}$,
which proves the continuity of the embedding.
\end{proof}

The previous theorem and corollary show that the native spaces are continuously nested with depth. Moreover, under the hypotheses of \cref{prop:deep-unit-ball-nonconvex}, the common unit ball of $R_L$ and $R_L^{L/2}$ is nonconvex. The first nontrivial depth separation occurs already between depths two and three.

\begin{proposition}
\label{prop:F2_F3_depth_separation}
Assume $\X = \R^d$, $d \geq 2$, and $D = 1$. Then $\F_2 \subsetneq \F_3$. Moreover, 
\begin{equation}
    \CPWL(\R^d;\R) \cap \F_2 = \F_{\Theta_2}. \label{eq:cpwl-intersection}
\end{equation}
\end{proposition}
\begin{proof}
    From~\cite[Example~4]{OngieFunctionSpace}, the pyramid function $f(x) = \max\{0, 1 - \norm{x}_1\}$, $x \in \R^d$, satisfies $f \not \in \F_2$ and $f \in \F_{\Theta_3}$ when $d \geq 2$. Since $\F_{\Theta_3} \subset \F_3$ by construction, we immediately have that $\F_2 \subsetneq \F_3$ for $d \geq 2$. The identity in \cref{eq:cpwl-intersection} follows from~\cite[Theorem~1]{mccarty2023piecewise}.
\end{proof}

Combining the CPwL realizability results stated above with the inclusion $\FThetaL \subset \F_L$ gives the following corollary.

\begin{corollary}
\label{cor:CPWL_subset_FL_subset_Lip}
Let $\X \subset \R^d$ be nonempty. For every $L \geq L_{\CPWL}(d)$ we have that
\begin{equation}
    \CPWL(\X;\R^D) \subset \F_L \subset \Lip(\X;\R^D).
\end{equation}
\end{corollary}
\begin{proof}
The inclusion $\F_L \subset \Lip(\X;\R^D)$ holds by construction. On the other hand, by definition, we have for $L \geq L_{\CPWL}(d)$ that $\CPWL(\R^d;\R^D) = \FThetaL$. Restricting functions to $\X$ and using $\FThetaL \subset \F_L$ yields $\CPWL(\X;\R^D) \subset \F_L$.
\end{proof}

Beyond these results, the relationship between the spaces $\F_L$ remains largely open. The nesting results (\cref{thm:deep-native-space-nesting,cor:deep-native-space-nesting}) reveal that depth enlarges the native space in a continuous manner, and the separation $\F_2 \subsetneq \F_3$ shows that this enlargement is nontrivial, at least for low depths. However, it is not known whether every containment
\begin{equation}
    \F_L \subset \F_{L+1}
\end{equation}
is strict, nor is it known whether the spaces eventually stabilize. It is also natural to ask whether the CPwL characterization of the depth-$2$ native space persists at greater depths, namely whether or not
\begin{equation}
    \CPWL(\R^d;\R^D) \cap \F_L = \FThetaL
\end{equation}
for every $L \geq 2$. A related question is whether increasing depth eventually exhausts all Lipschitz functions, either in the strong sense that
\begin{equation}
    \F_L = \Lip(\X;\R^D)
\end{equation}
for all sufficiently large $L$, or in the weaker sense that
\begin{equation}
    \bigcup_{L\geq 2} \F_L = \Lip(\X;\R^D).
\end{equation}
At present, we do not know whether either statement holds. Thus, while the parametric model classes eventually contain all CPwL functions, the corresponding saturation behavior of the native spaces remains an open problem.

A related perspective comes from the behavior of the representation cost as the depth tends to infinity. For deep linear neural networks, we previously saw that increasing depth yields a representation cost that interpolates between nuclear-norm regularization and rank minimization as depth increases. For nonlinear neural networks with one nonlinear layer and many linear layers, the authors of~\cite{parkinson2025relu} showed that the representation cost converges to a nonlinear notion of rank which they term the \emph{index rank} of a function as depth increases. For fully nonlinear networks, other nonlinear notions of rank were proposed by Jacot~\cite{JacotBNRegularity,JacotBNRank}. Let $\X \subset \R^d$ be an open convex set, and for $f \in \CPWL(\X;\R^D)$, define the \emph{Jacobian rank}
\begin{equation} \label{eq:Jac-rank}
    \rank_\mathrm{J}(f) = \esssup_{x\in\X} \rank({\D} f(x)),
\end{equation}
where we note that the Jacobian ${\D}f(x)$ exists for almost every $x \in \X$, and define the \emph{bottleneck rank}
\begin{equation} \label{eq:BN-rank}
    \rank_{\mathrm{BN}}(f) = \min\curly[\Big]{r\in\N \st f = h \circ g,\  g \in \CPWL(\X;\R^r),\ h \in \CPWL(\R^r;\R^D)}.
\end{equation}
In particular, it was shown\footnote{The parametric cost considered in~\cite{JacotBNRank} differs from the one presented in this section in that bias terms are also penalized, leading to a slightly different parametric representation cost. Nevertheless, the bounds given in \cref{eq:jacot-bounds} can be shown to hold both with and without penalizing bias terms.} in~\cite{JacotBNRank} that
\begin{equation}\label{eq:jacot-bounds}
    \rank_\mathrm{J}(f)
    \leq
    \liminf_{L\to\infty} \RcircL(f)
    \leq
    \limsup_{L\to\infty} \RcircL(f)
    \leq
    \rank_{\mathrm{BN}}(f),
    \quad
    f \in \CPWL(\X;\R^D),
\end{equation}
and the author of~\cite{JacotBNRank} conjectures that $\lim_{L\to\infty} \RcircL(f)$ exists and is equal to $\rank_{\mathrm{BN}}(f)$ for all $f \in \CPWL(\X;\R^D)$. If true, this would suggest that, at large depths, the parametric representation cost may favor functions with low-dimensional intermediate representations, in analogy with the low-rank inductive bias of deep linear networks. A related analysis for architectures with residual connections is given in~\cite{boix2025inductive}.

It is not obvious \textit{a priori} whether analogous bounds hold for the \emph{extended} representation cost $R_L$ in the limiting regime as $L \to \infty$. However, as a direct consequence of \cref{lem:jacobian-minor-bound}, the next proposition reveals that the Jacobian-rank lower bound holds for the limiting extended representation cost. 

For $f\in\Lip(\X;\R^D)$ with $\operatorname{int}(\X)\neq\varnothing$, we
extend the definition \cref{eq:Jac-rank} %
of the Jacobian rank to $f$ by defining
\begin{equation}
    \operatorname{rank}_\mathrm{J}(f)
    \coloneqq
    \esssup_{x\in\operatorname{int}(\X)}
    \operatorname{rank}\paren{{\D}f(x)},
    \label{eq:jacobian-rank-lip}
\end{equation}
which is well-defined for any $f\in\Lip(\X;\R^D)$.

\begin{proposition}\label{prop:rank-lower-bound}
Suppose that $\operatorname{int}(\X)\neq\varnothing$. Then, for every
$f\in\Lip(\X;\R^D)$,
\begin{equation}
    \operatorname{rank}_\mathrm{J}(f) \leq \liminf_{L\to\infty}R_L(f)
    .
    \label{eq:rank-lower-bound}
\end{equation}
\end{proposition}

\begin{proof}
Put $k\coloneqq\operatorname{rank}_\mathrm{J}(f)\leq\min\{d,D\}$. If $k=0$, the
claim is trivial, so assume $k\geq 1$. Define
\begin{equation}
    c
    \coloneqq
    \esssup_{x\in\operatorname{int}(\X)}\;
    \prod_{m=1}^{k}\sigma_m\paren{{\D} f(x)}.
\end{equation}
Observe that $\operatorname{rank}({\D} f(x))\geq k$ if and only if $\prod_{m=1}^{k}\sigma_m(\D f(x))>0$. Thus, by definition of $\operatorname{rank}_\mathrm{J}(f)$, the latter holds on a set of positive measure, and so $c>0$. Moreover, $\sigma_1(\D f(x))\leq\abs{f}_{\Lip(\X;\R^D)}$ holds for almost every $x \in \operatorname{int}(\X)$, so $c\leq\abs{f}_{\Lip(\X;\R^D)}^k<+\infty$. By
\cref{lem:jacobian-minor-bound},
\begin{equation}
    R_L(f)
    \geq
    k\,c^{2/(kL)},
    \quad L\geq 2.
\end{equation}
Since $c > 0$, we have $c^{2/(kL)}\to 1$ as $L\to\infty$.
Therefore, $\liminf_{L\to\infty}R_L(f)\geq k=\operatorname{rank}_\mathrm{J}(f)$.
\end{proof}
It remains open whether or not the corresponding upper bound via the bottleneck rank in \cref{eq:jacot-bounds} holds for $\limsup_{L \to \infty} R_L(f)$ for $f\in\Lip(\X;\R^D)$. It also remains open whether or not $\liminf_{L \to \infty} R_L(f) = \limsup_{L \to \infty} R_L(f)$ for $f\in\Lip(\X;\R^D)$.

\subsection{Functional Analysis of Deep Neural Networks}
We now compare the results of this section with prior function-space approaches to studying deep neural networks. There are now several frameworks that associate infinite-dimensional function spaces with deep neural architectures, including Banach spaces for multi-layer ReLU networks~\cite{wojtowytsch2020banach}, flow-induced and Barron-type spaces~\cite{ma2022barron}, approximation spaces~\cite{gribonval2022approximation}, compositional variation spaces~\cite{parhi2022kinds,parhi2026compositional,shenouda2024variation}, neural RKBSs~\cite{bartolucci2024neural}, hypothesis-space RKBSs~\cite{wang2025hypothesis}, neural Hilbert ladders~\cite{chen2023multi,chen2024neural}, and reproducing kernel chains~\cite{heeringa2025deep}. The common theme of these works is that deep neural networks can be studied through the lens of function spaces. However, they differ in their fundamental goals, which can be seen from the choice of architecture, the parameter-space regularizer, the identified function space, and/or the existence of a representer theorem.

These differences clarify which aspects of the present problem are addressed by each framework. Several existing approaches begin by identifying a function space for which the corresponding regularized data-fitting problem admits a minimizer that can be realized by a finite deep neural network. Such a result is a \emph{representer theorem for deep neural networks}. However, these representer theorems do not, by themselves, imply that solutions of the usual weight-decay-regularized parametric training problem are solutions of the associated function-space problem.

By contrast, we start from the vanilla deep ReLU parameterization and the usual weight-decay parameter cost, and identify the native space obtained by taking the $w^*$-l.s.c.\ regularization of the induced parametric representation cost. Thus, the native space is the function space induced by weight decay on the vanilla deep ReLU architecture. Moreover, under an overparameterization condition, the parameter-space and function-space problems have the same optimal value, every parameter-space minimizer induces a function-space minimizer, and at least one function-space minimizer has a finite-width realization; see \cref{thm:deep-relu-representer}.
We highlight in \cref{tab:deep-rep-thm-comparison} the precise sense in which the results of this section differ from prior work in the functional analysis of deep neural networks. 

The works of Jacot and coauthors are closest to our setup because they keep the vanilla deep ReLU architecture and use the usual weight-decay parameter cost. However, their analysis is carried out over the parametric model class $\F_\Theta$, rather than over a completed native space. The hypothesis-space construction of Wang et al.~\cite{wang2025hypothesis} is also related, since it starts from fixed-depth, fixed-width deep neural network atoms and constructs a vector-valued I-RKBS with these atoms. The resulting RKBS admits a representer theorem, but it arises directly from the I-RKBS construction and has no connection to the weight-decay parameter cost. Furthermore, as a byproduct of their construction, the architectures they considered are ``parallel'' neural networks and not given by the vanilla parameterization.
The neural Hilbert ladder (NHL) function spaces identified by Chen are quasi-Banach spaces rather than normed spaces; see~\cite[Theorem~3]{chen2024neural}. Although the NHL spaces and the native spaces $\F_L$ identified here are both quasi-Banach, they arise from different constructions. The NHL spaces are defined as infinite unions of RKHSs generated by hierarchical kernels and are endowed with a corresponding ``ladder'' complexity. In contrast, $\F_L$ is the $w^*$-l.s.c.\ completion of the vanilla depth-$L$ ReLU model class under the usual weight-decay cost. Thus, these constructions should be viewed as related but distinct, and identifying precise embeddings or equivalences between them is an interesting open problem.

\begin{table}[htb!]
\caption{Comparison of function-space results for depth-$L$ neural networks. Here $N$ denotes the number of data samples and $D$ denotes the output dimension. The column ``Vanilla Param.'' indicates whether the framework begins from the standard fully-connected ReLU neural network parameterization (e.g., without adding explicit linear layers, replacing the architecture by a kernelized construction, or considering architectures that have no parameter sharing). The column ``Weight Decay'' indicates whether the regularizer is induced by the usual weight-decay parameter cost. The column ``Function Space'' indicates whether the work identifies an associated ``completed'' function space. The column ``Rep.\ Theorem'' indicates whether a finite-dimensional representer theorem is available. The column ``Norm Based'' records whether the function-space complexity is known to be a \textit{bona fide} norm or seminorm. For this paper, the negative entry refers to \cref{prop:deep-unit-ball-nonconvex}.}

\centering
\footnotesize
\setlength{\tabcolsep}{3pt}
\renewcommand{\arraystretch}{1.20}
\noindent\makebox[\linewidth][c]{%
\resizebox{0.96\linewidth}{!}{%
\begin{tabular}{@{}
L{0.285\textwidth}
C{0.075\textwidth}
C{0.085\textwidth}
C{0.092\textwidth}
C{0.102\textwidth}
C{0.088\textwidth}
C{0.120\textwidth}
@{}}
\toprule
&
\multicolumn{4}{c}{\textbf{Desiderata}}
&
\multicolumn{2}{c}{\textbf{Complexity}}
\\
\cmidrule(lr){2-5}
\cmidrule(l){6-7}
\textbf{Framework}
&
\makecell{\textbf{Vanilla}\\\textbf{Param.}}
&
\makecell{\textbf{Weight}\\\textbf{Decay}}
&
\makecell{\textbf{Function}\\\textbf{Space}}
&
\makecell{\textbf{Rep.}\\\textbf{Theorem}}
&
\makecell{\textbf{Norm}\\\textbf{Based}}
&
\makecell{\textbf{Width}\\\textbf{Bound}}
\\
\midrule

E--Wojtowytsch~\cite{wojtowytsch2020banach}
&
\xmark
&
\xmark
&
\cmark
&
\xmark
&
\cmark
&
---
\\
\addlinespace[2pt]

Jacot et al.~\cite{JacotBNRegularity,JacotBNRank,JacotFeature}
&
\cmark
&
\cmark
&
\xmark
&
\cmark
&
---
&
$N(N+1)$
\\
\addlinespace[2pt]

Shenouda et al.~\cite{shenouda2024variation}
&
\xmark
&
\cmark
&
\cmark
&
\cmark
&
\cmark
&
$N^2$
\\
\addlinespace[2pt]

Bartolucci et al.~\cite{bartolucci2024neural}
&
\xmark
&
\xmark
&
\cmark
&
\cmark
&
\cmark
&
$O(N^L)$
\\
\addlinespace[2pt]

Wang et al.~\cite{wang2025hypothesis}
&
\xmark
&
\xmark
&
\cmark
&
\cmark
&
\cmark
&
$N D$
\\
\addlinespace[2pt]

Chen~\cite{chen2023multi,chen2024neural}
&
\xmark
&
\xmark
&
\cmark
&
\xmark
&
\xmark
&
---
\\
\addlinespace[2pt]

Heeringa et al.~\cite{heeringa2025deep}
&
\xmark
&
\xmark
&
\cmark
&
\cmark
&
\cmark
&
$N$
\\
\addlinespace[2pt]

\rowcolor{thispaper}
\textbf{This paper}
&
\cmark
&
\cmark
&
\cmark
&
\cmark
&
\xmark
&
$\boldsymbol{N^2}$
\\

\bottomrule
\end{tabular}%
}%
}
\label{tab:deep-rep-thm-comparison}
\end{table}

Within the comparison summarized in \cref{tab:deep-rep-thm-comparison}, the present work simultaneously satisfies the four listed desiderata: It keeps the vanilla fully-connected deep ReLU architecture, uses the usual weight-decay parameter cost, identifies the resulting completed native space, and yields a finite-width representer theorem with the minimizer-transfer property stated above. The rooted cost $R_L^{L/2}$ is a quasi-seminorm, and $\F_L$ is complete under the corresponding augmented quasi-norm. Under the hypotheses of \cref{prop:deep-unit-ball-nonconvex}, the unit balls of $R_L$ and $R_L^{L/2}$ coincide and are nonconvex. Consequently, $R_L^{L/2}$ is not a seminorm. This is the precise analogue of the nonconvex unit-ball geometry of $x \mapsto \|x\|_p^p$, $0<p<1$, in finite dimensions. Thus, we have revealed the following novel structural fact:
\begin{center}
    \emph{Weight decay induces a function-space complexity with a nonconvex unit ball.}
\end{center}
However, we remark that the existence of a norm equivalent to the quasi-norm $\norm{\dummy}_{\F_L}$ remains an open problem. 

\subsubsection{Approximation and Estimation with Deep Neural Networks}
There is also a related, but distinct, approximation-theoretic and statistical literature that explains why depth can be advantageous for structured target functions. Classical deep neural network approximation theory shows that ReLU networks can approximate broad smoothness and model classes with quantitative rates depending on smoothness, dimension, and network sparsity~\cite{bolcskei2019optimal,daubechies2022nonlinear,devore2021neural,elbrachter2021deep,yarotsky2017error}. For unstructured smoothness classes, these rates still reflect the curse of dimensionality. A central way around this limitation is to restrict attention to structured target classes. The principle of \emph{compositional sparsity} asserts that many high-dimensional functions can be represented as compositions of relatively few constituent functions, each depending only on a small number of variables~\cite{mhaskarLiaoPoggio2017when,mhaskarPoggio2016deep,poggio2024compositional,poggio2017why}. For such targets, deep neural architectures can achieve approximation and statistical rates governed by the local input dimensions of the constituent functions rather than by the ambient dimension~\cite{bauerKohler2019deep,dahmen2025compositional,schmidt2020nonparametric}.

Recent work has investigated this compositional viewpoint in the overparameterized regime by replacing parameter-count complexity with norm-controlled complexity. In particular,~\cite{huang2026learning} proves approximation and excess-risk bounds for sparse compositional functions represented by deep neural networks with bounded weight-decay-like parameter cost. This is close in spirit to the present paper, since it considers the overparameterized regime and measures complexity through norms of the weights rather than through the number of parameters. Nevertheless, the object of study is different. Those results give rates for approximating and learning prescribed sparse compositional target classes under norm constraints, whereas our goal is to identify the function space induced by the usual weight-decay parameter cost. These results are therefore complementary to the present work. In particular, it would be interesting to understand which function classes based on compositional sparsity are embedded in the native spaces $\F_L$. It would also be interesting to understand how the representation cost $R_L$ behaves on sparse compositional functions.

Finally, our investigation should be contrasted with approximation spaces for neural networks. Approximation spaces measure rates of approximation by model classes of increasing complexity and are often quasi-Banach spaces~\cite{gribonval2022approximation}. The native space $\F_L$, however, is not defined by an approximation rate. Thus, although both viewpoints yield quasi-Banach spaces, they capture different notions of complexity. An interesting research direction would be to understand how the depth-$L$ approximation spaces of deep neural networks compare to the depth-$L$ native spaces.

\section{Conclusion}

This paper developed the representation-cost framework for passing from parameter-space regularization to function-space regularization. This construction gives a unified perspective on several classical and modern methods. The well-known function spaces for kernel methods, wavelet methods, sparsity-promoting methods, and shallow neural networks all arise as special cases of the same abstract mechanism. In these cases, the representation cost recovers familiar Hilbert or Banach spaces and corresponding representer theorems naturally follow. The most distinctive and novel consequence appears for deep neural networks. For depth-$L$ fully connected ReLU networks with weight-decay parameter cost, we showed that $R_L^{L/2}$ is a quasi-seminorm and that the native space is complete under an associated quasi-norm. Under the hypotheses of \cref{prop:deep-unit-ball-nonconvex}, the unit balls of $R_L$, $R_L^{L/2}$, and $\norm{\dummy}_{\F_L}$ are all nonconvex. Thus, function-space complexity is nonconvex and hence not a (semi)norm. This gives an analogue of the nonconvex geometry of finite-dimensional $\ell^p$ regularization for $0 < p < 1$.

A principal open problem is to understand the depth-$L$ native spaces $\F_L$ more explicitly. While the case of shallow neural networks ($L=2$) is now well understood, much less is known for $L>2$. In particular, it remains open whether $\F_L$ admits an equivalent norm. Clarifying the structure of these quasi-Banach spaces, their relationships across depths, and their approximation-theoretic and statistical properties would provide a sharper functional-analytic understanding of deep learning.

More broadly, this paper proposed a common language to study data-fitting methods via representation costs. This allows one to study data-fitting methods at the level of functions rather than parameters. The representation-cost framework not only recovers classical native spaces, but also reveals new phenomena for deep neural networks. Thus, this viewpoint gives a systematic framework for studying the inductive bias of data-fitting methods.

\section*{Acknowledgments}
The authors used various versions of ChatGPT (OpenAI) and Claude (Anthropic) for assistance with exposition, organization, notation, and proof refinement during the preparation of this manuscript. The authors conceived the project, developed the mathematical framework, proved the results, and made all final editorial and mathematical decisions. 
The authors highlight that ChatGPT-5.6 Sol Pro provided an initial proof of \cref{prop:deep-unit-ball-nonconvex} and the interesting connection with null Lagrangians from the calculus of variations. The authors then greatly refined and rewrote the proof with the help of Claude Fable 5. In particular, Fable 5 suggested isolating \cref{lem:jacobian-minor-bound} as a standalone result. Consequently, this lemma was the missing ingredient that led the authors to prove \cref{prop:rank-lower-bound}, which they had tried before using different techniques to no avail.

\appendix

\section{Properties of Lower Semicontinuous Regularizations} \label[appendix]{app:lsc-reg}
This appendix develops the lower-semicontinuity machinery used in the paper in increasing levels of structure: First on general topological spaces, then on topological vector spaces, and finally on dual Banach spaces with their weak$^*$ and weak$^*$-compatible topologies.

\subsection{General Topological Spaces}
Throughout this subsection, $(\B,\tau)$ denotes a topological space.

\begin{definition}
A function $K:\B\to\eR$ is called \emph{$\tau$-inf-compact}\footnote{This property is sometimes called \emph{$\tau$-coercivity}. We use
$\tau$-inf-compactness to distinguish compactness of sublevel sets from
coercivity, which we formulate merely in terms of \emph{boundedness} of sublevel sets in a topological vector
space; see, e.g., \cref{def:coercive} for coercivity in Banach spaces.} if, for every $\alpha\in\R$, there exists a $\tau$-compact set $Q_\alpha\subset\B$ such that $\sls{K}{\alpha} \subset Q_\alpha$.
\end{definition}
Together with lower semicontinuity, inf-compactness yields the following existence result, which is the standard ``direct method'' in the calculus of variations; see, e.g.,~\cite[Proposition~1.2.2]{buttazzo1989semicontinuity}.

\begin{proposition} \label{prop:lsc_infcpt_give_min}
If $K:\B\to\eR$ is proper, $\tau$-l.s.c., and $\tau$-inf-compact, then $K$ admits a minimizer in $\B$, i.e., there exists $f^\star\in\B$ such that $K(f^\star)=\inf_{f\in\B}K(f)$.
\end{proposition}

\begin{proposition}\label{prop:lsc_equivalents}
Let $(\B,\tau)$ be a topological space and let $K:\B\to\eR$. Then the following are equivalent.
\begin{enumerate}[label=(\roman*)]
    \item\label{item:lsc_equivalents_lsc} $K$ is $\tau$-l.s.c.
    \item\label{item:lsc_equivalents_liminf} For every $f\in\B$,
    \begin{equation}
        K(f)
        =
        \liminf_{g\to f} K(g)
        \coloneqq
        \adjustlimits\sup_{U\in\mathfrak{N}_f}\inf_{g\in U} K(g),
    \end{equation}
    where $\mathfrak{N}_f$ denotes the $\tau$-neighborhood filter at $f$.
    \item\label{item:lsc_equivalents_epi} The epigraph
    \begin{equation}
        \epi K
        \coloneqq
        \curly{(f,\alpha)\in\B\times\R \st K(f)\leq \alpha}
    \end{equation}
    is closed in the product topology on $\B\times\R$.
    \item\label{item:lsc_equivalents_nets} For every net $(f_i)_{i\in I}$ in $\B$ such that $f_i \tauconverges f$,
    \begin{equation}
        K(f) \leq \liminf_{i\in I} K(f_i).
    \end{equation}
\end{enumerate}
\end{proposition}

\begin{proof}
The equivalence of \ref{item:lsc_equivalents_lsc} and \ref{item:lsc_equivalents_liminf} is standard; see \cite[Proposition~3.3]{dal2012introduction}. The equivalence of \ref{item:lsc_equivalents_lsc} and \ref{item:lsc_equivalents_epi} follows from \cite[Proposition~1.1.2~(i)]{buttazzo1989semicontinuity}. Finally, the equivalence of \ref{item:lsc_equivalents_lsc} and \ref{item:lsc_equivalents_nets} follows from the net characterization of closed sets; see \cite[Remark~1.1.4~(i)]{buttazzo1989semicontinuity}.
\end{proof}

\begin{proposition}\label{prop:lsc_properties}
Let $(\B,\tau)$ be a topological space. Then the following properties hold.
\begin{enumerate}[label=(\roman*)]
    \item\label{item:sum_lsc} If $K,J:\B\to(-\infty,+\infty]$ are $\tau$-l.s.c., then $K+J$ is $\tau$-l.s.c.
    
    \item\label{item:comp_with_nondec_lsc} If $K:\B\to\eR$ is $\tau$-l.s.c.\ and $\gamma:\eR\to\eR$ is left-continuous and nondecreasing, then $\gamma\circ K$ is $\tau$-l.s.c.
    
    \item\label{item:comp_with_continuous_map_lsc} Let $(\mathcal C,\sigma)$ be another topological space. If $K:\mathcal C\to\eR$ is $\sigma$-l.s.c.\ and $J:\B\to\mathcal C$ is $\tau$-$\sigma$-continuous, then $K\circ J$ is $\tau$-l.s.c.
    
    \item\label{item:sup_lsc} If $(K_i)_{i\in I}$ is a family of $\tau$-l.s.c.\ functions, then
    \begin{equation}
        K(f) \coloneqq \sup_{i\in I} K_i(f)
    \end{equation}
    is $\tau$-l.s.c.
\end{enumerate}
\end{proposition}

\begin{proof}\leavevmode

\begin{enumerate}[label=(\roman*)]
    \item
    This is standard; see \cite[Proposition~1.1.2(iii)]{buttazzo1989semicontinuity}.

    \item
    Fix $\alpha\in\R$. Since $\gamma$ is nondecreasing and left-continuous, the sublevel set
    \begin{equation}
        \curly{t\in\eR \st \gamma(t)\leq \alpha}
    \end{equation}
    is a closed lower interval in $\eR$. Hence it is either empty, all of $\eR$, or of the form $[-\infty,\beta]$ for some $\beta\in\eR$. Therefore,
    \begin{equation}
        \curly{f\in\B \st \gamma\circ K(f)\leq \alpha}
    \end{equation}
    is either empty, all of $\B$, or equal to $\curly{f\in\B \st K(f)\leq \beta}$. In all cases it is $\tau$-closed, since $K$ is $\tau$-l.s.c. Thus $\gamma\circ K$ is $\tau$-l.s.c.

    \item
    For any $\alpha\in\R$,
    \begin{equation}
        \curly{f\in\B \st K\circ J(f)\leq \alpha}
        =
        J^{-1}\paren*{\curly{y\in\mathcal C \st K(y)\leq \alpha}}.
    \end{equation}
    The set inside the preimage is $\sigma$-closed since $K$ is $\sigma$-l.s.c., and its preimage under the $\tau$-$\sigma$-continuous map $J$ is $\tau$-closed. Hence $K\circ J$ is $\tau$-l.s.c.

    \item
    This is standard; see \cite[Proposition~1.1.2(ii)]{buttazzo1989semicontinuity}. \qedhere
\end{enumerate}
\end{proof}

\begin{proposition}\label{prop:lscreg_equivalents}
Let $(\B,\tau)$ be a topological space, let $K:\B\to\eR$, and let $\overline K$ denote the $\tau$-l.s.c.\ regularization of $K$. Then the following hold.
\begin{enumerate}[label=(\roman*)]
    \item\label{item:lsc-reg-filter-def} For every $f\in\B$,
    \begin{equation}
        \overline K(f)
        =
        \liminf_{g\to f} K(g).
    \end{equation}

    \item\label{item:lsc-reg-epi-def} We have
    \begin{equation}
        \epi \overline K
        =
        \overline{\epi K},
    \end{equation}
    where the right-hand side denotes the closure of $\epi K$ in the product topology on $\B\times\R$.

    \item\label{item:lsc-reg-net-def} For every $f\in\B$,
    \begin{equation}
        \overline K(f)
        =
        \min\curly[\Big]{
            \liminf_{i\in I} K(f_i)
            \st
            (f_i)_{i\in I} \subset \B \text{ is a net and } f_i \tauconverges f
        }.
    \end{equation}

    \item\label{item:lsc_in_metric_space} If $\tau$ is induced by a metric $d$ on $\B$, then, for every $f\in\B$,
    \begin{equation}
        \overline K(f)
        =
        \seqreg{K}(f)
        =
        \lim_{\varepsilon\to0^+}
        \inf\curly[\Big]{
            K(g)
            \st
            g\in\dom(K),\ d(f,g)\leq \varepsilon
        }.
    \end{equation}
\end{enumerate}
\end{proposition}

\begin{proof}\leavevmode
\begin{enumerate}[label=(\roman*)]
    \item
    This is standard; see \cite[Corollary~2.1]{ekeland1999convex}.

    \item
    This is standard; see \cite[Corollary~2.1]{ekeland1999convex}.

    \item
    This follows from the net characterization of the l.s.c.\ regularization; see \cite[Proposition~1.3.1(ii)]{buttazzo1989semicontinuity}.

    \item
    Fix $f\in\B$. Since $\tau$ is metrizable, it is first-countable, and so $\overline K(f)=\seqreg K(f)$ by \cite[Proposition~1.3.3]{buttazzo1989semicontinuity}. Define
    \begin{equation}
        L(f)
        \coloneqq
        \lim_{\varepsilon\to0^+}
        \inf\curly[\Big]{
            K(g)
            \st
            g\in\dom(K),\ d(f,g)\leq \varepsilon
        }.
    \end{equation}
    We now show that $\seqreg K(f)=L(f)$.

    If $L(f)=+\infty$, then $\seqreg K(f)\leq L(f)$ is trivial. If $L(f)=-\infty$, then the neighborhood infimum in the definition of $L(f)$ equals $-\infty$ for every sufficiently small radius. Hence, for each $n\in\N$, one may choose $f_n\in\dom(K)$ such that $d(f_n,f)\leq1/n$ and $K(f_n)\leq-n$. Then $f_n\tauconverges f$ and $\seqreg K(f)=-\infty=L(f)$. Finally, suppose $L(f)\in\R$. For each $n\in\N$, choose $f_n\in\dom(K)$ such that $d(f_n,f)\leq 1/n$ and
    \begin{equation}
        K(f_n)
        \leq
        \inf\curly[\Big]{
            K(g)
            \st
            g\in\dom(K),\ d(f,g)\leq \frac{1}{n}
        }
        +
        \frac{1}{n}.
    \end{equation}
    Then $f_n\tauconverges f$, and taking the $\liminf$ of both sides above gives $\seqreg K(f) \leq L(f)$.

    Conversely, suppose $\seqreg K(f)<+\infty$, since otherwise the inequality $L(f)\leq \seqreg K(f)$ is trivial. Let $(f_n)_{n\in\N}\subset\B$ be any sequence such that $f_n\tauconverges f$ and $\liminf_{n\to\infty}K(f_n)<+\infty$. For any $\varepsilon>0$, we have $d(f_n,f)\leq\varepsilon$ for all sufficiently large $n$. Therefore,
    \begin{equation}
        \inf\curly[\Big]{
            K(g)
            \st
            g\in\dom(K),\ d(f,g)\leq \varepsilon
        }
        \leq
        \liminf_{n\to\infty}K(f_n).
    \end{equation}
    Letting $\varepsilon\to0^+$ and then taking the infimum over all such sequences yields $L(f)\leq \seqreg K(f)$. Hence $\overline K(f)=\seqreg K(f)=L(f)$. \qedhere
\end{enumerate}
\end{proof}

\begin{proposition}\label{prop:lscreg_properties}
Let $(\B,\tau)$ be a Hausdorff topological space, let $K:\B\to\eR$, and let $\overline K$ denote the $\tau$-l.s.c.\ regularization of $K$. Then the following properties hold.
\begin{enumerate}[label=(\roman*)]
    \item\label{item:inf-preserve} We have
    \begin{equation}
        \inf_{f\in\B} K(f) = \inf_{f\in\B} \overline K(f).
    \end{equation}

    \item\label{item:lsc-reg-cont-sum} If $J:\B\to\R$ is $\tau$-continuous, then
    \begin{equation}
        \overline{J+K} = J+\overline K.
    \end{equation}

    \item\label{item:inf-compact-transfer} If $K$ is $\tau$-inf-compact, then $\overline K$ is $\tau$-inf-compact.

    \item\label{item:cont-compose} If $\gamma:\eR\to\eR$ is continuous and nondecreasing, then
    \begin{equation}
        \overline{\gamma\circ K} = \gamma\circ\overline K.
    \end{equation}
\end{enumerate}
\end{proposition}

\begin{proof} \leavevmode
\begin{enumerate}[label=(\roman*)]
    \item
    This is standard; see \cite[Proposition~1.3.1(iii)]{buttazzo1989semicontinuity}.

    \item
    This is standard; see \cite[Proposition~1.3.1(v)]{buttazzo1989semicontinuity}.

    \item
    This is standard; see \cite[Proposition~1.3.1(iv)]{buttazzo1989semicontinuity}.

    \item Fix $f\in\B$. By
\cref{prop:lscreg_equivalents}~\ref{item:lsc-reg-filter-def} and the fact that
continuous nondecreasing functions commute with liminf along filters,
\begin{equation}
    \overline{\gamma\circ K}(f)
    =
    \liminf_{g\to f}\gamma(K(g))
    =
    \gamma\left(\liminf_{g\to f}K(g)\right)
    =
    \gamma(\overline K(f)),
\end{equation}
 which proves the claim. \qedhere
\end{enumerate}
\end{proof}

\subsection{Sequential Lower Semicontinuous Regularization}
We now compare the topological and sequential regularizations in a general topological space. The topology is still denoted by $\tau$. In general we have $\overline{K} \leq \seqreg{K} \leq K$, and the two regularizations need not coincide. However, they do agree under additional assumptions on the topology or the function. The simplest such condition is first-countability, i.e., if $(\B,\tau)$ is first-countable, then $\overline{K} = \seqreg{K}$ for all $K:\B\rightarrow\eR$~\cite[Proposition~1.3.3]{buttazzo1989semicontinuity}.
Since every metrizable space is first-countable, this applies in particular whenever $\tau$ is induced by a metric. However, in the infinite-dimensional settings considered in this paper, the relevant topologies are typically not metrizable and may fail to be first-countable. A condition better suited to our setting is the following, which exploits the metrizability of $\tau$-compact subsets.
\begin{proposition}[{\cite[Proposition~1.3.5]{buttazzo1989semicontinuity}}]\label{prop:lscreg_seqlscreg_coincide}
If every $\tau$-compact subset of $\B$ is metrizable and $K:\B\rightarrow\overline{\R}$ is $\tau$-inf-compact, then $\overline{K}^{\,\tau} = \overline{K}^{\mathrm{seq},\tau}$.
\end{proposition}

\begin{remark}\label{rmk:seq_lsc_reg_liminf_to_min}
Under the same assumptions of \cref{prop:lscreg_seqlscreg_coincide}, 
the infimum in \cref{eq:seq_lsc_reg_def} is attained for a sequence belonging to $\dom(K)$ having a \textit{bona fide} limit, i.e., we have
\begin{align}\label{eq:seq_lsc_reg_def_v2}
    \overline{K}^{\,\tau}(f) = \overline{K}^{\mathrm{seq},\tau}(f) = \min\left\{
        \alpha \in \eR
        \st
        (f_n)_{n\in\N}\subset\dom(K),\ f_n \tauconverges f,\ K(f_n) \rightarrow \alpha
    \right\}
\end{align}
whenever the infimum is $< +\infty$. This attainment assertion is proved in the next lemma.
\end{remark}

\begin{lemma}\label{lem:equiv_of_seq_lsc_reg_defs}
Let $(\B,\tau)$ be a topological space and let $K:\B\to\eR$. Suppose that
$K$ is $\tau$-inf-compact and that every $\tau$-compact subset of $\B$ is
metrizable. Then, for every $f\in\B$,
\begin{equation}\label{eq:equiv_of_seq_lsc_reg_defs}
\seqreg{K}(f)
=
\min\left\{
\alpha\in\eR
\st
(f_n)_{n\in\N}\subset\dom(K),\
f_n\tauconverges f,\
K(f_n)\to\alpha
\right\}
\end{equation}
whenever $\seqreg{K}(f) < +\infty$.
\end{lemma}

\begin{proof}
Fix $f\in\B$ and let $S(f)$ denote the set on the right-hand side of \cref{eq:equiv_of_seq_lsc_reg_defs}. If $\alpha\in S(f)$, then there exists $(f_n)_{n\in\N}\subset\dom(K)$ such that $f_n\tauconverges f$ and $K(f_n)\to\alpha$. This sequence is admissible in the definition of $\seqreg{K}(f)$, and therefore
\begin{equation}
\seqreg{K}(f) \leq \liminf_{n\to\infty}K(f_n) = \alpha.
\end{equation}
Taking the infimum over $\alpha\in S(f)$ gives $\seqreg{K}(f)\le \inf S(f)$, with the convention $\inf\varnothing=+\infty$.

It remains to show that this infimum is attained. If $\seqreg{K}(f)=+\infty$, the claim is immediate.
Assume $\seqreg{K}(f)<+\infty$. Choose $(\alpha_j)_{j\in\N}\subset\R$ with $\alpha_j\downarrow \seqreg{K}(f)$. By $\tau$-inf-compactness, there exists a $\tau$-compact set $Q$ such that $\sls{K}{\alpha_1}\subset Q$. Then $Q^\star\coloneqq Q\cup\{f\}$ is $\tau$-compact and hence metrizable. Let $d$ be a metric that induces $\tau_{Q^\star}$ (subspace topology).

For each $j$, since $\seqreg{K}(f)<\alpha_j$, there exists a sequence $(f_n^{(j)})_{n\in\N}\subset\B$ with $f_n^{(j)}\tauconverges f$ and $\liminf_{n\to\infty}K(f_n^{(j)})<\alpha_j$. Thus $K(f_n^{(j)})<\alpha_j$ for infinitely many $n$. Along these indices, $f_n^{(j)}\in\sls{K}{\alpha_1}\subset Q^\star$, and the corresponding subsequence still converges to $f$. Hence we may choose $n(j)$ such that, with $g_j\coloneqq f_{n(j)}^{(j)}$,
\begin{equation}
d(g_j,f)<\frac1j
\quad\text{and}\quad
K(g_j)<\alpha_j .
\end{equation}
Then $(g_j)_{j\in\N}\subset\dom(K)$, $g_j\tauconverges f$, and
\begin{equation}
\seqreg{K}(f)
\le
\liminf_{j\to\infty}K(g_j)
\le
\limsup_{j\to\infty}K(g_j)
\le
\lim_{j\to\infty}\alpha_j
=
\seqreg{K}(f).
\end{equation}
Therefore $K(g_j)\to\seqreg{K}(f)$, so $\seqreg{K}(f)\in S(f)$.
\end{proof}

\subsection{Topological Vector Spaces}
We now assume that $(\B,\tau)$ is a TVS. The algebraic structure allows subadditivity and homogeneity to pass to the l.s.c.\ regularization.

\begin{proposition}\label{prop:lscreg_preserves_properties}
Let $(\B,\tau)$ be a TVS, let $K:\B\to[0,+\infty]$, and let $\overline K$ denote the $\tau$-l.s.c.\ regularization of $K$.
Then the following properties hold.
\begin{enumerate}[label=(\roman*)]
    \item\label{item:subadd-lsc-reg} If $K$ is subadditive, then so is $\overline K$.
    \item\label{item:homo-lsc-reg} If $K$ is absolutely $p$-homogeneous for some $p>0$, then so is $\overline K$.
\end{enumerate}
\end{proposition}

\begin{proof} \leavevmode
\begin{enumerate}[label=(\roman*)]
    \item Fix $g\in\B$. We first prove that
\begin{equation}
    \overline K(f+g)\leq \overline K(f)+K(g),
    \quad f\in\B.
\end{equation}
If $K(g)=+\infty$, this is trivial. Hence assume $K(g)<+\infty$. Since
$\overline K\leq K$ and $K$ is subadditive, the function
\begin{equation}
    H(f)\coloneqq \overline K(f+g)-K(g)
\end{equation}
satisfies
\begin{equation}
    H(f)\leq K(f+g)-K(g)\leq K(f),
    \quad f\in\B.
\end{equation}
Moreover, $H$ is $\tau$-l.s.c.\ since translations are $\tau$-$\tau$-continuous. Hence, by the defining property of l.s.c.\ regularization,
$H\leq \overline K$. Therefore,
\begin{equation} \label{eq:djsakl}
    \overline K(f+g)\leq \overline K(f)+K(g),
    \quad f\in\B.
\end{equation}

Now fix $f\in\B$. If $\overline K(f)=+\infty$, the desired inequality is
trivial. Hence assume $\overline K(f)<+\infty$. The function
\begin{equation}
    J(g)\coloneqq \overline K(f+g)-\overline K(f)
\end{equation}
is $\tau$-l.s.c. Moreover, by \cref{eq:djsakl},
\begin{equation}
    J(g)\leq K(g),
    \quad g\in\B.
\end{equation}
Thus, by the defining property of l.s.c.\ regularization, $J\leq \overline K$.
Equivalently,
\begin{equation}
    \overline K(f+g)\leq \overline K(f)+\overline K(g),
    \quad f,g\in\B.
\end{equation}
This proves that $\overline K$ is subadditive.

    \item
    Since $\overline K\leq K$ and $K$ is absolutely $p$-homogeneous,
    \begin{equation}
        0\leq \overline K(0)\leq K(0)=0,
    \end{equation}
    so $\overline K(0)=0$. Now fix $\alpha\in\R\setminus\{0\}$. The function
    \begin{equation}
        H(f)\coloneqq |\alpha|^{-p}\overline K(\alpha f)
    \end{equation}
    is $\tau$-l.s.c.\ and satisfies
    \begin{equation}
        H(f)\leq |\alpha|^{-p}K(\alpha f)=K(f).
    \end{equation}
    Hence, by the defining property of l.s.c.\ regularization, $H\leq\overline K$, i.e.,
    \begin{equation}
        \overline K(\alpha f)\leq |\alpha|^p\overline K(f),
        \quad f\in\B.
    \end{equation}
    Applying this inequality with $\alpha^{-1}$ in place of $\alpha$ gives the reverse inequality. Therefore,
    \begin{equation}
        \overline K(\alpha f)=|\alpha|^p\overline K(f),
        \quad \alpha\neq 0,\ f\in\B.
    \end{equation}
    This proves that $\overline K$ is absolutely $p$-homogeneous. \qedhere
\end{enumerate}
\end{proof}

\subsection{Dual Banach Spaces and \texorpdfstring{Weak$^*$-Compatible}{WeakStar-Compatible} Topologies}
We now specialize the compactness discussion to dual Banach spaces. Let $(\B,\norm{\dummy}_{\B})$ be a Banach space with predual $(\A,\norm{\dummy}_{\A})$, identify $\B$ isometrically with $\A'$, and write $\ang{\mu,f}$ for the dual pairing. Thus,
\begin{equation}
\norm{f}_{\B} = \sup_{\substack{\mu\in\A\\ \norm{\mu}_{\A}\leq 1}} \abs{\ang{\mu,f}}.
\end{equation}
We endow $\B$ with the weak$^*$ topology $w^*=\sigma(\B,\A)$. This is the setting used throughout the paper for the ambient function spaces. The main idea is that weak$^*$ compactness converts boundedness into compactness by the Banach--Alaoglu theorem. Thus, in dual Banach spaces, coercivity (i.e., boundedness of sublevel sets) is equivalent to inf-compactness, which makes the direct method of the calculus of variations applicable. Recall from \cref{def:weakstar_compatible} that a topology is weak$^*$-compatible when it agrees with $w^*$ on every norm-bounded subset.

\begin{proposition}\label{prop:coercive_implies_infcpt}
A function $K:\B\to\eR$ is coercive if and only if it is $w^*$-inf-compact.
\end{proposition}
\begin{proof}
Suppose first that $K$ is coercive. Then, for every $\alpha\in\R$, there exists $\beta>0$ such that $\sls{K}{\alpha} \subset \sls{\norm{\dummy}_\B}{\beta}$. The latter set is $w^*$-compact by the Banach--Alaoglu theorem, so $K$ is $w^*$-inf-compact. Conversely, if $K$ is $w^*$-inf-compact, then every sublevel set is contained in a $w^*$-compact set, and every $w^*$-compact subset of $\B$ is norm-bounded. Therefore $K$ is coercive.
\end{proof}

Since inf-compactness is preserved under l.s.c.\ regularization (see \Cref{prop:lscreg_properties}~\ref{item:inf-compact-transfer}), the previous result implies the same is true of coercivity.
\begin{proposition}\label{prop:lscreg_is_coercive}
If $K:\B\rightarrow\eR$ is coercive then so is its $w^*$-l.s.c.\ regularization $\overline{K}$.
\end{proposition}

When the predual $\A$ is separable (i.e., has a countable dense subset), coercivity also allows us to replace the topological l.s.c.\ regularization with its sequential counterpart, which is typically easier to work with. The key ingredient is that separability of $\A$ ensures that the topology $w^*$ on $\B$ is metrizable when restricted to any norm-bounded subset,\footnote{See \cite[Theorem~3.28]{brezis2011functional} for the case of the unit ball in $\B$. The general result follows by a rescaling argument.} and hence on any $w^*$-compact subset.

\begin{proposition}\label{prop:coercive_implies_seqlscreg}
Suppose $\A$ is separable. If $K:\B\rightarrow \eR$ is coercive, then for all $f\in \B$
\begin{equation}
\overline{K}^{\,w^*}(f) = \overline{K}^{\mathrm{seq},w^*}(f) = \min\{\alpha \in \eR \st (f_n)_{n\in\N}\subset \dom(K),\ f_n\starconverges f,\ K(f_n) \rightarrow \alpha\}.
\end{equation}
\end{proposition}
\begin{proof}
Since $\A$ is separable, every $w^*$-compact subset of $\B$ is metrizable. Since $K$ is coercive, \cref{prop:coercive_implies_infcpt} implies that $K$ is $w^*$-inf-compact. The result then follows from \cref{prop:lscreg_seqlscreg_coincide} and \cref{lem:equiv_of_seq_lsc_reg_defs}.
\end{proof}

The following proposition shows that, for sequences, weak$^*$ convergence is
equivalent to convergence in any weak$^*$-compatible topology, provided the sequence is norm-bounded.
\begin{proposition}\label{lem:seq_star_converge_equiv}
Suppose $\tau$ is a weak$^*$-compatible topology on $\B$. Let $(f_n)_{n\in\N}\cup\{f\}\subset \B$. Then $f_n\starconverges f$ if and only if $f_n\altconverges f$ and $\sup_{n\in\N} \|f_n\|_\B < +\infty$.
\end{proposition}
\begin{proof}
If $\sup_{n\in\N} \|f_n\|_{\B} < +\infty$, then  $(f_n)_{n\in\N}\cup \{f\}$ belongs to a norm-bounded set, and so $f_n\altconverges f$ implies $f_n\starconverges f$. Conversely, if $f_n\starconverges f$, then for all $\mu \in \mathcal{A}$ we have $\langle \mu, f_n\rangle \rightarrow \langle \mu, f\rangle$ and hence $\sup_{n\in \N} |\langle \mu, f_n\rangle| < + \infty$. Therefore, identifying each $f_n$ as a continuous linear functional over $\mathcal{A}$, the family $(f_n)_{n\in\N}$ is pointwise uniformly bounded. Thus, by the uniform boundedness principle, the family $(f_n)_{n\in\N}$ is bounded in norm, i.e., $\sup_{n\in\N} \|f_n\|_{\B} < + \infty$. Therefore, the sequence $(f_n)_{n\in\N}$ is contained in a norm-bounded set, and so $f_n\altconverges f$.
\end{proof}

As a consequence of the previous result, the sequential l.s.c.\ regularization of a coercive function is the same under any weak$^*$-compatible topology.

\begin{proposition}\label{prop:seq_lsc_reg_weakstar_compatible} Suppose $\tau$ is a weak$^*$-compatible topology on $\B$. If $K:\B\rightarrow \eR$ is coercive, then $\overline{K}^{\mathrm{seq},w^*} = \overline{K}^{\mathrm{seq},\tau}$.
\end{proposition}
\begin{proof}
By \cref{lem:seq_star_converge_equiv}, every $w^*$-convergent sequence is norm-bounded and therefore $\tau$-convergent (since $\tau$ and
$w^*$ coincide on norm-bounded sets). Hence every sequence admissible in
the definition of $\overline{K}^{\mathrm{seq},w^*}$ is also admissible for
$\overline{K}^{\mathrm{seq},\tau}$, which gives $\overline{K}^{\mathrm{seq},\tau} \leq \overline{K}^{\mathrm{seq},w^*}$. For
the reverse inequality, fix any $f\in \B$. The case $\overline{K}^{\mathrm{seq},\tau}(f) = +\infty$ is trivial, so suppose $\overline{K}^{\mathrm{seq},\tau}(f) < +\infty$ and let
$(f_n)_{n\in\N}$ satisfy $f_n \tauconverges f$ with
$\alpha \coloneqq \liminf_{n\to\infty} K(f_n) < +\infty$. Passing to a
subsequence (that we do not relabel), we may assume $K(f_n) \to \alpha$. After discarding finitely many terms,
$\sup_{n\in\N} K(f_n) < +\infty$, and coercivity gives
$\sup_{n\in\N} \|f_n\|_\B < +\infty$. By
\cref{lem:seq_star_converge_equiv}, $f_n \starconverges f$, so
$\overline{K}^{\mathrm{seq},w^*}(f) \leq \alpha$. Taking the infimum over all such sequences
yields $\overline{K}^{\mathrm{seq},w^*}(f) \leq \overline{K}^{\mathrm{seq},\tau}(f)$.
\end{proof}

\begin{proposition}\label{prop:coercive_implies_seqlscreg_compatible}
Suppose $\A$ is separable, 
and $\tau$ is a weak$^*$-compatible topology on $\B$. If $K:\B\rightarrow \eR$ is coercive, then
\begin{equation}
\overline{K}^{\,w^*}(f) = \overline{K}^{\mathrm{seq},\tau}(f) = \min\{\alpha \in \eR \st (f_n)_{n\in\N}\subset \dom(K),\ f_n\tauconverges f,\ K(f_n) \rightarrow \alpha\}.
\end{equation}
\end{proposition}

\begin{proof}
By \cref{prop:coercive_implies_seqlscreg,prop:seq_lsc_reg_weakstar_compatible}, we have $\overline{K}^{\,w^*} = \overline{K}^{\mathrm{seq},w^*} = \overline{K}^{\mathrm{seq},\tau}$. It remains to justify the minimum representation. Fix $f\in\B$. If $\overline{K}^{\,w^*}(f)=+\infty$, then the admissible set contains no finite value, so its minimum is $+\infty$ under the convention that $\min\varnothing =+\infty$. Hence, suppose that $\overline{K}^{\,w^*}(f)<+\infty$. By \cref{prop:coercive_implies_seqlscreg}, there exists $(f_n)_{n\in\N}\subset\dom(K)$ such that $f_n\starconverges f$ and $K(f_n)\to \overline{K}^{\,w^*}(f)$. Since $w^*$-convergent sequences are norm-bounded, weak$^*$-compatibility of $\tau$ implies $f_n\tauconverges f$. Thus $(f_n)_{n\in\N}$ is admissible for $\overline{K}^{\mathrm{seq},\tau}(f)$, and the equality $\overline{K}^{\,w^*}(f)=\overline{K}^{\mathrm{seq},\tau}(f)$ shows that it attains the infimum.
\end{proof}

\section{Coercivity Modulo Subspaces}
\label[appendix]{app:coercive-mod-subspace-proof}
This appendix supplies the proofs and supporting results for the two structural statements from \cref{sec:background}. We first work in a general topological vector space and prove the factorization lemma. We then pass to dual Banach spaces, develop equivalent formulations of coercivity modulo a subspace, and prove the recovery-sequence proposition.

\subsection{Constancy Spaces and Factorization in Topological Vector Spaces}\label{subsec:constancy-factorization}

\begin{lemma}\label{lem:nesting_of_const_spaces}
Let $(\B,\tau)$ be a TVS and let $K:\B\to\eR$. Then, $\const(K) \subset \const(\overline{K}^{\,\tau})$, where $\overline{K}^{\,\tau}$ denotes the $\tau$-l.s.c.\ regularization of $K$.
\end{lemma}
\begin{proof}
Fix $g\in\const(K)$ and $\alpha\in\R$, and let $\TOp_{\alpha g}:\B\to\B$ denote the translation map $\TOp_{\alpha g}f \coloneqq f+\alpha g$. Since $g\in\const(K)$, we have that $K\circ \TOp_{\alpha g}=K$. Moreover, $\TOp_{\alpha g}: \B \to \B$ is $\tau$-$\tau$-continuous. This shows $\overline{K}^{\,\tau} \circ \TOp_{\alpha g}$ is $\tau$-l.s.c. Since $\overline{K}^{\,\tau}\leq K$, we also have $\overline{K}^{\,\tau} \circ \TOp_{\alpha g} \leq K\circ \TOp_{\alpha g} = K$. Therefore, by the defining property of l.s.c.\ regularization, we have $\overline{K}^{\,\tau}\circ \TOp_{\alpha g}\leq \overline{K}^{\,\tau}$, or equivalently, $\overline{K}^{\,\tau}(f+\alpha g)\leq \overline{K}^{\,\tau}(f)$ for all $f\in\B$. Applying the same argument with $-\alpha$ gives $\overline{K}^{\,\tau}(f-\alpha g)\leq \overline{K}^{\,\tau}(f)$ for all $f\in \B$. Replacing $f$ by $f + \alpha g$ in this inequality yields $\overline{K}^{\,\tau}(f)\leq \overline{K}^{\,\tau}(f+\alpha g)$. Thus, we have that $\overline{K}^{\,\tau}(f+\alpha g)=\overline{K}^{\,\tau}(f)$ for all $f\in\B$ and $\alpha\in\R$, i.e., $g\in\const(\overline{K}^{\,\tau})$.
\end{proof}

\subsubsection{Proof of \cref{lem:chatacterizing_R_mod_N}}\label[appendix]{subsec:factorization-proof}

\begin{proof}
Since $\M$ is a $\tau$-complement of $\Null \subset \const(K)$, we have that $K = K|_\M \circ \P_\M$. Furthermore, by \cref{lem:nesting_of_const_spaces}, we have that $\const(K)\subset \const(\overline{K}^{\,\tau})$. Hence, we also have that $\overline{K}^{\,\tau}=\left.\overline{K}^{\,\tau}\right|_\M \circ \P_\M$. Since $\P_\M$ is $\tau$-$\tau_\M$-continuous and $\overline{(K|_\M)}^{\,\tau_\M}$ is $\tau_\M$-l.s.c.\ by definition, we see that $\overline{(K|_\M)}^{\,\tau_\M}\circ \P_\M$ is $\tau$-l.s.c. Immediately from definitions, we have that $\overline{(K|_\M)}^{\,\tau_\M}\circ \P_\M \leq K|_\M\circ \P_\M = K$. By the defining property of l.s.c.\ regularization, we have that
\begin{equation}
    \overline{(K|_\M)}^{\,\tau_\M}\circ \P_\M \leq \overline{K}^{\,\tau}. \label{eq:factorize-leq}
\end{equation}
Next, since $\overline{K}^{\,\tau}$ is $\tau$-l.s.c., $\left.\overline{K}^{\,\tau}\right|_\M$ is $\tau_\M$-l.s.c. It also holds that $\left.\overline{K}^{\,\tau}\right|_\M \leq K|_\M$ since $\overline{K}^{\,\tau}$ is the $\tau$-l.s.c.\ regularization of $K$. Again, by the defining property of l.s.c.\ regularization, we have $\left.\overline{K}^{\,\tau}\right|_\M \leq \overline{(K|_\M)}^{\,\tau_\M}$. This implies
\begin{equation}
    \overline{K}^{\,\tau} = \left.\overline{K}^{\,\tau}\right|_\M\circ \P_\M \leq \overline{(K|_\M)}^{\,\tau_\M}\circ \P_\M. \label{eq:factorize-geq}
\end{equation}
Combining \cref{eq:factorize-leq,eq:factorize-geq} completes the proof.
\end{proof}

\subsection{Coercivity Modulo Subspaces in Dual Banach Spaces}\label{subsec:coercivity-modulo-subspaces}
Let $\B=\A'$ be a dual Banach space, endowed with the weak$^*$ topology $w^*=\sigma(\B,\A)$ introduced in \cref{sec:background}. The next proposition collects equivalent formulations of coercivity modulo a subspace under progressively stronger assumptions on the subspace.

\begin{proposition}\label{prop:coercive-mod-N-equivalences}
Let $K:\B\to\eR$ and let $\Null\subset\B$ be a subspace. Consider the following statements.
\begin{enumerate}[label=(\roman*)]
    \item\label{item:coercive-mod-N} $K$ is coercive modulo $\Null$.
    
    \item\label{item:coercive-quotient} For every $\alpha\in\R$, there exists $\beta>0$ such that
    \begin{equation}
        K(f)\leq \alpha
        \quad\Rightarrow\quad
        \norm{[f]}_{\B/\Null}\leq \beta,
    \end{equation}
    where $\norm{\dummy}_{\B/\Null}$ denotes the quotient norm on $\B/\Null$.
    
    \item\label{item:coercive-projector} For every $\alpha\in\R$, there exists $\beta>0$ such that
    \begin{equation}
        K(f)\leq \alpha
        \quad\Rightarrow\quad
        \norm{\P_\M f}_\B\leq \beta,
    \end{equation}
    where $\P_\M=\Id-\P_\Null$ is associated with a $w^*$-complement $\M$ of $\Null$.
    
    \item\label{item:coercive-restriction} The restriction $K|_\M:\M\to\eR$ is coercive with respect to $\norm{\dummy}_\M=\norm{\dummy}_\B$.
\end{enumerate}
Then the following hold.
\begin{enumerate}[label=(\alph*)]
    \item\label{item:coerive_in_quotient} If $\Null$ is norm-closed, then \ref{item:coercive-mod-N} $\Leftrightarrow$ \ref{item:coercive-quotient}.
    
    \item\label{item:coercive_projector_equiv} If $\Null$ is $w^*$-closed and admits a $w^*$-complement $\M$, then
    \ref{item:coercive-mod-N} $\Leftrightarrow$ \ref{item:coercive-quotient} $\Leftrightarrow$ \ref{item:coercive-projector}.
    
    \item\label{item:coercive_on_M} If, in addition to the hypotheses of \ref{item:coercive_projector_equiv}, $\Null\subset\const(K)$, then
    \ref{item:coercive-mod-N} $\Leftrightarrow$ \ref{item:coercive-quotient} $\Leftrightarrow$ \ref{item:coercive-projector} $\Leftrightarrow$ \ref{item:coercive-restriction}.
\end{enumerate}
\end{proposition}

\begin{proof} \leavevmode
\begin{enumerate}[label=(\alph*)]
    \item
    If $\Null$ is norm-closed, then $\B/\Null$ is a Banach space with quotient norm $\norm{[f]}_{\B/\Null} = \inf_{g\in\Null}\norm{f+g}_\B$. Hence \ref{item:coercive-mod-N} and \ref{item:coercive-quotient} are simply restatements of one another.

    \item
    By \ref{item:coerive_in_quotient}, it remains to compare the quotient norm with $\norm{\P_\M f}_\B$. Since $\P_\M f=f-\P_\Null f$ and $\P_\Null f\in\Null$,
    \begin{equation}
        \inf_{g\in\Null}\norm{f+g}_\B
        \leq
        \norm{f-\P_\Null f}_\B
        =
        \norm{\P_\M f}_\B.
    \end{equation}
    Conversely, since $\P_\M g=0$ for all $g\in\Null$,
    \begin{equation}
        \norm{\P_\M f}_\B
        =
        \inf_{g\in\Null}\norm{\P_\M(f+g)}_\B
        \leq
        \norm{\P_\M}_{\op}
        \inf_{g\in\Null}\norm{f+g}_\B.
    \end{equation}
    Here $\norm{\P_\M}_{\op}<+\infty$ since $w^*$-$w^*$-continuous linear operators are norm-norm continuous. Thus $\norm{[f]}_{\B/\Null}$ and $\norm{\P_\M f}_\B$ are equivalent norms, which proves the equivalence of \ref{item:coercive-quotient} and \ref{item:coercive-projector}.

    \item
    Assume \ref{item:coercive-projector}. If $f\in\M$ and $K(f)\leq\alpha$, then $\P_\M f=f$, so $\norm{f}_\B\leq\beta$. Hence $K|_\M$ is coercive, which proves \ref{item:coercive-restriction}. Conversely, assume \ref{item:coercive-restriction} and let $f\in\B$ satisfy $K(f)\leq\alpha$. Since $\Null\subset\const(K)$,
    \begin{equation}
        K(\P_\M f)
        =
        K(\P_\M f+\P_\Null f)
        =
        K(f)
        \leq
        \alpha.
    \end{equation}
    Since $\P_\M f\in\M$, coercivity of $K|_\M$ gives $\norm{\P_\M f}_\B\leq\beta$. This is \ref{item:coercive-projector}. \qedhere
\end{enumerate}
\end{proof}

The next result shows that coercivity modulo a subspace of the constancy space forces that subspace to be the \emph{entire} constancy space.

\begin{proposition}\label{prop:coercive_mod_determines_constancy}
If $K:\B\rightarrow \overline{\R}$ with $K\not\equiv+\infty$ is coercive modulo a norm-closed subspace $\Null \subset \const(K)$, then $\Null = \const(K)$.
\end{proposition}
\begin{proof} Since $K\not\equiv+\infty$, there exists $f_0\in\B$ with $K(f_0)<+\infty$. Choose $\alpha\in\R$ such that $K(f_0)\leq\alpha$.
By way of contradiction, suppose there exists $f \in \const(K)$ such that $f\notin\Null$. Then $K(f_0 + \gamma f) = K(f_0)\leq \alpha$ for all $\gamma > 0$. Coercivity modulo $\Null$ implies there exists $\beta > 0$ such that $\|[f_0 + \gamma f]\|_{\B/\Null} \leq \beta$ for all $\gamma > 0$. By the triangle inequality in $\B/\Null$,
\begin{equation}
\gamma\, \|[f]\|_{\B/\Null} \leq \|[f_0 + \gamma f]\|_{\B/\Null} + \|[f_0]\|_{\B/\Null} \leq \beta + \|[f_0]\|_{\B/\Null} < +\infty
\end{equation}
for all $\gamma > 0$, which forces $\|[f]\|_{\B/\Null} = 0$, i.e., $f \in \Null$, which is a contradiction.
\end{proof}

Recall that if $\M\subset\B=\A'$ is a $w^*$-closed subspace, then it is canonically a dual Banach space. Define
\begin{equation}
    \M_\perp
    \coloneqq
    \curly{\mu\in\A\st\langle\mu,m\rangle=0\text{ for all }m\in\M}.
\end{equation}
Then $\M=(\A/\M_\perp)'$ isometrically, with the norm inherited from $\B$, and the corresponding weak$^*$ topology agrees with the subspace topology inherited from $\B$; see, e.g.,~\cite[Chapter~8]{nariciTopological2010}.

We next combine coercivity modulo the constancy space with the factorization lemma to obtain recovery sequences on a topological complement.

\begin{proposition}\label{prop:lscreg_separable_predual}
Suppose $\A$ is separable, $\tau$ is a weak$^*$-compatible topology on $\B = \A'$, and $K:\B\to\eR$ is coercive modulo a subspace $\Null\subset\const(K)$ such that $\Null$ has $w^*$-complement $\M \subset \B$. Let $\P_\M$ denote the associated $w^*$-$w^*$-continuous projector onto $\M$, and let $w^*_\M$ and $\tau_\M$ denote the subspace topologies on $\M$ induced by $w^*$ and $\tau$, respectively. Then, for all $f\in\B$,
\begin{equation}\label{eq:set2}
\overline{K}^{\,w^*}(f)
  = \overline{(K|_{\M})}^{\mathrm{seq},\tau_\M}(\P_\M f)
  = \min\{\alpha\in\eR \st (g_n)_{n\in\N}\subset\dom(K|_\M),\
      g_n\xrightarrow{\tau_\M}\P_\M f,\ K(g_n)\to\alpha\}.
\end{equation}
\end{proposition}
\begin{proof}

Fix $f\in\B$. By \cref{lem:chatacterizing_R_mod_N}, $\overline{K}^{\,w^*}(f)=\overline{(K|_\M)}^{\,w^*_\M}(\P_\M f)$, where $\overline{(K|_\M)}^{\,w^*_\M}$ is the l.s.c.\ regularization of $K|_\M$ in the weak$^*$ topology $w^*_\M$ on $\M$. Since $K$ is coercive modulo $\Null \subset \const(K)$, $K|_\M$ is coercive on $\M$ by \cref{prop:coercive-mod-N-equivalences}~\ref{item:coercive_on_M}. Also, $\M=(\A/\M_\perp)'$ isometrically, and its quotient predual $\A/\M_\perp$ is separable. Finally, since $\tau$ is weak$^*$-compatible on $\B$, and $\M$ carries the same norm as $\B$, the subspace topology $\tau_\M$ is weak$^*$-compatible on $\M$. Applying \cref{prop:coercive_implies_seqlscreg_compatible} to $K|_\M:\M\to\eR$, with the topologies $w^*_\M$ and $\tau_\M$, gives the result.
\end{proof}

\subsubsection{Proof of \cref{prop:coercive_mod_subspace_implies_seqlscreg_compatible}}\label[appendix]{subsec:recovery-sequence-proof}

\begin{proof}
Let $\P_\Null:\B\rightarrow\B$ be a $w^*$-$w^*$-continuous projector onto $\Null$, $\M \coloneqq \ker \P_\Null$, and $\P_\M \coloneqq \Id - \P_{\Null}$. Let $w^*_\M$ denote the subspace topology on $\M$ induced by $w^*$.
Fix $f\in\B$ and let $J(f)$ denote the right-hand side of \cref{eq:set1} with $\inf$ in place of $\min$. Since $w^*$ is itself weak$^*$-compatible, applying
\cref{prop:lscreg_separable_predual} with $\tau = w^*$ gives
\begin{equation}\label{eq:Kbar_intermediate}
\overline{K}^{\,w^*}(f) = \min\{\alpha\in\eR \st (g_n)_{n\in\N}\subset\dom(K|_\M),\
      g_n\xrightarrow{w^*_\M}\P_\M f,\ K(g_n)\to\alpha\}.
\end{equation}

First we show $\overline{K}^{\,w^*}(f) \leq J(f)$. The inequality is trivial if $J(f)=+\infty$, so assume $J(f)<+\infty$. Let $(f_n)_{n\in\N}\subset\dom(K)$ satisfy $f_n\tauconverges f$ and $K(f_n)\to\alpha$ for some $\alpha<+\infty$. Set $g_n \coloneqq \P_\M f_n = f_n-\P_\Null f_n$. Since $\P_\Null f_n\in\Null\subset\const(K)$, we have $K(g_n)=K(f_n)\to\alpha$, so $(g_n)_{n\in\N}\subset\dom(K|_\M)$ with $\sup_n K(g_n)<+\infty$. By \cref{prop:coercive-mod-N-equivalences}~\ref{item:coercive_on_M}, $K|_\M$ is coercive, so $(g_n)_{n\in\N}$ is norm-bounded, i.e., $(g_n)_{n\in\N}\subset\{\norm{\dummy}_\B \leq r\}$ for some $r>0$. Since $\A$ is separable, $\{\norm{\dummy}_\B \leq r\}$ is $w^*$-compact and metrizable, so there is a subsequence with $g_{n_j}\starconverges h$, and $h\in\M$ since $\M$ is $w^*$-closed. Being norm-bounded and $w^*$-convergent, the subsequence also satisfies $g_{n_j}\tauconverges h$ by \cref{lem:seq_star_converge_equiv}. Since $\tau$ is a vector-space topology,
\begin{equation}
\P_\Null f_{n_j} = f_{n_j} - g_{n_j} \tauconverges f - h .
\end{equation}
As $\P_\Null f_{n_j}\in\Null$ and $\Null$ is sequentially $\tau$-closed, we get $f-h\in\Null=\ker\P_\M$, and so $h = \P_\M h = \P_\M f$. Thus $(g_{n_j})_{j\in\N}$ is admissible in \cref{eq:Kbar_intermediate} with $K(g_{n_j})\to\alpha$, and so $\overline{K}^{\,w^*}(f)\leq\alpha$. Taking the infimum over admissible sequences yields $\overline{K}^{\,w^*}(f)\leq J(f)$.

If $\overline{K}^{\,w^*}(f)=+\infty$, the above inequality shows $J(f) = +\infty$. So assume $\overline{K}^{\,w^*}(f)<+\infty$. By \cref{eq:Kbar_intermediate} there exists $(g_n)_{n\in\N}\subset\dom(K|_\M)$ with $g_n\xrightarrow{w^*_\M}\P_\M f$ and $K(g_n)\to\overline{K}^{\,w^*}(f)$. Set $f_n\coloneqq g_n+\P_\Null f$, so that $f_n\starconverges f$, hence $f_n\tauconverges f$ by \cref{lem:seq_star_converge_equiv}. Since $\P_\Null f\in\Null\subset\const(K)$, we have $f_n\in\dom(K)$ and $K(f_n)=K(g_n)\to\overline{K}^{\,w^*}(f)$. Thus $\alpha=\overline{K}^{\,w^*}(f)$ is admissible in \cref{eq:set1}. Combined with the previous inequality, the infimum defining $J(f)$ equals $\overline{K}^{\,w^*}(f)$ and is attained by $(f_n)_{n\in\N}$.
\end{proof}

\section{Equivalence with Representation Cost Definition in Prior Work}
\label[appendix]{app:equivalence}

In this appendix we show that the representation cost introduced in~\cite{OngieFunctionSpace} for shallow ReLU networks with the usual squared Euclidean weight-decay parameter cost coincides with the representation cost defined in the present work. The model considered in~\cite{OngieFunctionSpace} agrees with the setup in \cref{subsec:deep_nets_rep_cost} specialized to $L=2$, $\X=\R^d$, and $D=1$. Here we also assume $\X=\R^d$, but allow arbitrary $L\geq 2$ and $D\geq 1$. Throughout this appendix, set $\B \coloneqq \Lip(\R^d;\R^D)$. To simplify notation, we write $\Theta$ in place of $\ThetaL$, $C$ in place of $C_L$, $\Rcirc$ in place of $\RcircL$, and $R \coloneqq \overline{\Rcirc}$ in place of $R_L$. 

For the weight-decay parameter cost $C$ from \cref{subsec:deep_nets_rep_cost}, the representation cost from~\cite{OngieFunctionSpace} is the function $\tilde R:\Lip(\R^d;\R^D)\to[0,+\infty]$ defined by
\begin{equation}\label{eq:Rtilde_def}
    \tilde R(f)
    \coloneqq
    \lim_{\varepsilon\to 0}
    \inf\left\{
        C(\theta)
        \st
        \theta\in\Theta,\,
        \sup_{\|x\|_2\leq 1/\varepsilon}
        \norm{f_\theta(x)-f(x)}_2
        \leq \varepsilon,\,
        f_\theta(0)=f(0)
    \right\}.
\end{equation}
In this appendix, we prove that $\tilde R=R$ (\cref{prop:Rtilde_equals_R}). We begin by rewriting $\tilde{R}$ in a more convenient form.

\begin{lemma}\label{lem:rewrite_Rtilde_a_little}
Let $f\in\Lip(\R^d;\R^D)$ and set $f_0(x)\coloneqq f(x)-f(0)$.
Then
\begin{equation}\label{eq:Rtilde_equiv}
    \tilde R(f)
    =
    \lim_{\varepsilon\to 0}
    \inf\left\{
        C(\theta)
        \st
        \theta\in\Theta,\,
        \sup_{\|x\|_2\leq 1/\varepsilon}
        \norm{f_\theta(x)-f_0(x)}_2
        \leq \varepsilon,\,
        f_\theta(0)=0
    \right\}.
\end{equation}
\end{lemma}

\begin{proof}
Fix $\varepsilon>0$. If $\theta\in\Theta$ is admissible in
\cref{eq:Rtilde_def}, shift only the outer-layer bias to obtain
$\theta'\in\Theta$ such that
\begin{equation}
    f_{\theta'}(x)=f_\theta(x)-f_\theta(0)=f_\theta(x)-f(0).
\end{equation}
Then $C(\theta')=C(\theta)$, $f_{\theta'}(0)=0$, and
\begin{equation}
    \norm{f_{\theta'}(x)-f_0(x)}_2
    =
    \norm{f_\theta(x)-f(x)}_2
    \leq \varepsilon
\end{equation}
for all $\|x\|_2 \leq 1/\varepsilon$. Conversely, shifting the
outer-layer bias by $f(0)$ maps every admissible parameter in
\cref{eq:Rtilde_equiv} to an admissible parameter in \cref{eq:Rtilde_def} with
the same cost. Thus the two infima agree for every $\varepsilon>0$.
\end{proof}

Next we rewrite $R$ in a form more amenable to analysis. Let $\tau$ denote the topology on $\Lip(\R^d;\R^D)$ of uniform convergence on compact subsets. For $n\in\N$, define the family of seminorms
\begin{equation}
    \rho_n(f)
    \coloneqq
    \sup_{\|x\|_2\leq n} \norm{f(x)}_2,
    \quad
    f\in\Lip(\R^d;\R^D),
\end{equation}
and consider the associated metric
\begin{equation}\label{eq:cpt_uniform_metric}
    \delta(f,g)
    \coloneqq
    \sum_{n=1}^{\infty}
    2^{-n}
    \frac{\rho_n(f-g)}{1+\rho_n(f-g)}.
\end{equation}
Then $\delta$ metrizes $\tau$. Moreover, $\tau$ is weak$^*$-compatible on
$\Lip(\R^d;\R^D)$. Indeed, by the Arzel\`a--Ascoli theorem, pointwise convergence and uniform convergence on compact subsets agree on norm-bounded subsets of $\Lip(\R^d;\R^D)$, and pointwise convergence is weak$^*$-compatible on $\Lip(\R^d;\R^D)$.

Let $\Null \coloneqq \Const(\R^d;\R^D)$ and define $\P_\Null:\Lip(\R^d;\R^D)\to\Lip(\R^d;\R^D)$ by $\P_\Null f \equiv f(0)$. Then $\P_\Null$ is $w^*$-$w^*$-continuous and
\begin{equation}
    \ker(\P_\Null)=\Lip_0(\R^d;\R^D)
    \coloneqq
    \{f\in\Lip(\R^d;\R^D)\st f(0)=0\}.
\end{equation}
Finally, by \cref{lem:deep_relu_coercive_bnd}, $C$ is parameter coercive
modulo $\Null$, and hence $\Rcirc$ is coercive modulo $\Null$.

\begin{lemma}\label{lem:rewrite_R_to_match_Rtilde}
Let $f\in\Lip(\R^d;\R^D)$ and set $f_0(x)\coloneqq f(x)-f(0)$. Then
\begin{equation}\label{eq:R_seqreg_Lip0}
    R(f)
    =
    \overline{(\Rcirc|_{\Lip_0(\R^d;\R^D)})}^{\mathrm{seq},\tau}(f_0).
\end{equation}
Moreover,
\begin{equation}\label{eq:R_metric_characterization}
    R(f)
    =
    \lim_{\varepsilon\to 0}
    \inf\left\{
        C(\theta)
        \st
        \theta\in\Theta,\,
        \delta(f_0,f_\theta)\leq\varepsilon,\,
        f_\theta(0)=0
    \right\}.
\end{equation}
\end{lemma}

\begin{proof}
Since $\tau$ is weak$^*$-compatible, \cref{prop:lscreg_separable_predual} applied with $\M=\Lip_0(\R^d;\R^D)$ and $\P_\M=\Id-\P_\Null$, gives \cref{eq:R_seqreg_Lip0}. Since $\delta$ metrizes $\tau$, \cref{prop:lscreg_equivalents}~\ref{item:lsc_in_metric_space} gives
\begin{align}
    \overline{(\Rcirc|_{\Lip_0(\R^d;\R^D)})}^{\mathrm{seq},\tau}(f_0)
    &=
    \lim_{\varepsilon\to 0}
    \inf\left\{
        \Rcirc(g)
        \st
        g\in\Lip_0(\R^d;\R^D),\,
        \delta(f_0,g)\leq\varepsilon
    \right\} \notag\\
    &=
    \lim_{\varepsilon\to 0}
    \inf\left\{
        C(\theta)
        \st
        \theta\in\Theta,\,
        \delta(f_0,f_\theta)\leq\varepsilon,\,
        f_\theta(0)=0
    \right\}.
\end{align}
Combining the previous two displays proves the result.
\end{proof}

\begin{proposition}\label{prop:Rtilde_equals_R}
For every $f\in\Lip(\R^d;\R^D)$, $R(f)=\tilde R(f)$.
\end{proposition}

\begin{proof}
Fix $f\in\Lip(\R^d;\R^D)$ and set $f_0(x)\coloneqq f(x)-f(0)$. We first prove $R(f)\leq \tilde R(f)$. The claim is trivial if $\tilde R(f)=+\infty$. Otherwise, by \cref{lem:rewrite_Rtilde_a_little}, for each $n\in\N$ we may choose $\theta_n\in\Theta$ such that, with $f_n\coloneqq f_{\theta_n}$,
\begin{equation}
    \norm{f_n(x)-f_0(x)}_2
    \leq \frac{1}{n}
    \quad\text{for all } \norm{x}_2\leq n,
    \quad
    f_n(0)=0,
\end{equation}
and
\begin{equation}
    C(\theta_n)
    \leq
    \inf\left\{
        C(\theta)
        \st
        \theta\in\Theta,\,
        \sup_{\norm{x}_2\leq n}\norm{f_\theta(x)-f_0(x)}_2
        \leq \frac{1}{n},\,
        f_\theta(0)=0
    \right\}
    +
    \frac{1}{n}.
\end{equation}
Then $f_n\in\Lip_0(\R^d;\R^D)$ and $f_n\tauconverges f_0$. Therefore, by
\cref{lem:rewrite_R_to_match_Rtilde},
\begin{equation}\label{eq:R_leq_Rtilde}
    R(f)
    \leq
    \liminf_{n\to\infty}\Rcirc(f_n)
    \leq
    \liminf_{n\to\infty}C(\theta_n)
    \leq
    \tilde R(f).
\end{equation}

We now prove the reverse inequality.  Since the quantities inside the limits in \cref{eq:Rtilde_equiv,eq:R_metric_characterization} are monotone as $\varepsilon\downarrow 0$, the limits may be taken along any sequence of positive parameters tending to zero. Define
\begin{equation}
    \Phi_n
    \coloneqq
    \left\{
        \theta\in\Theta
        \st
        \delta(f_0,f_\theta)\leq \frac{2^{-n}}{n+1},\,
        f_\theta(0)=0
    \right\},
    \quad
    A_n
    \coloneqq
    \inf_{\theta\in\Phi_n} C(\theta),
\end{equation}
and
\begin{equation}
    \Psi_n
    \coloneqq
    \left\{
        \theta\in\Theta
        \st
        \rho_n(f_0-f_\theta)\leq \frac{1}{n},\,
        f_\theta(0)=0
    \right\},
    \quad
    B_n
    \coloneqq
    \inf_{\theta\in\Psi_n} C(\theta).
\end{equation}
Then, by \cref{lem:rewrite_R_to_match_Rtilde,lem:rewrite_Rtilde_a_little}, $R(f)=\lim_{n\to\infty}A_n$ and $\tilde R(f)=\lim_{n\to\infty}B_n$. We claim that $\Phi_n\subset\Psi_n$. Indeed, if $\theta\in\Phi_n$, then
\begin{equation}
    2^{-n}
    \frac{\rho_n(f_\theta-f_0)}{1+\rho_n(f_\theta-f_0)}
    \leq
    \delta(f_0,f_\theta)
    \leq
    \frac{2^{-n}}{n+1}.
\end{equation}
Thus
\begin{equation}
    \frac{\rho_n(f_\theta-f_0)}{1+\rho_n(f_\theta-f_0)}
    \leq
    \frac{1}{n+1},
\end{equation}
which implies $\rho_n(f_\theta-f_0)\leq 1/n$. Hence $\theta\in\Psi_n$, so
$\Phi_n\subset\Psi_n$. Therefore $B_n\leq A_n$ for every $n\in\N$, and so
\begin{equation}
    \tilde R(f)
    =
    \lim_{n\to\infty}B_n
    \leq
    \lim_{n\to\infty}A_n
    =
    R(f).
\end{equation}
Combining this with \cref{eq:R_leq_Rtilde} proves the result.
\end{proof}

\section{Proof of \Cref{lem:jacobian-minor-bound}}
\label[appendix]{app:singval-lower-bnd}

The proof of \cref{lem:jacobian-minor-bound} uses some basic facts about matrix exterior powers that we recall here. Given any matrix $A \in \R^{m\times n}$, for index sets $I\subset\{1,\ldots,m\}$ and $J\subset\{1,\ldots,n\}$ with $\abs{I}=\abs{J}=k$, we let $A_{I,J}\in\R^{k\times k}$ denote the corresponding submatrix of $A\in\R^{m\times n}$, so that the $k\times k$ minors of $A$ are the determinants $\det(A_{I,J})$. We also let $\exterior{k} A\in\R^{\binom{m}{k}\times\binom{n}{k}}$ denote the $k$th exterior power of $A$.
When $k>\min\curly{m,n}$ we interpret $\exterior{k}A$ as the zero matrix. Padding singular values by zeros if necessary, we have the standard identity
\begin{equation}
    \norm{\exterior{k} A}_\op = \prod_{m=1}^{k}\sigma_m(A).
    \label{eq:exterior-power-op-norm}
\end{equation}

In this appendix, we also adopt the notation that $\overset{*}{\rightharpoonup}$ denotes weak$^*$ convergence when a canonical predual is clear from context. We still use $\starconverges$ explicitly for weak$^*$ convergence in $\Lip(\X;\R^D)$. We let $\partial_i$ denote the (weak) partial derivative with respect to the $i$th-coordinate.

\begin{proof}
Fix $1 \leq k \leq \min\{d,D\}$. First, we prove that the desired inequality holds for the parametric representation cost $R_{\ThetaL}$ in place of the extended representation cost $R_L$. Fix $f \in \F_{\Theta_L}$, and let $\theta\in\ThetaL$ be such that $f = f_\theta$ on $\X$. At almost every $x\in\operatorname{int}(\X)$, $f$ is differentiable and its Jacobian is given by
\begin{equation}
    {\D} f(x)
    =W_LP_{L-1}(x)W_{L-1}\cdots P_1(x)W_1,
\end{equation}
where each $P_\ell(x)$ is diagonal with diagonal entries in $\{0,1\}$. By \eqref{eq:exterior-power-op-norm}, we have that
\begin{equation}
\norm{\exterior{k}P_\ell(x)}_\op = \prod_{m=1}^{k}\sigma_m(P_\ell(x))\leq 1.
\end{equation}
Again by
\cref{eq:exterior-power-op-norm} and the arithmetic mean--geometric mean (AM--GM) inequality, we have
\begin{equation}
    \norm{\exterior{k} W_\ell}_\op
    =
    \prod_{m=1}^{k}\sigma_m(W_\ell)
    \leq
    \paren*{\frac{1}{k}\sum_{m=1}^{k}\sigma_m(W_\ell)^2}^{k/2}
    \leq
    \paren*{\frac{\norm{W_\ell}_F^2}{k}}^{k/2}.
\end{equation}
Therefore, by the multiplicativity identity of exterior powers and another application of the AM--GM inequality, we have for almost every $x\in\operatorname{int}(\X)$ that
\begin{equation}
    \norm{\exterior{k}\D f(x)}_\op
    \leq
    \prod_{\ell=1}^{L}\norm{\exterior{k} W_\ell}_\op
    \leq
    k^{-kL/2}\paren*{\prod_{\ell=1}^{L}\norm{W_\ell}_F^2}^{k/2}
    \leq
    \paren*{\frac{C_L(\theta)}{k}}^{kL/2}.
    \label{eq:network-exterior-bound}
\end{equation}

Since $R_{\Theta_L}(f)<+\infty$, choose a sequence
$(\theta_n)_{n\in\N}\subset\ThetaL$ such that
\begin{equation}
f=f_{\theta_n}\text{ on }\X
\quad\text{and}\quad
C_L(\theta_n)\to R_{\Theta_L}(f).
\end{equation}
For each $n\in\N$, \cref{eq:network-exterior-bound} holds with
$\theta=\theta_n$ outside a null set $Z_n\subset\operatorname{int}(\X)$.
Consequently, outside the null set
$Z\coloneqq\bigcup_{n\in\N}Z_n$, we have
\begin{equation}
\norm{\exterior{k}{\D} f(x)}_\op
\leq
\paren*{\frac{C_L(\theta_n)}{k}}^{kL/2}
\quad\text{for every }n\in\N.
\end{equation}
Letting $n\to\infty$ gives
\begin{equation}
\norm{\exterior{k}\D f(x)}_\op
\leq
\paren*{\frac{R_{\Theta_L}(f)}{k}}^{kL/2}
\quad\text{for almost every }x\in\operatorname{int}(\X).
\label{eq:parametric-bound}
\end{equation}

Fix an open ball $B\subset\subset\operatorname{int}(\X)$ and write
$E\coloneqq\R^{\binom{D}{k}\times\binom{d}{k}}$, equipped with the operator norm
$\norm{\cdot}_\op$. Write $E^*$ for the same underlying matrix space equipped
with the nuclear norm. Under the Frobenius inner product, $E$ is isometrically the
dual of $E^*$. Since $B$ has finite measure and $E$ is finite dimensional,
$L^\infty(B;E)$ is therefore isometrically the dual of $L^1(B;E^*)$. Moreover,
for $f\in\Lip(\X;\R^D)$ we have
$\exterior{k}{\D} f\in L^\infty(B;E)$. Define
\begin{equation}
    J_{B}(f)
    \coloneqq
    \esssup_{x \in B} \norm{\exterior{k} {\D} f(x)}_\op
    =
    \norm{\exterior{k}{\D} f}_{L^\infty(B;E)}.
\end{equation}
We claim that $J_{B}$ is sequentially $w^*$-l.s.c.\ on $\Lip(\X;\R^D)$, i.e., if
$f_n\starconverges f$ in $\Lip(\X;\R^D)$ then
\begin{equation}
    J_{B}(f)
    \leq
    \liminf_{n\to\infty}J_{B}(f_n).
    \label{eq:minor-seq-lsc}
\end{equation}

By \cref{rem:vv-lip-con}, $f_n\starconverges f$ in $\Lip(\X;\R^D)$ is equivalent to
$\sup_{n\in\N}\norm{f_n}_{\Lip(\X;\R^D)}<+\infty$ together with $f_n\to f$ pointwise
on $\X$. Since $B$ is a bounded convex subset of $\X$, \cref{rem:lip-sobolev} gives
\begin{equation}
\norm{{\D} f_n}_{L^\infty(B;\R^{D\times d})}
=
\abs{f_n}_{\Lip(B;\R^D)}
\leq
\norm{f_n}_{\Lip(\X;\R^D)}.
\end{equation}
The uniform $\Lip(\X;\R^D)$-bound and the boundedness of $B$ also give
\begin{equation}
\sup_{n\in\N}\norm{f_n}_{L^\infty(B;\R^D)}<+\infty.
\end{equation}
Thus, $(f_n)_{n\in\N}$ is bounded in $W^{1,\infty}(B;\R^D)$. Being
equi-Lipschitz, the $f_n$ converge to $f$ uniformly on $\overline{B}$ by the
compactness of $\overline{B}$, and hence in $L^\infty(B;\R^D)$. For the
derivatives, let $\varphi\in C^\infty_c(B)$, $j\in\{1,\ldots,D\}$ and
$r\in\{1,\ldots,d\}$; integrating by parts and using this uniform convergence
yields
\begin{equation}
    \ang{\partial_r f_{n,j},\varphi}
    =-\ang{f_{n,j},\partial_r\varphi}
    \to-\ang{f_j,\partial_r\varphi}
    =\ang{\partial_r f_j,\varphi}.
\end{equation}
Together with the uniform $L^\infty$-bound and the density of $C^\infty_c(B)$ in
$L^1(B)$, this gives ${\D} f_n\overset{*}{\rightharpoonup}{\D} f$ in
$L^\infty(B;\R^{D\times d})$, and therefore
$f_n\overset{*}{\rightharpoonup}f$ in $W^{1,\infty}(B;\R^D)$.

Each component of the map $f \mapsto \exterior{k}\D f$ is a $k\times k$ minor
$\det((\D f)_{I,J})$, hence a null Lagrangian, and thus sequentially weak$^*$
continuous from $W^{1,\infty}(B;\R^D)$ to
$L^\infty(B)$~\cite[Theorem~4.3 and Theorem~4.5]{benesova2017weak}
(see also \cite{ball1981null}). Applying this fact entrywise and using the
finite dimensionality of $E$ gives
\begin{equation}
    \exterior{k} \D f_n
    \overset{*}{\rightharpoonup}
    \exterior{k} \D f
    \quad\text{in }L^{\infty}(B;E).
    \label{eq:ext-weak}
\end{equation}
Finally, the norm of a dual Banach space is weak$^*$-l.s.c. Hence,
under the duality between $L^\infty(B;E)$ and $L^1(B;E^*)$ described above,
\cref{eq:ext-weak} implies \cref{eq:minor-seq-lsc}.

For every open ball $B\subset\subset\operatorname{int}(\X)$ and every
$g\in\Lip(\X;\R^D)$, \cref{eq:parametric-bound} gives
\begin{equation}
J_B(g)
\leq
\paren*{\frac{R_{\Theta_L}(g)}{k}}^{kL/2},
\label{eq:parametric-JB-bound}
\end{equation}
where the inequality is trivial when $g\notin\F_{\Theta_L}$.

Now fix $f\in\Lip(\X;\R^D)$. If $R_L(f)=+\infty$, the desired inequality is
trivial. Suppose that $R_L(f)<+\infty$. The space $\Lip(\X;\R^D)$ has a
separable predual. By \cref{lem:deep_relu_coercive_bnd}, $R_{\Theta_L}$ is coercive modulo $\Const(\X;\R^D) = \const(R_{\Theta_L})$, which is $w^*$-complemented and $w^*$-closed. Therefore,
by \cref{prop:coercive_mod_subspace_implies_seqlscreg_compatible} applied with $\tau=w^*$, there exists a sequence
$(f_n)_{n\in\N}\subset\F_{\Theta_L}$ such that
\begin{equation}
f_n\starconverges f
\text{ in }\Lip(\X;\R^D)
\quad\text{and}\quad
R_{\Theta_L}(f_n)\to R_L(f).
\end{equation}
It follows from \cref{eq:minor-seq-lsc,eq:parametric-JB-bound} that
\begin{equation}
J_B(f)
\leq
\liminf_{n\to\infty}J_B(f_n) \leq
\liminf_{n\to\infty}
\paren*{\frac{R_{\Theta_L}(f_n)}{k}}^{kL/2}
=
\paren*{\frac{R_L(f)}{k}}^{kL/2}.
\end{equation}
Thus, for every $f\in\Lip(\X;\R^D)$,
\begin{equation}
J_B(f)
\leq
\paren*{\frac{R_L(f)}{k}}^{kL/2}.
\end{equation}
By definition of $J_B$ and the identity
\cref{eq:exterior-power-op-norm}, this says exactly that
\begin{equation}
\prod_{m=1}^{k}\sigma_m(\D f(x))
=
\norm{\exterior{k}\D f(x)}_\op
\leq
\paren*{\frac{R_L(f)}{k}}^{kL/2}
\quad\text{for almost every }x\in B.
\end{equation}
Since $\operatorname{int}(\X)$ admits a countable cover by open balls
compactly contained in $\operatorname{int}(\X)$, the bound holds almost
everywhere on $\operatorname{int}(\X)$, which proves
\cref{eq:jacobian-minor-bound}.
\end{proof}

\section{Proof of \Cref{prop:deep-unit-ball-nonconvex}}
\label[appendix]{app:nonconvex}
\begin{proof}
Choose
$x_0\in\operatorname{int}(\X)$ and $\rho>0$ such that $B(x_0,2\rho)$ (open ball of radius $2\rho$ centered at $x_0$) satisfies $B(x_0,2\rho)\subset\X$. Let $u_1,u_2$ be the first two canonical basis vectors of $\R^d$, and let $v_1,v_2$ be the first two canonical basis vectors of $\R^D$. For $k\in\{1,2\}$, define
\begin{equation}
    f_k(x)
    \coloneqq
    v_k\paren{u_k^\T(x-x_0)+\rho}_+,
    \quad x\in\X.
\end{equation}
This function is realized by a depth-$L$ network with scalar hidden
layers and affine maps
\begin{equation}
    a_1(x)=u_k^\T(x-x_0)+\rho,
    \quad
    a_\ell(t)=t
    \quad(2\leq\ell\leq L-1),
    \quad
    a_L(t)=tv_k.
\end{equation}
Every weight matrix in this realization has squared Frobenius norm
one, so $C_L(\theta)=1$ and therefore $R_L(f_k)\leq 1$. Also,
$f_k$ is $1$-Lipschitz. For every $t\in(0,\rho)$, both $x_0$ and
$x_0+tu_k$ lie in $\X$, and
\begin{equation}
    \frac{
    \norm{f_k(x_0+tu_k)-f_k(x_0)}_2
    }{t}
    =1.
\end{equation}
Hence, $\abs{f_k}_{\Lip(\X;\R^D)}=1$. Since adding a constant does not change the Lipschitz seminorm, while
subtracting $f_k(\e)$ makes the base point value vanish,
\begin{equation}
    \inf_{g\in\Const(\X;\R^D)}
    \norm{f_k+g}_{\Lip(\X;\R^D)}
    =
    \abs{f_k}_{\Lip(\X;\R^D)}
    =
    1.
\end{equation}
Applying \cref{lem:deep_relu_coercive_bnd} gives
$1\leq R_L(f_k)^{L/2}$. Thus
\begin{equation}
    R_L(f_1)=R_L(f_2)=1.
    \label{eq:two-unit-cost-ridges}
\end{equation}

Let
\begin{equation}
    h\coloneqq\frac{f_1+f_2}{2}.
\end{equation}
On $B(x_0,\rho/2)$, both ReLU arguments are strictly positive, and
hence
\begin{equation}
    \D h
    =
    \frac{1}{2}
    \paren{v_1u_1^\T+v_2u_2^\T}.
\end{equation}
Since $u_1,u_2$ and $v_1,v_2$ are orthonormal, the nonzero singular values
of $\D h$ are $\sigma_1(\D h)=\sigma_2(\D h)=1/2$, so
\begin{equation}
    \sigma_1(\D h)\,\sigma_2(\D h)
    =
    \frac{1}{4}
    \quad\text{on }B(x_0,\rho/2),
\end{equation}
a subset of $\operatorname{int}(\X)$ of positive Lebesgue measure.
Applying \cref{lem:jacobian-minor-bound} with $k=2$ gives
\begin{equation}
    \frac{1}{4}
    \leq
    \paren*{\frac{R_L(h)}{2}}^{L},
\end{equation}
so $R_L(h)\geq 2^{1-2/L}$. Conversely, by the absolute $(2/L)$-homogeneity and subadditivity of $R_L$,
\begin{align}
    R_L(h)
    =
    2^{-2/L}R_L(f_1+f_2)
    \leq
    2^{-2/L}\paren{R_L(f_1)+R_L(f_2)}
    =
    2^{1-2/L},
\end{align}
where we used \cref{eq:two-unit-cost-ridges}. Hence
\begin{equation}
    R_L\paren*{\frac{f_1+f_2}{2}}
    =
    2^{1-2/L}
    >
    1.
\end{equation}
Thus the unit ball $\{R_L\leq 1\}$ is nonconvex.

Finally, since $t\mapsto t^{L/2}$ is increasing on
$[0,+\infty]$,
\begin{equation}
    \curly{R_L\leq 1}
    =
    \curly{R_L^{L/2}\leq 1},
\end{equation}
so the unit ball of the rooted representation cost is the same
nonconvex set. By \cref{thm:quasi-Banach-geometry},
$R_L^{L/2}$ is absolutely $1$-homogeneous on $\F_L$. 
Set $\widetilde f_k \coloneqq f_k-f_k(\e)$, $k=1,2$. Since $R_L$ is
invariant under addition of constants,
\begin{equation}
R_L(\widetilde f_k)=R_L(f_k)=1
\quad\text{and}\quad
\widetilde f_k(\e)=0,
\end{equation}
so $\norm{\widetilde f_k}_{\F_L}=1$. On the other hand,
\begin{equation}
\norm*{\frac{\widetilde f_1+\widetilde f_2}{2}}_{\F_L} = R_L\paren*{\frac{\widetilde f_1+\widetilde f_2}{2}}^{L/2}
=
R_L(h-h(\e))^{L/2}
=
R_L(h)^{L/2} = (2^{1 - 2/L})^{L/2}>1.
\end{equation}
Thus, $\sls{\norm{\dummy}_{\F_L}}{1}$ is also nonconvex. Since the unit ball of a \textit{bona fide} (semi)norm must be convex, the final assertions follow.
\end{proof}

\section{Existence of Minimizers for Deep Neural Networks}

\begin{lemma}
\label{lem:deep-relu-finite-sample-surjectivity}
Fix $L\geq2$ and let $x_1,\ldots,x_N\in\X\subset\R^d$ be distinct. For every $K\geq N$, we have $\EOp(\F_{\Theta_L(K)})=(\R^D)^N$, where
\begin{equation}
    \EOp f\coloneqq \big(f(x_1),\ldots,f(x_N)\big).
\end{equation}
In particular,
\begin{equation}
    \EOp(\FThetaL)=(\R^D)^N.
\end{equation}
\end{lemma}

\begin{proof}
Fix arbitrary targets $z_1,\ldots,z_N\in\R^D$. If $N=1$, the constant map $x\mapsto z_1$ gives the result. Assume $N\geq2$. Since the points $x_i$ are distinct, there exists $u\in\R^d$ such that the scalars $t_i\coloneqq u^\T x_i$ are pairwise distinct. Indeed, it suffices to choose $u$ outside the finite union of hyperplanes
\begin{equation}
    \bigcup_{1\leq i<j\leq N}\curly{u\in\R^d\st u^\T(x_i-x_j)=0}.
\end{equation}
Suppose $t_1<\cdots<t_N$, and define
\begin{equation}
    s_i\coloneqq\frac{z_{i+1}-z_i}{t_{i+1}-t_i}\in\R^D,
    \quad i=1,\ldots,N-1.
\end{equation}
Consider the vector-valued linear spline
\begin{equation}
    g(t)
    \coloneqq
    z_1+s_1(t-t_1)
    +\sum_{j=2}^{N-1}(s_j-s_{j-1})(t-t_j)_+.
\end{equation}
We have $g(t_1)=z_1$, and on each interval $(t_i,t_{i+1})$ its derivative is
\begin{equation}
    s_1+\sum_{j=2}^{i}(s_j-s_{j-1})=s_i.
\end{equation}
It follows immediately that $g(t_i)=z_i$ for every $i$. Moreover,
\begin{equation}
    s_1(t-t_1)=s_1(t-t_1)_+-s_1(t_1-t)_+,
\end{equation}
so $x\mapsto g(u^\T x)$ is realized by a two-layer ReLU network with at most $2+(N-2)=N$ hidden neurons and satisfies $g(u^\T x_i)=z_i$.

For $L>2$, the first hidden activation vector is componentwise nonnegative. Insert $L-2$ identity affine maps with zero bias and identity weight matrix before the final affine map. Thus, each ReLU leaves that vector unchanged. This gives a depth-$L$ network with hidden widths at most $N$. For $K\geq N$, zero-padding gives a parameter in $\Theta_L(K)$. Hence $\EOp(\F_{\Theta_L(K)})=(\R^D)^N$, and the final assertion follows from $\F_{\Theta_L(K)}\subset\FThetaL$.
\end{proof}

\begin{remark}
    This proof provides a vector-valued generalization of the famous universal interpolation theorem for neural networks~\cite[Theorem~5.1]{pinkus1999approximation}.
\end{remark}

\begin{theorem}
\label{thm:parametric-deep-existence-width-bound}
Fix $L\geq2$ and let $x_1,\ldots,x_N\in\X\subset\R^d$ be distinct.
Let $\ThetaL$ denote the parameter space of all finite-width depth-$L$ fully
connected ReLU networks
\begin{equation}
    f_\theta
    =
    a_L\circ\sigma\circ a_{L-1}\circ\sigma\circ\cdots\circ\sigma\circ a_1,
\end{equation}
where $a_\ell(z)=W_\ell z+b_\ell$, the hidden-layer widths are arbitrary and
finite, and
\begin{equation}
    C_L(\theta)
    \coloneqq
    \frac{1}{L}\sum_{\ell=1}^L\norm{W_\ell}_F^2.
\end{equation}
For $K\in\N$, let $\Theta_L(K)\subset\ThetaL$ denote the parameter space of
networks whose hidden-layer widths are at most $K$. Let
$G:(\R^D)^N\to[0,+\infty]$ be proper, l.s.c., and coercive, and let
$\lambda>0$. Then the unrestricted parameter-space problem
\begin{equation}
    \inf_{\theta\in\ThetaL}
    G\big(f_\theta(x_1),\ldots,f_\theta(x_N)\big)
    +
    \lambda C_L(\theta)
    \label{eq:appendix-unrestricted-deep-param}
\end{equation}
admits a minimizer. Moreover, it admits a minimizer
$\theta^\star\in\Theta_L(N^2)$. In fact, for every $K\geq N^2$, the
finite-dimensional problem
\begin{equation}
    \min_{\theta\in\Theta_L(K)}
    G\big(f_\theta(x_1),\ldots,f_\theta(x_N)\big)
    +
    \lambda C_L(\theta)
    \label{eq:appendix-fixed-width-deep-param}
\end{equation}
has a nonempty solution set and has the same minimum value as
\cref{eq:appendix-unrestricted-deep-param}. Consequently, every minimizer of
\cref{eq:appendix-fixed-width-deep-param} is a minimizer of
\cref{eq:appendix-unrestricted-deep-param}.
\end{theorem}

\begin{proof}
Fix $K\geq N^2$. We first prove existence over $\Theta_L(K)$. By zero-padding,
identify $\Theta_L(K)$ with the parameter space of networks whose hidden-layer
widths are all exactly $K$. By
\cref{lem:deep-relu-finite-sample-surjectivity}, the objective is proper on
$\Theta_L(K)$. Let $(\theta_n)_{n\in\N}\subset\Theta_L(K)$ be a minimizing
sequence and write
\begin{equation}
    \theta_n
    =
    \big((W_{1,n},b_{1,n}),\ldots,(W_{L,n},b_{L,n})\big).
\end{equation}
After discarding finitely many terms, there exists $\alpha<+\infty$ such that
\begin{equation}
    G\big(f_{\theta_n}(x_1),\ldots,f_{\theta_n}(x_N)\big)
    +
    \lambda C_L(\theta_n)
    \leq \alpha.
\end{equation}
Since $\lambda>0$, the weight matrices $W_{\ell,n}$ are uniformly bounded.
Since $C_L\geq0$,
\begin{equation}
    G\big(f_{\theta_n}(x_1),\ldots,f_{\theta_n}(x_N)\big)
    \leq \alpha.
\end{equation}
By coercivity of $G$, the output matrices
\begin{equation}
    Y_n
    \coloneqq
    \begin{bmatrix}
        f_{\theta_n}(x_1) & \cdots & f_{\theta_n}(x_N)
    \end{bmatrix}
    \in\R^{D\times N}
\end{equation}
are uniformly bounded.

We next replace $\theta_n$ by an equivalent representative, with the same
weights and the same outputs on the training data, whose hidden activations
are bounded. Set $h_{i,n}^0=x_i$ and
\begin{equation}
    h_{i,n}^\ell
    =
    \sigma(W_{\ell,n}h_{i,n}^{\ell-1}+b_{\ell,n}),
    \quad
    \ell=1,\ldots,L-1.
\end{equation}
Fix a hidden layer $\ell$ and define
\begin{equation}
    m_n^\ell
    \coloneqq
    \paren*{
        \min_{1\leq i\leq N}(h_{i,n}^\ell)_1,
        \ldots,
        \min_{1\leq i\leq N}(h_{i,n}^\ell)_K
    }
    \in[0,+\infty)^K.
\end{equation}
For every coordinate with positive minimum, all corresponding preactivations
are positive on the training data. Therefore, decreasing the corresponding
bias in layer $\ell$ by that minimum subtracts the same quantity from the
corresponding hidden activation. Doing this in all coordinates replaces
$h_{i,n}^\ell$ by $h_{i,n}^\ell-m_n^\ell$. This change is exactly canceled in
the next preactivation by replacing
\begin{equation}
    b_{\ell+1,n}
    \quad\text{with}\quad
    b_{\ell+1,n}+W_{\ell+1,n}m_n^\ell,
\end{equation}
with the same interpretation for $\ell=L-1$, where the correction is made to
the final bias $b_{L,n}$. Hence this operation leaves all weights unchanged,
leaves the network outputs on the training data unchanged, and normalizes the
hidden activations so that
\begin{equation}
    \min_{1\leq i\leq N}(h_{i,n}^\ell)_k=0
    \quad
    \text{for every hidden layer $\ell$ and coordinate $k$.}
    \label{eq:hidden-min-zero}
\end{equation}

We claim that the normalized hidden activations are uniformly bounded. This is
proved by induction on $\ell$. The input activations are fixed. Suppose
$\{h_{i,n}^{\ell-1}:1\leq i\leq N,\ n\in\N\}$ is bounded. Fix a coordinate
$k$ in layer $\ell$, and choose $j$ such that
$(h_{j,n}^\ell)_k=0$. Let $w_{k,n}^{\ell,\T}$ be the $k$th row of
$W_{\ell,n}$. For every $i$,
\begin{equation}
    0
    \leq
    (h_{i,n}^\ell)_k
    \leq
    \abs{
        w_{k,n}^{\ell,\T}h_{i,n}^{\ell-1}
        -
        w_{k,n}^{\ell,\T}h_{j,n}^{\ell-1}
    }
    \leq
    \norm{w_{k,n}^{\ell}}_2
    \norm{h_{i,n}^{\ell-1}-h_{j,n}^{\ell-1}}_2.
\end{equation}
The right-hand side is uniformly bounded, so the $\ell$th hidden activations
are uniformly bounded. This proves the claim.

Passing to a subsequence, we may assume that all weights, all normalized hidden
activation vectors, and the output matrices converge:
\begin{equation}
    W_{\ell,n}\to W_\ell,
    \quad
    h_{i,n}^\ell\to h_i^\ell,
    \quad\text{and}\quad
    Y_n\to Y.
\end{equation}
We show that the limiting hidden activations are realized by finite biases.
Fix a hidden layer $\ell$ and coordinate $k$. Set
\begin{equation}
    a_{i,n}\coloneqq w_{k,n}^{\ell,\T}h_{i,n}^{\ell-1}
    \quad\text{and}\quad
    a_i\coloneqq w_k^{\ell,\T}h_i^{\ell-1}.
\end{equation}
Then $a_{i,n}\to a_i$ and
\begin{equation}
    (h_{i,n}^\ell)_k=\sigma(a_{i,n}+b_{k,n}^\ell).
\end{equation}
After passing to a further subsequence, either $b_{k,n}^\ell$ converges to a
finite number $b_k^\ell$, or $b_{k,n}^\ell\to-\infty$. The case
$b_{k,n}^\ell\to+\infty$ is impossible because the activations are bounded.
In the finite case,
\begin{equation}
    (h_i^\ell)_k=\sigma(a_i+b_k^\ell).
\end{equation}
In the case $b_{k,n}^\ell\to-\infty$, the limiting activation vector is zero,
and it is realized by any finite $b_k^\ell$ satisfying
\begin{equation}
    b_k^\ell<-\max_{1\leq i\leq N} a_i.
\end{equation}
Since there are only finitely many hidden coordinates, this constructs finite
biases $b_1,\ldots,b_{L-1}$ that realize all limiting hidden activations on
the training data.

Let $y_i$ be the $i$th column of the limiting output matrix $Y$. Since
\begin{equation}
    y_{i,n}=W_{L,n}h_{i,n}^{L-1}+b_{L,n},
\end{equation}
passing to differences gives
\begin{equation}
    y_i-y_1
    =
    W_L(h_i^{L-1}-h_1^{L-1}).
\end{equation}
Thus, with
\begin{equation}
    b_L\coloneqq y_1-W_Lh_1^{L-1},
\end{equation}
we have
\begin{equation}
    y_i=W_Lh_i^{L-1}+b_L,
    \quad i=1,\ldots,N.
\end{equation}
Hence there exists $\theta\in\Theta_L(K)$ whose training-data outputs are the
limits of those of $\theta_n$ and whose weights are the limits of the weights
of $\theta_n$. By lower semicontinuity of $G$ and continuity of the Frobenius
norm,
\begin{equation}
    G\big(f_\theta(x_1),\ldots,f_\theta(x_N)\big)+\lambda C_L(\theta) \leq
    \liminf_{n\to\infty}
    \left[
        G\big(f_{\theta_n}(x_1),\ldots,f_{\theta_n}(x_N)\big)
        +
        \lambda C_L(\theta_n)
    \right].
\end{equation}
Therefore the minimum over $\Theta_L(K)$ is attained for every $K\geq N^2$.

We now prove the width reduction. Let $\theta\in\ThetaL$ be arbitrary. Fix a
hidden layer $\ell\in\{1,\ldots,L-1\}$ of width $M$. On the training data,
write
\begin{equation}
    h_i^0=x_i
    \quad\text{and}\quad
    h_i^\ell=\sigma(W_\ell h_i^{\ell-1}+b_\ell).
\end{equation}
Let $w_k^\T$ be the $k$th row of $W_\ell$, let $\beta_k$ be the $k$th entry of
$b_\ell$, and let $v_k$ be the $k$th column of $W_{\ell+1}$. Neurons with
$v_k=0$ may be deleted. Neurons with $w_k=0$ are constant on the training data
and may be absorbed into the next bias. Thus it suffices to consider active
neurons with $w_k\neq0$.

Set
\begin{equation}
    \rho_k\coloneqq\norm{w_k}_2,
    \quad
    \widetilde w_k\coloneqq\frac{w_k}{\rho_k},
    \quad
    \widetilde\beta_k\coloneqq\frac{\beta_k}{\rho_k},
    \quad\text{and}\quad
    \widetilde v_k\coloneqq\rho_k v_k.
\end{equation}
By positive homogeneity,
\begin{equation}
    v_k\sigma(w_k^\T h_i^{\ell-1}+\beta_k)
    =
    \widetilde v_k
    \sigma(\widetilde w_k^\T h_i^{\ell-1}+\widetilde\beta_k).
\end{equation}
Define
\begin{equation}
    \Phi_{k,i}
    \coloneqq
    \sigma(\widetilde w_k^\T h_i^{\ell-1}+\widetilde\beta_k),
    \quad
    V\coloneqq
    \begin{bmatrix}
        \widetilde v_1 & \cdots & \widetilde v_M
    \end{bmatrix},
\end{equation}
and
\begin{equation}
    \Psi
    \coloneqq
    \begin{bmatrix}
        W_{\ell+1}h_1^\ell & \cdots & W_{\ell+1}h_N^\ell
    \end{bmatrix}.
\end{equation}
Then $\Psi=V\Phi$. Consider the multi-task lasso problem
\begin{equation}
    \min_{\widehat V=[\widehat v_1,\ldots,\widehat v_M]}
    \sum_{k=1}^M\norm{\widehat v_k}_2
    \quad
    \subj
    \quad
    \Psi=\widehat V\Phi.
    \label{eq:appendix-layer-mtl}
\end{equation}
By \cite[Theorem~11]{shenouda2024variation}, there exists a solution
$V^\star=[v_1^\star,\ldots,v_M^\star]$ with at most
$r_\Phi r_\Psi$ nonzero columns, where
\begin{equation}
    r_\Phi\coloneqq\rank(\Phi)
    \quad\text{and}\quad
    r_\Psi\coloneqq\rank(\Psi).
\end{equation}
Since $\Phi$ and $\Psi$ both have $N$ columns, $r_\Phi r_\Psi\leq N^2$.

For every nonzero column $v_k^\star$, define
\begin{equation}
    w_k^\star
    \coloneqq
    \norm{v_k^\star}_2^{1/2}\widetilde w_k,
    \quad
    \beta_k^\star
    \coloneqq
    \norm{v_k^\star}_2^{1/2}\widetilde\beta_k,
    \quad\text{and}\quad
    u_k^\star
    \coloneqq
    \norm{v_k^\star}_2^{-1/2}v_k^\star.
\end{equation}
Delete the zero columns. Since $V^\star\Phi=\Psi$, this replacement preserves
$W_{\ell+1}h_i^\ell$ for every $i=1,\ldots,N$, and therefore preserves all
subsequent activations and all network outputs on the training data. Moreover,
the new contribution of this layer pair to the squared Frobenius cost is
\begin{equation}
    2\sum_{k=1}^M\norm{v_k^\star}_2,
\end{equation}
whereas the original contribution is at least
\begin{equation}
    2\sum_{k=1}^M\norm{\widetilde v_k}_2.
\end{equation}
Since $V^\star$ minimizes \cref{eq:appendix-layer-mtl} and $V$ is feasible,
the replacement does not increase the parameter cost.

Applying this compression successively to each hidden layer produces
$\widetilde\theta\in\Theta_L(N^2)$ such that
\begin{equation}
    f_{\widetilde\theta}(x_i)=f_\theta(x_i),
    \quad i=1,\ldots,N,
\end{equation}
and
\begin{equation}
    C_L(\widetilde\theta)\leq C_L(\theta).
\end{equation}
Consequently,
\begin{equation}
    \inf_{\theta\in\ThetaL}
    G\big(f_\theta(x_1),\ldots,f_\theta(x_N)\big)+\lambda C_L(\theta)
    =
    \inf_{\theta\in\Theta_L(N^2)}
    G\big(f_\theta(x_1),\ldots,f_\theta(x_N)\big)+\lambda C_L(\theta).
    \label{eq:appendix-inf-equals-N2}
\end{equation}

By the fixed-width existence result, the minimum over $\Theta_L(N^2)$ is
attained. By \cref{eq:appendix-inf-equals-N2}, any minimizer over
$\Theta_L(N^2)$ is an unrestricted minimizer. Finally, if $K\geq N^2$, then
$\Theta_L(N^2)\subset\Theta_L(K)\subset\ThetaL$ by zero-padding, and the same
compression argument gives equality of the infima over $\Theta_L(K)$ and
$\ThetaL$. The fixed-width existence result gives a minimizer over
$\Theta_L(K)$, and equality of infima implies that any such minimizer is an
unrestricted minimizer. This completes the proof.
\end{proof}

\bibliographystyle{plain} 
\bibliography{ref}

\end{document}